\documentclass[12pt,a4paper]{amsart}
\usepackage{amsfonts}
\usepackage{amsmath}
\usepackage{amssymb}
\usepackage{faktor}
\usepackage{amscd}
\usepackage{array}
\usepackage{amsthm}
\usepackage{subfigure}
\usepackage{mathdots}
\usepackage{mathrsfs}
\usepackage{algorithm}
\usepackage[noend]{algpseudocode}
\usepackage{graphicx}
\usepackage{color}
\usepackage{pgf}
\usepackage{scrpage2}
\usepackage{multirow}
\usepackage{listings}
\usepackage{tikz}
\usepackage{graphicx}
\usepackage[left=2.5cm, right=2.5cm,top=2.5cm,bottom=2.5cm]{geometry}

\newcommand{\N}{\ensuremath{\mathbb{N}}}

\renewcommand{\S}{\ensuremath{\mathbb{S}}}

\newcommand{\R}{\ensuremath{\mathbb{R}}}
\newcommand{\C}{\ensuremath{\mathbb{C}}}

\newcommand{\X}{\mathbb{X}}

\newcommand{\supp}{\textnormal{supp}}

\def\3{\ss}
\def\ep{\varepsilon}
\def\la{\lambda}

\def\Ga{\Gamma}

\newmuskip\pFqmuskip

\newcommand*\pFq[6][8]{
  \begingroup % only local assignments
  \pFqmuskip=#1mu\relax
  \begingroup\lccode`\~=`\,
  \lowercase{\endgroup\let~}\pFqcomma
  {}_{#2}F_{#3}{\left(\genfrac..{0pt}{}{#4}{#5};#6\right)}%
  \endgroup
}
\newcommand*\pRegFq[6][8]{
  \begingroup % only local assignments
  \pFqmuskip=#1mu\relax
  \begingroup\lccode`\~=`\,
  \lowercase{\endgroup\let~}\pFqcomma
  {}_{#2}\tilde{F}_{#3}{\left(\genfrac..{0pt}{}{#4}{#5};#6\right)}%
  \endgroup
}

\newcommand{\pFqcomma}{\mskip\pFqmuskip}

\DeclareMathOperator*{\Heavy}{H}

\DeclareMathOperator*{\diag}{diag}
\DeclareMathOperator*{\trace}{trace}

\DeclareMathOperator*{\spann}{span}
\DeclareMathOperator{\dist}{dist}
\DeclareMathOperator{\F}{F}

\DeclareMathOperator{\G}{\mathcal{G}}

\DeclareMathOperator*{\vol}{vol}

\DeclareMathOperator{\OOO}{O}
\DeclareMathOperator*{\SO}{SO}

\newtheorem{thm}{Theorem}[section]
\newtheorem{lemma}[thm]{Lemma}
\newtheorem{remark}[thm]{Remark}

\newtheorem{example}[thm]{Example}
\newtheorem{corollary}[thm]{Corollary}
\newtheorem{proposition}[thm]{Proposition}

\numberwithin{equation}{section}
\numberwithin{table}{section}
\numberwithin{figure}{section}

\long\def\symbolfootnote[#1]#2{\begingroup%
\def\thefootnote{\fnsymbol{footnote}}\footnote[#1]{#2}\endgroup}
\renewcommand{\mathbf}[1]{\ensuremath{\boldsymbol{#1}}}

\newcommand{\rank}{ \operatorname{rank}}

\renewcommand{\thefootnote}{\fnsymbol{footnote}}

\allowdisplaybreaks
\title[Discrepancy kernels]{Spectral decomposition of discrepancy kernels on the Euclidean ball, the special orthogonal group, and the Grassmannian manifold}
\date{\today}
\date{}

\author[J.~Dick]{Josef Dick}
\address[J.~Dick]{University of New South Wales, School of Mathematics and Statistics, Sydney, NSW, 2052, Australia}
\email{josef.dick@unsw.edu.au}

\author[M.~Ehler]{Martin Ehler}
\address[M.~Ehler]{University of Vienna,
Department of Mathematics,
Oskar-Morgenstern-Platz 1, 
A-1090 Vienna
}
\email{martin.ehler@univie.ac.at}

\author[M.~Gr\"af]{Manuel Gr\"af}
\address[M.~Gr\"af]{Austrian Academy of Sciences, Acoustics Research Institute, Vienna, Austria
}
\email{mgraef@kfs.oeaw.ac.at}

\author[C.~Krattenthaler]{Christian Krattenthaler}
\address[C.~Krattenthaler]{University of Vienna,
Department of Mathematics,
Oskar-Mor\-gen\-stern-Platz 1, 
A-1090 Vienna
}
\email{christian.krattenthaler@univie.ac.at}

\begin{document}

\begin{abstract}
To numerically approximate Borel probability measures by finite atomic measures, we study the spectral decomposition of discrepancy kernels when restricted to compact subsets of $\R^d$. For restrictions to the Euclidean ball in odd dimensions, to the rotation group $\SO(3)$, and to the Grassmannian manifold $\G_{2,4}$, we compute the kernels' Fourier coefficients and determine their asymptotics. The $L_2$-discrepancy is then expressed in the Fourier domain that enables efficient numerical minimization based on the nonequispaced fast Fourier transform. For $\SO(3)$, the nonequispaced fast Fourier transform is publicly available, and, for $\G_{2,4}$, the transform is derived here. We also provide numerical experiments for $\SO(3)$ and $\G_{2,4}$. 
\end{abstract}

\keywords{Discrepancy kernels, spectral decompositions, 
Euclidean ball, special orthogonal group, Grassmannian manifold,
nonequispaced fast Fourier transform%
%submit to Mathematics of Computation or Foundations of Computational Mathematics?
}

\subjclass[2010]{Primary 42C10; Secondary
42B10, 43A65, 43A85, 65T40% 
% Trigonometric approximation and interpolation
%%42C15, %General harmonic expansions, frames
%%30E05, %Functions of a complex variable, Moment problems, interpolation problems
%%65F30% other matrix algorithms
}
%\noindent\textit{2010 AMS Mathematics Subject Classification} : \text{
%65T40, % Trigonometric approximation and interpolation
%42C15, %General harmonic expansions, frames
%30E05, %Functions of a complex variable, Moment problems, interpolation problems
%65F30 % other matrix algorithms
%}
%\end{abstract}
\maketitle

\section{Introduction}\label{sect:Einleitung}
Consider a Borel probability measure $\mu:\mathscr{B}(\R^d)\rightarrow[0,1]$ on $\R^d$, where $\mathscr{B}(\R^d)$ denotes the Borel sigma algebra on $\R^d$. For fixed $n\in\N$, we aim to allocate a suitable $n$-point set $\{x_1,\ldots,x_n\}\subset\R^d$ such that the normalized atomic measure 
\begingroup
\setlength{\abovedisplayskip}{2pt}
\setlength{\belowdisplayskip}{2pt}
\begin{equation}\label{eq:nu def}
\nu_n:=\frac{1}{n}\sum_{j=1}^n\delta_{x_j}
\end{equation}
\endgroup
approximates $\mu$. Here, $\delta_{x_j}:\mathscr{B}(\R^d)\rightarrow\{0,1\}$ denotes the point measure localized at $x_j$. To quantify the $L_2$-discrepancy between $\mu$ and $\nu_n$, select a measure $\beta$ on $\mathscr{B}(\R^d)$ with $\mu,\delta_x\in L_2(\mathscr{B}(\R^d),\beta)$, for all $x\in\R^d$, and consider
\begin{equation}\label{eq:fund def}
\mathscr{D}_{\beta}(\mu,\nu_n):=\|\mu-\nu_n\|^2_{L_2(\mathscr{B}(\R^d),\beta)}=\int_{\mathscr{B}(\R^d)} \left| \mu(B)- \nu_n(B)\right|^2{\rm d} \beta(B),
\end{equation}
cf.~\cite{Kuipers:1974la,Matousek:2010kb,Nowak:2010rr}, see Section~\ref{sec:intro 2} for explicit examples. For fixed $n\in\N$, we aim to minimize $\mathscr{D}_{\beta}(\mu,\nu_n)$ among all $n$-point sets $\{x_1,\ldots,x_n\}\subset\R^d$. The present manuscript is concerned with discretizations of \eqref{eq:fund def} that facilitate numerical minimization.

The associated \emph{discrepancy kernel} $K_\beta:\R^d\times\R^d\rightarrow\R$ is defined by 
\begin{equation}\label{eq:K disc2}
K_\beta(x,y):= \langle \delta_x,\delta_y\rangle_{L_2(\mathscr{B}(\R^d),\beta)}=\int_{\mathscr{B}(\R^d)} \delta_x(B)\delta_y(B) {\rm d} \beta(B),
\end{equation}
and we assume it is continuous. Fubini's Theorem and $\mu(B)=\int_{\R^d} \delta_x(B){\rm d}\mu(x)$
applied to 
\begin{equation*}
\|\mu-\nu_n\|^2_{L_2(\mathscr{B}(\R^d),\beta)}=\|\mu\|^2_{L_2(\mathscr{B}(\R^d),\beta)}-2\langle \mu,\nu_n\rangle_{L_2(\mathscr{B}(\R^d),\beta)}+\|\nu_n\|^2_{L_2(\mathscr{B}(\R^d),\beta)}
\end{equation*}
yields that \eqref{eq:fund def} is identical to
\begin{equation}\label{eq:D neu}
\mathscr{D}_{\beta}(\mu,\nu_n)=\iint\limits_{\R^d\times\R^d} K_{\beta}(x,y){\rm d}\mu(x) {\rm d}\mu(y) -2\sum_{j=1}^n\int_{\R^d} \frac{K_{\beta}(x,x_j)}{n}{\rm d}\mu(x)
+\sum_{i,j=1}^n \frac{K_{\beta}(x_i,x_j)}{n^2} .
\end{equation}

If a compact set $\X\subset\R^d$ is known in advance such that $\supp(\mu)\subset\X$, then we shall restrict the minimization to $\{x_1,\ldots,x_n\}\subset \X$, so that only the restricted kernel $K_\beta|_{\X\times \X}$ matters. By endowing $\X$ with a finite Borel measure $\sigma_\X$ having full support, Mercer's Theorem yields an orthonormal basis $\{\phi_l\}_{l=0}^\infty$ for $L_2(\X,\sigma_\X)$ and coefficients $(a_l)_{l=0}^\infty$ such that the spectral decomposition 
\begin{equation}\label{eq:FT exp K}
K_\beta|_{\X\times \X}(x,y) = \sum_{l=0}^\infty a_l \phi_l(x)\overline{\phi_l(y)},\quad x,y\in\X,
\end{equation}
holds with absolute and uniform convergence. We call $(a_l)_{l=0}^\infty$ the Fourier coefficients of the kernel $K_\beta|_{\X\times \X}$. If $\supp(\mu),\supp(\nu_n)\subset\X$, then the Fourier expansion of the $L_2$-discrepancy \eqref{eq:D neu} is
\begin{equation}\label{eq:fund eq min}
%\mathscr{D}_{\beta}(\mu,\nu_n) = \sum_{l=0}^\infty a_l  \Big| \hat{\mu}_{l}-\frac{1}{n}\sum_{j=1}^n\overline{\phi_{l}(x_j)}  \Big|^2,
%\qquad \hat{\mu}_l:=\int_{\X}\overline{ \phi_l(x)}{\rm d}\mu(x),
\mathscr{D}_{\beta}(\mu,\nu_n) = \sum_{l=0}^\infty a_l  \left| \hat{\mu}_{l}-\hat{\nu}_{n,l}  \right|^2,
\qquad \hat{\mu}_l:=\int_{\X}\overline{ \phi_l(x)}{\rm d}\mu(x),\quad \hat{\nu}_{n,l}:=\frac{1}{n}\sum_{j=1}^n\overline{\phi_{l}(x_j)},
\end{equation}
where the Fourier coefficients $\hat{\mu}_l$ and $\hat{\nu}_{n,l}$ of the measures $\mu$ and $\nu_n$, respectively, are well-defined if $a_l\neq 0$.  
Truncation of the discretization \eqref{eq:fund eq min} enables the use of the nonequispaced fast Fourier transform, thereby offering more efficient minimization of $\mathscr{D}_\beta(\mu,\nu_n)$, cf.~\cite{Graf:2013zl,Graf:2011lp}. Thus, we aim to 
\begin{quotation}
A) compute $(a_l)_{l=0}^\infty$ and $(\phi_l)_{l=0}^\infty$ in the Fourier expansion \eqref{eq:FT exp K} of $K_\beta|_{\X\times \X}$.
\end{quotation}

%The $L_2$-discrepancy $\mathscr{D}_\beta(\mu,\nu_n)$ is closely related to the construction of quadrature points for integration of functions in the reproducing kernel Hilbert space $\mathscr{H}_\beta(\X)$ generated by $K_\beta|_{\X\times \X}$. Koksma-Hlawka type inequalities yield that $\mathscr{D}_\beta(\mu,\nu_n)$ is the smallest constant
%such that  
%\begin{equation}\label{eq:eq eq}
%\Big|\int_\X f(x){\rm d}\mu(x) - \frac{1}{n}\sum_{j=1}^n f(x_j)  \Big|^2 \leq \mathscr{D}_\beta(\mu,\nu_n) \|f\|_{\mathscr{H}_\beta(\X)}^2,\quad f\in \mathscr{H}_\beta(\X),
%\end{equation}
The $L_2$-discrepancy $\mathscr{D}_\beta(\mu,\nu_n)$ also coincides with the worst case integration error 
\begin{equation}\label{eq:eq eq}
\mathscr{D}_\beta(\mu,\nu_n) = \sup_{\|f\|_{\mathscr{H}_\beta(\X)}\leq 1 }\left|\int_\X f(x){\rm d}\mu(x) - \frac{1}{n}\sum_{j=1}^n f(x_j)  \right|^2 
\end{equation}
with respect to the reproducing kernel Hilbert space $\mathscr{H}_\beta(\X)$ generated by $K_\beta|_{\X\times \X}$, 
cf.~\cite{Brauchart:2013il,Brauchart:2013pt,Gnewuch:2012jy,Graf:2013zl}. 
%For independent random points drawn from $\mu$, the expected $L_2$-discrepancy satisfies $\mathbb{E}\mathscr{D}_\beta(\mu,\nu_n)\lesssim \frac{1}{n}$, in particular, this inequality holds for the sequence of minimizing points. 
To specify $\mathscr{H}_\beta(\X)$, we aim to 
\begin{quotation}
B) identify $\mathscr{H}_\beta(\X)$ with a classical function space.
\end{quotation}
Fourier decay properties generally quantify Sobolev smoothness. To accomplish B), we aim to determine the asymptotics of $K_\beta|_{\X\times \X}$'s Fourier coefficients $(a_l)_{l=0}^\infty$ in \eqref{eq:FT exp K}. 

For $\X=\S^{d-1}$ and a particular choice of $\beta$, the kernel $K_{\beta}|_{\S^{d-1}\times \S^{d-1}}$ essentially coincides with the Euclidean distance, see \cite{Brauchart:2013il,Brauchart:2013pt}. The Fourier expansion is determined in \cite{Baxter:2001pb}, and the 
decay of the Fourier coefficients yields that $K_{\beta}|_{\S^{d-1}\times \S^{d-1}}$ reproduces the Sobolev space $\mathscr{H}_\beta(\S^{d-1})=\mathbb{H}^{\frac{d}{2}}(\S^{d-1})$. For the sphere and the torus, the nonequispaced fast Fourier transform is available, and both A) and B) are discussed in \cite{Graf:2011lp,Graf:2013fk}.  

\smallskip
This manuscript is dedicated to derive analogous results for other compact sets $\X$. We focus on the unit ball, the special orthogonal group, and the Grassmannian manifold, 
\begin{align*}
\mathbb{B}^d&:=\{x\in\R^d:\|x\|\leq 1\},\\
\SO(d)&:=\{x\in\R^{d\times d} : \det(x)=1,\, x^{-1}=x^\top \},\\
\mathcal{G}_{k,d}&:=\{x\in\R^{d\times d} : x^\top=x,\, x^2=x,\, \trace(x)=k\}.
\end{align*}
%Guided by the results on the sphere, we derive a discrepancy kernel that essentially coincides with the Euclidean distance on $\X\times\X$. Our main interest is in 
We achieve goal A) for $\mathbb{X}=\mathbb{B}^d$ with odd $d$. Both goals,  A) and B), are achieved for $\SO(3)$ and $\G_{2,4}$. We also provide numerical experiments. For $\SO(3)$, the computations are based on the nonequispaced fast Fourier transform designed in \cite{Graf:2009ye,Potts:2009gb}. For $\G_{2,4}$, we derive the nonequispaced fast Fourier transform by parametrization via the double covering $\mathbb{S}^2\times \mathbb{S}^2$ and developing the respective transform there. We also accomplish B) for the general cases $\SO(d)$ and $\G_{k,d}$. 

%\bigskip
%The above derivations are still valid for atomic measures with weights $w_1,\ldots,w_n\geq 0$ by replacing $\frac{1}{n}\sum_{j=1}^n \ldots$ with $\sum_{j=1}^n w_j\ldots$ in \eqref{eq:nu def}, \eqref{eq:fund eq min}, and \eqref{eq:eq eq}. In fact, our outline offers further variations. One alternative is to replace the point measures $\delta_x$ in \eqref{eq:K disc2} with more general Borel measures $\alpha_x:\mathscr{B}(\R^d)\rightarrow [0,\infty)$ and define $K_\beta(x,y):= \langle \alpha_x,\alpha_y\rangle_{L_2(\mathscr{B}(\R^d),\beta)}$. One may also skip \eqref{eq:fund def} altogether, start with a continuous, positive definite kernel $K:\mathbb{X}\times\mathbb{X}\rightarrow\R$, and take \eqref{eq:fund eq min} as the definition of discrepancy with respect to this kernel. Furthermore, the measure $\nu_n$ is not necessarily required to be atomic. In particular, our theoretical results are also useful if $\nu_n$ is more general, for instance, supported on a curve in $\X$. This relates to recent results in \cite{Chauffert:2017eb,Gournay:2019kt}, where measures on a $2$-dimensional cube are approximated by a measure supported on a curve with further side constraints. In Section \ref{sec:8} of the present paper, we shall extend these approximations to compact metric spaces and provide numerical illustrations for the sphere. 
%

\section{Two introductory examples}\label{sec:intro 2}
%The present section is dedicated to provide few examples well-known in the literature. 
We first present a well-known elementary example on the interval $[0,s]$, for which both aims A) and B) are achieved. Second, to support our perspective on discrepancy, we prove that the so-called 
Askey function is a discrepancy kernel of the form \eqref{eq:K disc2}.
\subsection{The Brownian motion kernel on $[0,s]$}\label{sec:expl choice beta and D}
%Let $s>0$ and suppose that the two Borel probability measures $\mu$ and $\nu_n=\frac{1}{n}\sum_{j=1}^n\delta_{x_j}$ are supported on the interval $[0,s]$, where $\delta_{x_j}$ denotes the Dirac measure of $x_j\in[0,s]$.  
Let ${\rm d}r$ be the Lebesgue measure on $[0,\infty)$. The mapping
$ %\begin{equation*}
h:[0,\infty)\rightarrow \mathscr{B}(\R)$ defined by $r\mapsto [r,\infty)
$ %\end{equation*}
induces the pushforward measure $\beta:=h_*({\rm d}r)$ that induces the discrepancy 
\begin{equation*}
\mathscr{D}_\beta(\mu,\nu_n)= \int_{0}^\infty |\mu([r,\infty))-\nu_n([r,\infty))|^2 {\rm d}r. 
\end{equation*}
The associated discrepancy kernel $K_\beta:\R\times\R\rightarrow\R$ is\footnote[1]{For $r\in\R$, we use the notation $r_+=\begin{cases}
r,& r\geq 0,\\ 0,& 
\text{otherwise.} \end{cases}$} 
\begin{equation*}
K_\beta(x,y) = \int_0^\infty \delta_x([r,\infty))\delta_y([r,\infty)) {\rm d}r =\min(x,y)_+,
%\begin{cases}
% \min(x,y),& x,y\geq 0,\\
% 0,& \text{otherwise},
%\end{cases}
\end{equation*}
so that $\mathscr{D}_\beta(\delta_x,\delta_y)=|x-y|$ for $x,y\in[0,\infty)$. The restriction of the kernel $K_\beta$ to $[0,s]\times[0,s]$ has the Fourier expansion 
\begin{align*}
K_\beta(x,y) %&  = \frac{1}{2}(x+y-|x-y|)  \\
&=\sum_{\substack{m\in\N\\ m \text{ odd}}} \frac{4s^2}{m^2\pi^2}\cdot \frac{\sin(\frac{\pi}{2s} mx)}{\sqrt{\frac{s}{2}}}\cdot\frac{\sin(\frac{\pi}{2s} my)}{\sqrt{\frac{s}{2}}},\quad x,y\in[0,s],
\end{align*}
with respect to the Lebesgue measure $\sigma_{[0,s]}$ on $[0,s]$. The reproducing kernel Hilbert space is
\begin{equation*}
\mathscr{H}_{K_\beta}([0,s])=\{f:[0,s]\rightarrow \C \;:\; f \text{ is absolutely continuous, } f(0)=0,\; f'\in L_2([0,s])\},
\end{equation*}
where the inner product between $f$ and $g$ is given by $\langle f',g'\rangle_{L_2([0,s])}$, cf.~\cite{Alpay:2015mb,Dick:2014bj} and \cite[Section~9.5.5]{Nowak:2010rr}. Note that $K_\beta|_{[0,1]\times[0,1]}$ is often called the Brownian motion kernel and $\mathscr{H}_{K_\beta}([0,s])$ is continuously embedded into the Sobolev space $\mathbb{H}^1([0,s])$.

\subsection{Askey's function and its restrictions}\label{sec:a}
Many positive definite kernels in the literature are of the form \eqref{eq:K disc2} and, hence, are discrepancy kernels. For odd $d$, Askey's kernel function $(x,y)\mapsto (1-\|x-y\|)^{\frac{d+1}{2}}_+$ is positive definite, cf.~\cite{Gneiting:1999ay}. In the following, we shall check that it is of the form \eqref{eq:K disc2}.

Denote the Euclidean ball of radius ${s}$ centered at $z\in\R^d$ by 
\begin{equation*}
\mathbb{B}^d_{s}(z):=\{x\in\R^d:\|x-z\|\leq {s}\},
\end{equation*}
with the conventions $\mathbb{B}^d_{s}:=\mathbb{B}^d_{s}(0)$ and $\mathbb{B}^d:=\mathbb{B}^d_1$. Fix $r>0$ and consider the discrepancy
\begin{equation}
\mathscr{D}_{d,r}(\mu,\nu_n):=\frac{1}{\vol(\mathbb{B}^d_{\frac{r}{2}})}
%\frac{(2r)^d\Gamma(\frac{d}{2}+1)}{\pi^{d/2}} 
\int_{\R^d} \left| \mu(\mathbb{B}^d_{\frac{r}{2}}(z))- \nu_n(\mathbb{B}^d_{\frac{r}{2}}(z))\right|^2{\rm d} z,
\end{equation}
where $\vol(\mathbb{B}^d_{\frac{r}{2}})=\frac{\pi^{d/2}}{\Gamma(\frac{d}{2}+1)}(\frac{r}{2})^d$.  
The associated discrepancy kernel is 
\begin{equation}\label{eq:kern nur mit r}
K_{d,r}(x,y)=\frac{1}{\vol(\mathbb{B}^d_{\frac{r}{2}})}\int_{\R^d} \delta_x(\mathbb{B}^d_{\frac{r}{2}}(z))\delta_y(\mathbb{B}^d_{\frac{r}{2}}(z)) {\rm d} z. 
\end{equation}
In order to additionally integrate over $r$, recall the 
(generalized) hypergeometric functions
\begin{equation*}
\pFq{k}{l}{f_1,\ldots,f_k}{g_1,\ldots,g_l}{z}:=\sum_{n=0}^\infty \frac{(f_1)_n \cdots (f_k)_n}{(g_1)_n \cdots (g_l)_n}\frac{z^n}{n!},
\end{equation*}
where $f_1,\ldots,f_k,g_1,\ldots,g_l,z\in\R$ and $(f)_n:=f\cdot (f+1)\cdots(f+n-1)$ denotes the Pochhammer symbol with $(f)_0:=1$. We consider $G_d:[0,\infty)\rightarrow\R$ given by
\begin{equation*}
G_d(r) = \begin{cases}
\pFq{2}{1}{-\frac{d+1}{4},-\frac{d-1}{4}}{-\frac{d}{2}}{r^2},& 0\leq r\leq 1,\\
0,& 1<r.
\end{cases}
\end{equation*}
Since $d$ is odd, either $\frac{d+1}{4}$ or $\frac{d-1}{4}$ is a natural number, so that the series terminates and 
$G_d$ is a polynomial in $r^{2}$ on $[0,1]$. 
%, and we can compute 
%\begin{equation*}
%G_d(r)=1+\sum_{l\geq 1} \frac{1}{(-8)^l l!} \frac{(d+1)(d-1)\cdots (d+3-4l)}{d(d-2)\cdots(d-2(l-1))}r^{2l},\quad 0\leq r\leq 1.
%\end{equation*}
%See Figure \ref{fig:Gd} for $G_3$. 
%\begin{figure}
%\includegraphics[width=.3\textwidth]{Gd_fig.pdf}
%\caption{$G_3$ on $[0,1]$}\label{fig:Gd}
%\end{figure}
By integration with respect to $G_d$, we obtain the $L_2$-discrepancy and the associated discrepancy kernel 
\begin{equation*}%\label{eq:disc def 1}
\mathscr{D}_{d}(\mu,\nu_n):=
%\!\!\int_0^\infty \!\frac{2^d\Gamma(\frac{d}{2}+1)}{r^{d}\pi^{d/2}}  \int_{\R^d} \!\Big| \mu(\mathbb{B}^d_{\frac{1}{2r}}(z))- \nu_n(\mathbb{B}^d_{\frac{1}{2r}}(z))\Big|^2{\rm d} z\; {\rm d} G_d(r) = \!\!
\int_0^\infty  \!\!\mathscr{D}_{d,r}(\mu,\nu_n){\rm d} G_d(r),\quad \text{ and }\quad K_{d}(x,y)= \int_0^\infty \!K_{d,r}(x,y) {\rm d} G_d(r),
\end{equation*}
respectively. It turns out that $K_d$ coincides with Askey's function.
%Both, $\mathscr{D}_{d}$ and $K_{d}$, are superpositions of $\mathscr{D}_{d,r}$ and $K_{d,r}$, for $r\in[0,\infty)$, respectively, that are weighted by $G_d(r)$. 
\begin{thm}\label{thm:Askey}
Let $d$ be odd. The discrepancy kernel $K_d$ satisfies %coincides with Askey's kernel function, i.e., 
\begin{equation}\label{eq:Kud}
K_d(x,y)=(1-\|x-y\|)^{\frac{d+1}{2}}_+,\quad x,y\in\R^d.
\end{equation}
\end{thm}
The proof is presented 
in Appendix~\ref{sec:Askey}. Provided that $d\geq 3$, Askey's kernel function  reproduces the Sobolev space $\mathbb{H}^{\frac{d+1}{2}}(\R^d)$ with an equivalent norm, see \cite{Wendland:2004wd}. 
%\begin{remark}\label{rem:0098}
%Results on kernel restrictions from \cite{Fuselier:2012jt} imply that $K_{d}|_{\S^{d-1}\times \S^{d-1}}$ reproduces $\mathbb{H}^{\frac{d}{2}}(\S^{d-1})$ with equivalent norms provided that $d\geq 3$. Analogously, for $d\geq 2$, $K_{{d^2}}|_{\SO(d)\times \SO(d)}$ and $K_{{d^2}}|_{\G_{k,d}\times \G_{k,d}}$ reproduce  $\mathbb{H}^{\frac{d(d-1)+2}{4}}(\SO(d))$ and $\mathbb{H}^{\frac{k(d-k)+1}{2}}(\G_{k,d})$ with equivalent norms, respectively. 
%\end{remark}

%\section{Discrepancy kernels on the sphere}
\section{The distance kernel on $\S^{d-1}$}\label{sec:sphere}
This section is dedicated to recall results on discrepancy kernels on the sphere $\S^{d-1}\subset\R^d$, for $d\geq 2$, from \cite{Brauchart:2013il,Brauchart:2013pt,Graf:2013zl,Sk20} that shall guide our subsequent investigations.

%Suppose that the two Borel probability measures $\mu$ and $\nu_n=\frac{1}{n}\sum_{j=1}^n\delta_{x_j}$ are supported on the sphere $\S^{d-1}$, where $\delta_{x_j}$ denotes the Dirac measure of $x_j\in\S^{d-1}$.  
Denote the geodesic ball of radius $r$ centered at $z\in\S^{d-1}$ by 
\begin{equation*}
B^{\S^{d-1}}_r(z):= \{x\in\S^{d-1} : \dist_{\S^{d-1}}(x,z)\leq r\},
\end{equation*}
where $\dist_{\S^{d-1}}(x,z)=\arccos(\langle x,z\rangle)$ is the geodesic distance on $\S^{d-1}$. We define 
\begin{equation*}
h:[0,\pi]\times\S^{d-1}\rightarrow \mathscr{B}(\R^d),\qquad (r,z)\mapsto B^{\S^{d-1}}_r(z)
\end{equation*}
and endow $[0,\pi]$ with the weighted Lebesgue measure $\sin(r){\mathrm d}r$, whereas $\S^{d-1}$ carries the normalized, orthogonal invariant surface measure $\sigma_{\S^{d-1}}$. The push-forward $\beta_d:=h_*(\sin(r){\mathrm d}r \otimes \sigma_{\S^{d-1}})$ is a measure on $\mathscr{B}(\R^d)$, so that the associated $L_2$-discrepancy is
\begin{equation*}
\mathscr{D}_{\beta_d}(\mu,\nu_n)= \int_{0}^\pi\int_{\S^{d-1}} |\mu(B^{\S^{d-1}}_r(z))-\nu_n(B^{\S^{d-1}}_r(z))|^2 {\rm d}\sigma_{\S^{d-1}}(z)\sin(r){\rm d}r. 
\end{equation*}
The associated discrepancy kernel is 
\begin{equation}\label{eq:def K new}
K_{\beta_d}(x,y)=\int_{0}^\pi\int_{\S^{d-1}}\delta_x(B^{\S^{d-1}}_r(z))\delta_y(B^{\S^{d-1}}_r(z))   {\rm d}\sigma_{\S^{d-1}}(z)\sin(r){\rm d}r,\quad x,y\in \R^{d}.
\end{equation}
According to \cite{Brauchart:2013il,Brauchart:2013pt,Graf:2013zl}, see also \cite{Alexander:1975rp}, $K_{\beta_d}$ satisfies
\begin{equation}\label{eq:K1 formula}
K_{\beta_d}(x,y)=1-\frac{\Gamma(\frac{d}{2})}{2\sqrt{\pi}\Gamma(\frac{d+1}{2})}\|x-y\|,\qquad x,y\in\S^{d-1}.
\end{equation}
If either $x$ or $y$ is not contained in $\S^{d-1}$, then $K_{\beta_d}(x,y)=0$. 

Choose $\sigma_\X:=\sigma_{\S^{d-1}}$ for the decomposition \eqref{eq:FT exp K} and let $\{Y^m_{{l}} : l=1,\ldots,Z(d,m)\}\subset L_2(\S^{d-1},\sigma_{\S^{d-1}})$ denote the set of orthonormal spherical harmonics of degree $m$ on $\S^{d-1}$, where 
$%\begin{equation*}
Z(d,m)=\frac {2m+d-2}{d-2}\binom {m+d-3}m.
$ %\end{equation*}
For $\tau>(d-1)/2$, the Sobolev space $\mathbb{H}^\tau(\S^{d-1})$ is the reproducing kernel Hilbert space associated with the reproducing kernel
\begin{equation}\label{eq:native kernel sphere}
(x,y)\mapsto \sum_{m=0}^\infty (1+m(m+d-2))^{-\tau}\sum_{l=1}^{Z(d,m)}Y^m_{l}(x)\overline{Y^m_{l}(y)},\quad x,y\in\S^{d-1}.
\end{equation}
The coefficients in the Fourier expansion 
\begin{equation*}
1-\frac{\Gamma(\frac{d}{2})}{2\sqrt{\pi}\Gamma(\frac{d+1}{2})}\|x-y\|= \sum_{m=0}^\infty c_m \sum_{l=1}^{Z(d,m)}Y^m_{l}(x)\overline{Y^m_{l}(y)},\quad x,y\in\S^{d-1},
\end{equation*}
satisfy $|c_m|\sim m^{-d}$, cf.~\cite{Brauchart:2013il}. This is the same asymptotics as the coefficients in \eqref{eq:native kernel sphere} for $s=d/2$. Therefore, $K_{\beta_d}|_{\S^{d-1}\times\S^{d-1}}$ reproduces the Sobolev space $\mathscr{H}_{\beta_d}(\S^{d-1})=\mathbb{H}^{\frac{d}{2}}(\S^{d-1})$ with an equivalent 
norm\footnote[2]{The kernel $K_{\beta_d}|_{\S^{d-1}\times\S^{d-1}}$ generates an inner product $\langle\cdot,\cdot\rangle_{K_{\beta_d}|_{\S^{d-1}\times\S^{d-1}}}$ in $\mathbb{H}^{\frac{d}{2}}(\S^{d-1})$, for which it is reproducing and the induced norm is equivalent to the standard norm in $\mathbb{H}^{\frac{d}{2}}(\S^{d-1})$, which is induced by the standard kernel \eqref{eq:native kernel sphere}.}. 

In order to determine the Fourier coefficients of kernels on the sphere that are polynomial in $\|x-y\|$, such as $K_{\beta_d}|_{\S^{d-1}\times\S^{d-1}}$, we require the Fourier coefficients of the monomial terms
$ 
\|x-y\|^p$. For any $p\in\N$, the Fourier expansion 
\begin{equation}\label{eq:expans sphere uni}
2^{-\frac{p}{2}}\|x-y\|^p = \sum_{m=0}^\infty a_m(p,\S^{d-1}) \sum_{l=1}^{Z(d,m)}Y^m_{l}(x)\overline{Y^m_{l}(y)},\quad x,y\in\S^{d-1},
\end{equation}
holds with coefficients 
determined by 
\begin{equation}\label{eq:ft coeff}
a_m(p,\S^{d-1}):=\iint\limits_{\S^{d-1}\times \S^{d-1}} 2^{-\frac{p}{2}}\|x-y\|^p \overline{Y^m_{l}(x)}Y^m_{l}(y){\rm d}\sigma_{\S^{d-1}}(x){\rm d}\sigma_{\S^{d-1}}(y).
\end{equation}
Note that \eqref{eq:ft coeff} is well-defined for the entire range $p>-(d-1)$ and $p$ is not required to be an integer. For $p>0$, the following proposition is essentially due to \cite{Baxter:2001pb}, see also \cite{Brauchart:2013il,Brauchart:fk}. Simple continuation arguments cover the full range of $p$, and the asymptotics  $\frac{\Gamma(-\frac{p}{2}+m)}{\Gamma(\frac{p}{2}+d-1+m)}=m^{-(p+d-1)}(1+o(1))$ are standard.
\begin{proposition}[\cite{Baxter:2001pb}]\label{prop:sphere unt}
Suppose $d\geq 2$. For any $p >-(d-1)$, we have
\begin{align}
  a_m(p,\S^{d-1})  
  & =\frac{2^{d} \Gamma(\frac{d}{2})}{4\sqrt{\pi}}\cdot\frac{2^{p/2} \Gamma(\frac{d}{2}+\frac{p}{2}-\frac{1}{2})}{\Gamma(-\frac{p}{2})}\cdot \frac{\Gamma(-\frac{p}{2}+m)}{\Gamma(\frac{p}{2}+d-1+m)}.\label{eq:una}
\end{align}
In particular, if $p\not\in 2\N$, then 
\begin{align}\label{eq:as sphe}
|a_m(p,\S^{d-1})|&=\left|\frac{2^{d} \Gamma(\frac{d}{2})}{4\sqrt{\pi}}\cdot\frac{2^{p/2} \Gamma(\frac{d}{2}+\frac{p}{2}-\frac{1}{2})}{\Gamma(-\frac{p}{2})}\right|m^{-(p+d-1)}(1+o(1)),
\end{align}
and the series \eqref{eq:expans sphere uni} terminates if $p\in 2\N$. 
\end{proposition}
For $p\in2\N$, the term $\Gamma(-\frac{p}{2})$ is not well-defined and \eqref{eq:una} is to be understood with the convention $\frac{\Gamma(-\frac{p}{2}+m)}{\Gamma(-\frac{p}{2})}=(-\frac{p}{2})_m$. Hence, we observe $a_m(p,\S^{d-1})  = 0$ for all $m > p/2$ if  $p\in 2\N$. 

It is noteworthy that the kernel $K_{d,r}$ in \eqref{eq:kern nur mit r} for $d=3$ is a discrepancy kernel that does not generate a Sobolev space on $\R^d$ but its restriction does. The proof of the following proposition is presented 
in Appendix~\ref{sec:A3}.
\begin{proposition}\label{prop:010}
Let $r\geq 1$. The reproducing kernel Hilbert space of $K_{3,r}$, given by \eqref{eq:kern nur mit r} with $d=3$, is continuously embedded into $\mathbb{H}^{2}(\R^3)$, but the reverse embedding does not hold. In contrast, $K_{3,r}|_{\S^{2}\times\S^{2}}$ reproduces $\mathbb{H}^{\frac{3}{2}}(\S^{2})$ with an equivalent norm.
\end{proposition}

%\subsection{Numerical examples on $\S^2$}
To provide numerical examples for $d=3$, Proposition~\ref{prop:sphere unt} provides the coefficients $(a_m)_{m=0}^\infty$ in the kernel expansion of $K_{\beta_{3}}$,
\begin{equation*}
1-\frac{1}{4}\|x-y\|= \sum_{m=0}^\infty a_m \sum_{l=1}^{2m+1}Y^m_{l}(x)\overline{Y^m_{l}(y)},\quad x,y\in\S^2.
\end{equation*}
For $\supp(\mu),\supp(\nu_n)\subset \S^2$, the $L_2$-discrepancy \eqref{eq:fund eq min} for $K_{\beta_{3}}$ with $\mathbb{X}=\S^2$  becomes
\begin{equation}\label{eq:series truncatinga}
\mathscr{D}_{\beta_{3}}(\mu,\nu_n) = \sum_{m=0}^\infty a_m \sum_{l=1}^{2m+1}  \left|\hat{\mu}^m_{l}- \frac{1}{n}\sum_{j=1}^n\overline{Y^m_{l}(x_j)}\right|^2,
\end{equation}
where $\hat{\mu}^m_{l}$ denotes the Fourier coefficient of $\mu$ with respect to $Y^m_{l}$, cf.~\eqref{eq:fund eq min}.  By truncating this series, the nonequispaced fast Fourier transform on $\S^2$, cf.~\cite{Graf:2011lp,Keiner:2009cv,Kunis:2003pt}, enables efficient minimization of
\begin{equation}\label{eq:min sphere}
\sum_{m=0}^M a_m \sum_{l=1}^{2m+1}   \left|\hat{\mu}^m_{l}- \frac{1}{n}\sum_{j=1}^n\overline{Y^m_{l}(x_j)}\right|^2
\end{equation}
among all $n$-point sets $\{x_1,\ldots,x_n\}\subset\S^2$ for fixed $n$. We are most interested in $n\gg M$. 
See Figure~\ref{fig:S2}
for a numerical experiment with $M=8$ and $n=50$. 
\begin{figure}
\includegraphics[width=.4\textwidth]{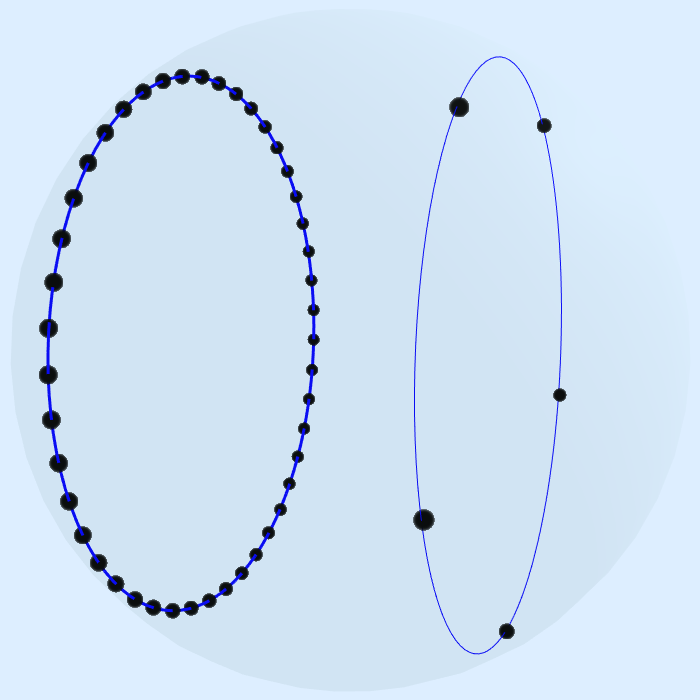}
\caption{The target measure $\mu$ is supported on two circles on the sphere $\S^2$ with weight ratio $9/1$. Numerical minimization of \eqref{eq:min sphere} splits $50$ points into $45$ points equally distributed on one and $5$ points on the other circle.}\label{fig:S2}
\end{figure}

\section{Discrepancy kernels on compact sets}\label{sec:b}
Here we discuss discrepancy kernels that extend the kernels of the previous section in a natural way. 
For $d\geq 1$, let us define the half-space 
\begin{equation*}
\Omega^d_r(z):=\{x\in \R^d : \langle z,x\rangle \geq r \}\in\mathscr{B}(\R^d),\quad z\in\S^{d-1},\quad r\in\R.
\end{equation*}
For fixed $s>0$, we consider the mapping
$ %\begin{equation*}
h:[-s,s]\times\S^{d-1}\rightarrow \mathscr{B}(\R^d)$ defined by $(r,z)\mapsto \Omega^d_r(z)
$ %\end{equation*}
and endow $[-s,s]$ with the Lebesgue measure ${\mathrm d}r$. The push-forward measure $\beta_{d,s}:=h_*({\mathrm d}r\otimes \sigma_{\S^{d-1}})$ leads to the associated $L_2$-discrepancy
\begin{equation*}%\label{eq:discrepancy for special beta half spaces}
\mathscr{D}_{\beta_{d,s}}(\mu,\nu_n)=\int_{-s}^{s} \int_{\S^{d-1}} \left| \mu(\Omega^d_r(z)) -  \nu_n(\Omega^d_r(z))\right|^2{\rm d} \sigma_{\S^{d-1}}(z) {\rm d} r.
\end{equation*}
The associated \emph{discrepancy kernel} is
\begin{equation}\label{eq:K disc}
K_{\beta_{d,s}}(x,y)=\int_{-s}^{s} \int_{\S^{d-1}} \delta_x(\Omega^d_r(z))\delta_y(\Omega^d_r(z)) {\rm d} \sigma_{\S^{d-1}}(z) {\rm d} r,\qquad x,y\in\R^d.
\end{equation}
Since $B^{\S^{d-1}}_r(z) = \Omega^d_{\cos(r)}(z)\cap \S^{d-1}$, for $r\in[0,\pi]$ and $z\in\S^{d-1}$, 
we deduce 
\begin{equation*}
K_{\beta_d}|_{\S^{d-1}\times\S^{d-1}}=K_{\beta_{d,1}}|_{\S^{d-1}\times\S^{d-1}},\qquad d\geq 2,
\end{equation*} 
with $K_{\beta_d}$ as in \eqref{eq:def K new}. 
In contrast to $K_{\beta_d}$, the kernel $K_{\beta_{d,s}}$ is not identically zero outside of $\S^{d-1}\times\S^{d-1}$ and makes also sense for $d=1$. 
\begin{example}\label{ex:d=1}
For $d=1$, we have $\S^0=\{\pm 1\}$, so that the half-spaces are $\Omega^1_r(1)=[r,\infty)$ and $\Omega^1_r(-1)=(-\infty,-r]$. 
Direct calculation of \eqref{eq:K disc} yields 
\begin{equation*}
K_{\beta_{1,s}}(x,y)= \begin{cases}
s-\frac{1}{2}|x-y|,& |x|,|y|\leq s,\\
\frac{s}{2}+\frac{xy}{2|y|},& |x|\leq s\leq |y|,\\
s\Heavy(xy), & s\leq |x|,|y|,
\end{cases}
\end{equation*}
%with $K_{\beta_{1,s}}(y,x)=K_{\beta_{1,s}}(x,y)$, for $x,y\in\R$, and 
where $\Heavy$ is the 
Heaviside step function. % see also Figure \ref{fig:example}. 
%\begin{figure}
%%\includegraphics[width=.4\textwidth]{exampleKernel_2.png}
%\includegraphics[width=.5\textwidth]{exampleKernel_3.png}
%\caption{$K_{\beta_{1,1}}(x,y)$ for $x,y\in [-2,2]$.}\label{fig:example}
%\end{figure}
\end{example}
%We now provide the Fourier expansion of the restriction $K_{\beta_{1,s}}|_{[-s,s]\times[-s,s]}$:
\begin{proposition}\label{prop:interval}
The Fourier expansion of the kernel $K_{\beta_{1,s}}|_{[-s,s]\times[-s,s]}$ with respect to the Lebesgue measure $\sigma_{[-s,s]}$ on $[-s,s]$ is 
\begin{multline*}
K_{\beta_{1,s}}(x,y)   = \sum_{\substack{m\in\N\\ m \text{ odd}}} \frac{4}{m^2\pi^2}\cdot\frac{1}{s} \cdot\sin(\frac{\pi}{2s}mx)\sin(\frac{\pi}{2s}my) \\
 + \hspace{-.5cm}\sum_{\{u>0 \,:\,\tan(u)=\frac{1}{u}\} }  \frac{1}{u^2}\cdot\frac{1}{s(\sin(u)^2+1)}\cdot \cos(\frac{u}{s}x)\cos(\frac{u}{s}y),\quad x,y\in[-s,s].
\end{multline*}
Its reproducing kernel Hilbert space is
\begin{equation*}
\mathscr{H}_{K_{\beta_{1,s}}}([-s,s])=\{f:[-s,s]\rightarrow \C \;:\; f \text{ is absolutely continuous, } f'\in L_2([-s,s])\},
\end{equation*}
where the inner product between $f$ and $g$ is given by 
\begin{equation*}
\frac{1}{2s}\big(f(-s)+f(s)\big)\big(\overline{g(-s)+g(s)}   \big)+\langle f',g'\rangle_{L_2([-s,s])}.
\end{equation*}
\end{proposition}
Note that $\mathscr{H}_{K_{\beta_{1,s}}}([-s,s])$ is continuously embedded into $\mathbb{H}^1([-s,s])$. The proof of Proposition~\ref{prop:interval} is presented 
in Appendix~\ref{sec:A4}.
It uses that, 
%\begin{remark}\label{remark:Sturms}
up to a constant, $K_{\beta_{1,s}}|_{[-s,s]\times[-s,s]}$ is the Green's function of the $1$-dimensional harmonic equation $\Delta u = f$ on $[-s,s]$ with the boundary conditions $u'(s) = - u'(-s)$ and $u(s) + u(-s) = -2u'(s)$. 
%\end{remark}

It turns out that $K_{\beta_{d,s}}$ has a simple form on $\mathbb{B}^d_{s}\times \mathbb{B}^d_s$.
\begin{thm}\label{thm:compact set}
For $d\geq 1$, the discrepancy kernel $K_{\beta_{d,s}}$ satisfies
\begin{equation}\label{eq:dist is rkh}
K_{\beta_{d,s}}(x,y) = s-\frac{\Gamma(\frac{d}{2})}{2\sqrt{\pi}\Gamma(\frac{d+1}{2})}\|x-y\|,\qquad x,y\in \mathbb{B}^d_{s}.
\end{equation}
\end{thm}
The identity \eqref{eq:dist is rkh} for $x,y\in\S^{d-1}$ with $s=1$ has been established in \cite{Brauchart:2013pt}, see also \eqref{eq:K1 formula}. Essentially, 
the same proof still works for the more general situation. 
%\begin{remark}
%For $d=3$, one can compute $K_{\beta_{3,r_0}}(x,x) ={r_0}$, for all $x\in\R^3$, which goes along with \eqref{eq:dist is rkh}. However, direct computations yield
%\begin{equation*}
%K_{\beta_{3,r_0}}(x,0) = \frac{{r_0}}{2}+\frac{{r_0}^2}{4\|x\|}, \qquad x\in\R^3\setminus \mathbb{B}^3_{r_0}(0),
%\end{equation*}
%so that the explicit formula \eqref{eq:dist is rkh} does not hold for all $x,y\in\R^d$.
%\end{remark}
Theorem~\ref{thm:compact set} provides a simple form of $K_{\beta_{d,s}}|_{\X\times\X}$ with $\X= \mathbb{B}^d_{s}$, which may facilitate further computations. An immediate consequence is $\mathscr{D}_{\beta_{d,s}}(\delta_x,\delta_y)=\frac{\Gamma(\frac{d}{2})}{\sqrt{\pi}\Gamma(\frac{d+1}{2})}\|x-y\|$, for $x,y\in \mathbb{B}^d_{s}$. 

%the computation of the restricted kernel's Fourier expansion.
%For $d=1$, Proposition \ref{prop:interval} fully covers $K_{\beta_{d,s}}|_{\mathbb{B}^d_{s}\times\mathbb{B}^d_{s}}$. The spectral decomposition for odd $d$ is derived in the next section. 

%
%
%the spectral decomposition of $K_{\beta_{d,s}}|_{\mathbb{B}^d_{s}\times\mathbb{B}^d_{s}}$ complicated and is quite involved and postponed to Section \ref{sec:ball}. Here, we shall proceed with the smaller subsets $\SO(3)$ and $\G_{2,4}$.
%

\section{The Euclidean ball $\mathbb{B}^d$}\label{sec:ball}
This section is dedicated to derive the Fourier expansion of the discrepancy kernel $K_{\beta_{d,1}}$ in \eqref{eq:K disc} on $\mathbb{B}^d$. Proposition~\ref{prop:interval} has covered $d=1$, and we now derive the spectral decomposition of 
\begin{equation*}
K^{d,p}:\mathbb{B}^d\times \mathbb{B}^d \rightarrow \R,\qquad (x,y)\mapsto \|x-y\|^p,
\end{equation*}
for all odd $d$ with odd $p>1-d$ with respect to the Lebesgue measure $\sigma_{\mathbb{B}^d}$ on $\mathbb{B}^d$. The case $d=3$ with $p=-1$ is discussed in \cite{Kalmenov:2011to}.

Let $\{\mathcal{C}^{\alpha}_m: m\in\N,\; \alpha>-1/2\}$ denote the family of Gegenbauer polynomials with the standard normalization 
\begin{equation*}
  \mathcal{C}^{\alpha}_m(1) =\binom{m+2\alpha-1}{m}= \frac{\Gamma(m+2\alpha)}{\Gamma(2\alpha)\Gamma(m+1)},\quad \alpha\neq 0.
\end{equation*}
By $\alpha=\frac{d}{2}-1$, the addition theorem for spherical harmonics yields
\begin{equation}\label{eq:C into Y}
 \sum_{l=1}^{Z(d,m)} Y^m_{l}(x)\overline{Y^m_{l}(y)} 
  %= \frac{Z(d,m)}{\mathcal{C}_m^{\frac{d}{2}-1}(1)}\mathcal{C}^{\frac{d}{2}-1}_m(\langle x,y\rangle)
  =  \tfrac{2m+d-2}{d-2} \mathcal{C}^{\frac{d}{2}-1}_m(\langle x,y\rangle),\quad x,y\in\S^{d-1}.
\end{equation}
For $m\in\N$, let us define the kernels $\mathcal{K}^{d,p}_m:[0,1]\times[0,1]\rightarrow \R$,
\begin{equation}\label{eq:kernel K}
%\mathcal{K}^{d,p}_m(r,s) = c^{d,p}_{m}\Big(\frac{s}{r}\Big)^n  \pFq{2}{1}{\frac{1}{2}-\frac{d}{2},m-\frac{1}{2}}{m+\frac{d}{2}}{\Big(\frac{s}{r}\Big)^2},\quad r\geq s,
 \mathcal K_m^{d,p}(r,s) 
% &:= \frac{(-\frac p2)_m}{(\frac d2 -1 )_m} (r^2+s^2)^{\frac p2 - m} (r s)^m \;_2F_1\left( \tfrac m2-\tfrac p4, \tfrac{m+1}{2} -\tfrac p4; m+\tfrac{d}{2}; 4 \Big(\frac{sr}{s^2+r^2} \Big)^2 \right),\\
  = \! \frac{(-\frac p2)_m}{(\frac d2 -1 )_m} \left(\tfrac{\min(r,s)}{\max(r,s)}\right)^m \!\!\max(r,s)^p \pFq{2}{1}{m-\frac{p}{2},1-\frac{d+p}{2}}{m+\frac{d}{2}}{\left(\tfrac{\min(r,s)}{\max(r,s)}\right)^2}.
\end{equation}
%\begin{proposition}\label{prop:Kn}
For $d\geq 3$ and arbitrary real $p>1-d$, we deduce from \cite{Cohl:2013ao} that 
%the following decomposition into radial and spherical components holdswe have 
\begin{equation}\label{eq:exp of Kn}
\|x-y\|^p = \sum_{m=0}^\infty \mathcal{K}^{d,p}_m(\|x\|,\|y\|) \,\mathcal{C}^{\frac{d}{2}-1}_m\big(\big\langle \tfrac{x}{\|x\|},\tfrac{y}{\|y\|}\big\rangle\big),\quad x,y\in\mathbb{B}^d,\; (x\ne y \text{ if } p < 0).
\end{equation}
%\end{proposition}
For $x=0$ or $y=0$, the 
right-hand side of \eqref{eq:exp of Kn} is well-defined by 
analytic continuation. 
%The proof is on Page \pageref{proof:kn ding} in the appendix. 

Using the addition formula \eqref{eq:C into Y} we obtain
\begin{equation*}\label{eq:exp of Kn_add}
K^{d,p}(x,y) = \sum_{m=0}^\infty  \tfrac{d-2}{2m+d-2} \mathcal{K}^{d,p}_m(\|x\|,\|y\|) \sum_{l=1}^{Z(d,m)}Y^{m}_{l}(\tfrac{x}{\|x\|}) Y^{m}_{l}(\tfrac{y}{\|y\|}),\quad x,y\in\mathbb{B}^d.
\end{equation*}
The Fourier expansion of $\mathcal K_{m}^{d,p}$ with respect to the measure $r^{d-1} \mathrm dr$ satisfies 
\begin{equation}
  \label{eq:spec_KM}
  \mathcal K_{m}^{d,p}(r,s) = \sum_{j=1}^{\infty} \lambda_{m,j}^{d,p} \varphi_{m,j}^{d,p}(r) \varphi_{m,j}^{d,p}(s),
  \qquad \int_{0}^{1} \mathcal K_{m}^{d,p}(r,s) \varphi_{m,j}^{d,p}(r)r^{d-1} \mathrm d r = \lambda_{m,j}^{d,p} \varphi^{d,p}_{m,j}(s),
\end{equation}
where $\int_{0}^{1} |\varphi_{m,j}^{d,p}(r)|^{2} r^{d-1}\mathrm dr = 1$. Then, by setting
\begin{equation*}
\varphi_{m',j',l'}^{d,p}(x):= \varphi_{m',j'}^{d,p}(\|x\|) Y^{m'}_{l'}(\tfrac{x}{\|x\|}),  \quad x \in \mathbb B^{d}, 
\end{equation*}
direct computations yield %we arrive at the relation
\begin{equation}\label{eq:exp of Kn_add_int}
T^{d,p} \varphi_{m',j',l'}^{d,p}(x)  = \lambda_{m',j'}^{d,p} \tfrac{(d-2) \mathrm{vol}(\mathbb S^{d-1})}{2m'+d-2} \varphi_{m',j',l'}^{d,p}(x)
\end{equation}
with the scaling $\int_{\mathbb B^{d}}\left| \varphi_{m',j',l'}^{d,p}(x)\right|^{2} \mathrm d x
%  =  \mathrm{vol}(\mathbb S^{d-1})  \int_{0}^{1} |\varphi_{m',j'}^{d,p}(r)|^{2} r^{d-1} \mathrm d r \int_{\mathbb S^{d-1}} |Y_{l'}^{m'}(z)|^{2}  \mathrm d \sigma_{d-1}(z)
  = \mathrm{vol}(\mathbb S^{d-1})$. 
%{\color{blue}
%\begin{equation}\label{eq:exp of Kn_add_int}
%  \begin{aligned}
%    & T^{d,p} \varphi_{m',j',l'}^{d,p}(x) \\
%    = &\int_{\mathbb B^{d}} K^{d,p}(x,y) \varphi_{m',j',l'}^{d,p}(y) \mathrm dy \\
%     = &\int_{\mathbb B^{d}} \sum_{m=0}^\infty  \frac{d-2}{2m+d-2} \mathcal{K}^{d,p}_m(\|x\|,\|y\|)  \varphi_{m',j'}^{d,p}(\|y\|) \sum_{l=1}^{Z(d,m)}Y^{m}_{l}(\tfrac{x}{\|x\|}) Y^{m}_{l}(\tfrac{y}{\|y\|}) Y^{m'}_{l'}(\tfrac{y}{\|y\|}) \mathrm dy\\
%     = &\sum_{m=0}^\infty   \frac{(d-2) \mathrm{vol}(\mathbb S^{d-1})}{2m+d-2}  \int_{0}^{1} \mathcal{K}^{d,p}_m(\|x\|,r) \varphi_{m',j'}^{d,p}(r) r^{d-1} \mathrm dr \sum_{l=1}^{Z(d,m)}Y^{m}_{l}(\tfrac{x}{\|x\|}) \int_{\mathbb S^{d-1}} Y^{m}_{l}(z) Y^{m'}_{l'}(z)  \mathrm d \sigma_{d-1}(z) \\
%     = & \frac{(d-2) \mathrm{vol}(\mathbb S^{d-1})}{2m'+d-2} Y^{m'}_{l'}(\tfrac{x}{\|x\|}) \int_{0}^{1} \mathcal{K}^{d,p}_{m'}(\|x\|,r) \varphi_{m',j'}^{d,p}(r) r^{d-1} \mathrm dr\\
%     = & \lambda_{m',j'}^{d,p} \frac{(d-2) \mathrm{vol}(\mathbb S^{d-1})}{2m'+d-2} \varphi_{m',j',l'}^{d,p}(x)
%\end{aligned}
%\end{equation}
%and
%\[
%  \int_{\mathbb B^{d}}| \varphi_{m',j',l'}^{d,p}(x)|^{2} \mathrm d x
%  =  \mathrm{vol}(\mathbb S^{d-1})  \int_{0}^{1} |\varphi_{m',j'}^{d,p}(r)|^{2} r^{d-1} \mathrm d r \int_{\mathbb S^{d-1}} |Y_{l'}^{m'}(z)|^{2}  \mathrm d \sigma_{d-1}(z)
%  = \mathrm{vol}(\mathbb S^{d-1}).
%\]
%}
This leads to the Fourier expansion
\begin{align}\label{eq:spec total}
K^{d,p} (x,y)
%&= \sum_{m=0}^\infty \sum_{j=1}^\infty\lambda^{d,p}_{m,j} \sum_{l=1}^{Z(d,m)}\phi^{d,p}_{m,j,l}(x) \phi^{d,p}_{m,j,l}(y)\\
& = \sum_{m=0}^\infty \sum_{j=1}^\infty \lambda_{m,j}^{d,p} \tfrac{(d-2)\mathrm{vol}(\mathbb S^{d-1})}{2m+d-2} \sum_{l=1}^{Z(d,m)}\frac{\varphi^{d,p}_{m,j,l}(x) }{\sqrt{\mathrm{vol}(\mathbb S^{d-1})} } \frac{\varphi^{d,p}_{m,j,l}(y)}{\sqrt{\mathrm{vol}(\mathbb S^{d-1})}} 
%\int_{\mathbb B^{d}} |\varphi_{m,j,l}^{d,p}(z)|^{2} \mathrm dz} 
,\quad x,y\in\mathbb{B}^d.
\end{align}
%The eigenfunctions of $K^{d,p}$ with nonzero eigenvalues $\lambda^{d,p}_{m,j}$ can then be chosen of the form 
%\begin{equation}\label{eq:eig special form}
%\phi^{d,p}_{m,j,l} (x)=\varphi^{d,p}_{m,j}(\|x\|) Y^m_l(\tfrac{x}{\|x\|}),\quad x\in\mathbb{B}^d,
%\end{equation}
%where $\varphi^{d,p}_{m,j}:[0,1]\rightarrow\R$ are the eigenfunctions of $\mathcal{K}^{d,p}_m$ with nonzero eigenvalue $\frac{d-2}{2m+d-2}\lambda^{d,p}_{m,j}$ with respect to the measure $r^{d-1}{\rm d}r$. 
Thus, the original problem is reduced to the spectral decomposition of the sequence of kernels $\mathcal{K}^{d,p}_m$, for $m\in\N$. The kernel $\mathcal{K}^{d,p}_m$ induces the integral operator
\begin{equation}\label{eq:def of Tm}
T_m^{d,p} : L_2([0,1],r^{d-1}{\rm d}r) \rightarrow \mathcal{C}([0,1]),\quad f\mapsto  \int_0^1 \mathcal{K}^{d,p}_m(\cdot,r)f(r)r^{d-1}{\rm d}r,
\end{equation}
with eigenvalues $\lambda^{d,p}_{m,j}$ and eigenfunctions $\varphi^{d,p}_{m,j}$. 
We now specify these eigenvalues and eigenfunctions, where $J_{\nu}$ denotes the Bessel function of the first kind of order $\nu$ and 
$\zeta_{k}:={\rm e}^{2\pi {\rm i}/k}$ is the $k$-th root of unity.
\begin{thm}\label{thm:ball ultimate}
Suppose that both $d\geq 3$ and $p>1-d$ are odd and let $m\in\N$. Then the following holds:
\begin{itemize}
\item[a)] Any eigenvalue $\lambda\neq 0$ of $T_m^{d,p}$ is in a one-to-one correspondence with 
\begin{equation}\label{eq:general solutions Dn2}
\omega =   \big|\lambda^{-1}2^{d+p-2}(d+2m-2)(-\tfrac p2)_{\frac{d+p}{2}-1}(\tfrac{d+p}{2}-1)!\big|^{\frac{1}{d+p}}
%\mathcal{J}^{d,\omega_l}_{m},\qquad \ell=0,\ldots,\frac{d+p}{2}-1,
\end{equation}
with $\omega$ satisfying $\det(A(\omega))=0$, where 
\begin{equation}\label{eq:det formula eigenvalue}
  A(\omega) =
  \begin{cases}
      \left(\zeta_{d+p}^{-i\ell} J_{m+\frac{d}{2}-i-1}(\zeta_{d+p}^{\ell}\omega)
      \right)_{i=1,\,\ell=0}^{\frac{d+p}{2},\,\frac{d+p}{2}-1}, & (-\tfrac p2 )_{\frac{d+p}{2}-1}  \lambda > 0,\\
      \left(\zeta_{2(d+p)}^{-i(2\ell+1)}J_{m+\frac{d}{2}-i-1}(\zeta_{2(d+p)}^{2\ell+1}\omega)
      \right)_{i=1,\,\ell=0}^{\frac{d+p}{2},\,\frac{d+p}{2}-1}, & (-\tfrac p2 )_{\frac{d+p}{2}-1}  \lambda < 0.\\
    \end{cases}
    % A(\omega)=\begin{pmatrix}
%\zeta_{d+p}^{-1\cdot 0} J_{m+\frac{d-2}{2}-1}(\zeta_{d+p}^0\omega) & \hdots & \zeta_{d+p}^{-1\cdot(\frac{d+p}{2}-1)} J_{m+\frac{d-2}{2}-1}(\zeta_{d+p}^{\frac{d+p}{2}-1}\omega)\\
% \vdots & \ddots & \vdots\\
%      \zeta_{d+p}^{-\frac{d+p}{2} \cdot 0} J_{m+\frac{d-2}{2}-\frac{d+p}{2}}(\zeta_{d+p}^0\omega) &  \hdots & \zeta_{d+p}^{-\frac{d+p}{2}\cdot(\frac{d+p}{2}-1)} J_{m+\frac{d-2}{2}-\frac{d+p}{2}}(\zeta_{d+p}^{\frac{d+p}{2}-1}\omega)
%\end{pmatrix}.
\end{equation}
\item[b)] The eigenfunctions are exactly 
\begin{equation*}
  r\mapsto \begin{cases}
    \sum_{\ell=1}^{\frac{d+p}{2}} c_\ell \,r^{1-\frac{d}{2}}J_{m+\frac{d}{2}-1}(\zeta^{\ell-1}_{d+p}\omega r ),& (-\tfrac p2 )_{\frac{d+p}{2}-1}  \lambda > 0,\\
    \sum_{\ell=1}^{\frac{d+p}{2}} c_\ell \,r^{1-\frac{d}{2}}J_{m+\frac{d}{2}-1}(\zeta^{2\ell-1}_{2(d+p)}\omega r ),& (-\tfrac p2 )_{\frac{d+p}{2}-1}  \lambda < 0,\\
  \end{cases}
\end{equation*}
where $c\in\R^{\frac{d+p}{2}}$ is in the nullspace of $A(\omega)$. 

\end{itemize}
\end{thm}

\begin{remark}
Computer experiments seem to indicate that the nullspace of $A(\omega)$ is one-dimensional if $\det(A(\omega))=0$. In that case, the function 
\begin{equation}\label{eq:minor formula}
  r\mapsto \begin{cases}
    \sum_{\ell=1}^{\frac{d+p}{2}} (-1)^{\ell} A_{[1,\ell]}(\omega) \,r^{1-\frac{d}{2}}J_{m+\frac{d}{2}-1}(\zeta^{\ell-1}_{d+p}\omega r ),& (-\tfrac p2 )_{\frac{d+p}{2}-1}  \lambda > 0,\\
    \sum_{\ell=1}^{\frac{d+p}{2}} (-1)^{\ell} A_{[1,\ell]}(\omega) \,r^{1-\frac{d}{2}}J_{m+\frac{d}{2}-1}(\zeta^{2\ell-1}_{2(d+p)}\omega r ),& (-\tfrac p2 )_{\frac{d+p}{2}-1}  \lambda < 0,   
  \end{cases}
\end{equation}
where $A_{[1,\ell]}(\omega)$ denotes the $(1,\ell)$ minor of $A(\omega)$, spans the eigenspace  associated 
with $\lambda$.
 \end{remark}

Appendix~\ref{sec:app sec ball} is dedicated to the proof of Theorem~\ref{thm:ball ultimate}. The proof reveals strong ties with polyharmonic operators on the unit ball and higher order differential operators on the interval $[0,1]$. We refer to \cite{Adcock:2011by} for structurally related spectral decompositions of polyharmonic operators on $[0,1]$ with homogeneous Neumann boundary conditions.   
%We now specify the eigenfunctions:
%\begin{corollary}
%Under the assumptions and notations of Theorem \ref{thm:ball ultimate}, let $\det(A(\omega))=0$. The eigenfunctions of $T_m^{d,p}$ with respect to $\lambda$ are exactly  
%\begin{equation*}
%r\mapsto \sum_{\ell=1}^{\frac{d+p}{2}} c_\ell \frac{J_{m+\frac{d}{2}-1}(\zeta^{\ell-1}_{d+p}\omega r )}{r^{\frac{d}{2}-1}},
%\end{equation*}
%where $c\in\R^{\frac{d+p}{2}}$ is in the nullspace of $A(\omega)$.  In particular, 
%\begin{equation*}
%r\mapsto \sum_{\ell=1}^{\frac{d+p}{2}} (-1)^{\ell} \det(A^{[\ell]}(\omega)) \frac{J_{m+\frac{d}{2}-1}(\zeta^{\ell-1}_{d+p}\omega r )}{r^{\frac{d}{2}-1}}
%\end{equation*}
%is in the eigenspace associated to $\lambda$, where $A^{[\ell]}(\omega)$ is the $(1,\ell)$ minor of $A(\omega)$.
%\end{corollary}

\begin{corollary}[$d=3$, $p=1$]\label{cor:example}
The nonzero eigenvalues of $T^{3,1}_m$ for $m\in\N\setminus \{0\}$ are exactly the positive solutions of the equation
      \begin{equation*}
       J_{m-\frac{1}{2}}(\omega) J_{m-\frac{3}{2}}( \mathrm i \omega) - \mathrm i J_{m-\frac{1}{2}}( \mathrm i \omega)  J_{m-\frac{3}{2}}(\omega)
      = 0,
      \end{equation*}
with $\lambda=-\omega^{-4}(4m+2)$. The corresponding eigenspaces are $1$-dimensional with the representative
 \begin{equation}\label{eq:minor formula2}
   f_{\lambda}(r) 
       = r^{-\frac 12}\left(J_{m+\frac{1}{2}}(\omega r)\mathrm i J_{m-\frac{1}{2}}(\mathrm i \omega) +  J_{m+\frac{1}{2}}(\mathrm i \omega r)J_{m-\frac{1}{2}}(\omega)\right).
      \end{equation}
%{\color{red}provide also $\mathcal{K}_0^{d,1}$ for $1-\frac{\Gamma(\frac{d}{2})}{2\sqrt{\pi}\Gamma(\frac{d+1}{2})}\|x-y\|$ and provide everything for $m=0$ mit und ohne Konstante. }
%      {\color{red}decay rate of eigenvalues!!! 
%      
%       }
\end{corollary}
%\begin{proof}[Proof of Corollary \ref{cor:example}]
The formulas in Corollary~\ref{cor:example} are derived from Theorem~\ref{thm:ball ultimate}. Since $J_{m-\frac{1}{2}}({\rm i}\omega)\neq 0$, for all $\omega\in\R\setminus\{0\}$, the eigenspaces are $1$-dimensional and \eqref{eq:minor formula2} is not the zero-function.  

%\begin{remark}
In view of \eqref{eq:dist is rkh} in Theorem~\ref{thm:compact set} we are particularly interested in kernels of the form $c-\|x-y\|$. In this case, the expansion holds with $-\mathcal{K}^{d,1}_m$ for $m\geq 1$ and $c-\mathcal{K}^{d,1}_0$ for $m=0$. 
%\end{remark}

%
%Note that, 
%up to a constant depending on $d$ and $p$, the kernel $K^{d,p}$ is the Green's function of a $d$-dimensional, polyharmonic equation $\Delta^{\frac{d+p}{2}}u=f$ on $\mathbb{B}^d$ with certain nonlocal boundary conditions, cf.~\cite{Kalmenov:2012fq}. By using spherical coordinates, each $\mathcal{K}^{d,p}_m$ is the Green's function of a $1$-dimensional differential equation $D_m^{\frac{d+p}{2}} u =f$ on $[0,1]$ with boundary conditions in $r=1$, where  
%\begin{equation*}
%D_m := \partial^{2}_r + \frac{d-1}{r}\partial_r - m(m+d-2) \frac{1}{r^{2}}.
%\end{equation*}
%Let us specify the situation for $d=3$ and $p=1$:
%\begin{example}[$d=3$, $p=1$]
%It can be checked that the $1$-dimensional kernel $\mathcal{K}^{3,1}_m$ is the Green's function of the $1$-dimensional differential equation $D_m^2 u =f$ with boundary conditions 
%\begin{equation}\label{eq:boundary 1 d in example}
%\begin{pmatrix}
%(m-1)(m+1) & 2m+1&1&0\\
%0& m(m+2) & 2m+3 & 1
%\end{pmatrix}
%\begin{pmatrix}
%u(1)\\ u'(1)\\u''(1)\\ u'''(1)
%\end{pmatrix}
%=0.
%\end{equation}
%{\color{red}Technically, we need 4 equations, no?!? We got rid of second kind Bessel, how do we formulate this here?}
%
%Thus, the $3$-dimensional kernel $(x,y)\mapsto \mathcal{K}^{3,1}_m(\|x\|,\|y\|) \,\mathcal{C}^{\frac{3}{2}-1}_m\big(\big\langle \tfrac{x}{\|x\|},\tfrac{y}{\|y\|}\big\rangle\big)$ is associated to the polyharmonic equation $\Delta^2 u =f$ on $\mathbb{B}^3$ and \eqref{eq:boundary 1 d in example} translates into normal boundary conditions on $\S^{2}$. 
% 
%
%
%\end{example}

\section{The rotation group $\SO(3)$}\label{sec:SO3}
In this section we derive the Fourier expansion of the discrepancy kernel on the special orthogonal group $\SO(3)$. The eigenfunctions turn out to be classical functions but the coefficients and their decay rates need to be determined. We 
also provide numerical experiments by using the nonequispaced fast Fourier transform on $\SO(3)$.
\subsection{Fourier expansion on $\SO(3)$}
By identifying $\R^{d\times d}$ with $\R^{d^2}$, Theorem~\ref{thm:compact set} applies to subsets of $\R^{d\times d}$ endowed with the trace inner product 
\begin{equation*}
\langle x,y\rangle_{\F}:=\trace(x^\top y),\quad x,y\in\R^{d\times d},
\end{equation*}
and the induced Frobenius norm $\|\cdot\|_{\F}$ on $\R^{d\times d}$. 
In this way, $\SO(3)$ is contained in $\mathbb{B}^{9}_{\sqrt{3}}$, and it is natural to consider ${s}=\sqrt{3}$. We endow $\SO(3)$ with the normalized Haar measure $\sigma_{\SO(3)}$. Let $\{\mathcal{D}^m_{k,l}:k,l=-m,\ldots,m\}$ denote the Wigner $\mathcal{D}$-functions on $\SO(3)$, which are closely related to the irreducible representations of $\SO(3)$ and provide an orthonormal basis for $L_2(\SO(3))$, cf.~\cite{Varshalovich:1988vh}. For $p>0$, the Fourier expansion
\begin{equation}\label{eq:SO exp}
2^{-\frac{p}{2}}\|x-y\|^p_{\F} =\sum_{m=0}^\infty a_m(p,\SO(3)) \sum_{k,l=-m}^m \mathcal{D}^m_{k,l}(x)\overline{\mathcal{D}^m_{k,l}(y)},\quad x,y\in\SO(3),
\end{equation}
holds and, analogous to \eqref{eq:ft coeff}, the coefficients are determined by
 \begin{equation}\label{eq:ft coeff2}
a_m(p,\SO(3)):=\iint\limits_{\SO(3)\times \SO(3)}2^{-\frac{p}{2}}\|x-y\|_{\F}^p \overline{\mathcal{D}^m_{k,l}(x)}\mathcal{D}^m_{k,l}(y){\rm d}\sigma_{\SO(3)}(x){\rm d}\sigma_{\SO(3)}(y).
\end{equation}
We now compute these coefficients for the entire range $p>-3$.
\begin{proposition}\label{prop:SO}
For $p > -3$, the coefficients \eqref{eq:ft coeff2} are 
\begin{equation}\label{eq:into SO}
  a_m(p,\SO(3)) =   \frac{2^p\Gamma(\frac{p}{2}+\frac{3}{2})}{\sqrt{\pi}\Gamma(-\frac{p}{2})}\cdot\frac{\Gamma(-\frac{p}{2}+m)}{\Gamma(\frac{p}{2}+2+m)}\cdot\frac{1}{(m+\frac{1}{2})}.
\end{equation}
In particular, if $p\not\in 2\N$, then 
\begin{equation*}
|a_m(p,\SO(3))| = \left| \frac{2^p\Gamma(\frac{p}{2}+\frac{3}{2})}{\sqrt{\pi}\Gamma(-\frac{p}{2})}\right|m^{-(p+3)}(1+o(1)),\qquad m\in\N,
\end{equation*}
and the series \eqref{eq:SO exp} terminates if $p\in 2\N$. 
\end{proposition}
The proof is given in Appendix~\ref{sec:A6}. For $p\in 2\N$, we again apply the convention $\frac{\Gamma(-\frac{p}{2}+m)}{\Gamma(-\frac{p}{2})}=(-\frac{p}{2})_m$ in \eqref{eq:into SO}, so that $a_m(p,\SO(3))=0$ for all $m>\frac{p}{2}$ if $p\in 2\N$. Provided that $\tau>3/2$, the Sobolev space $\mathbb{H}^\tau(\SO(3))$ is the reproducing kernel Hilbert space associated with the reproducing kernel
\begin{equation*}
(x,y)\mapsto \sum_{m=0}^\infty (1+m(m+1))^{-\tau}\sum_{k,l=-m}^m \mathcal{D}^m_{k,l}(x)\overline{\mathcal{D}^m_{k,l}(y)},\quad x,y\in\SO(3).
\end{equation*}
The choice $p=1$ in Proposition~\ref{prop:SO} implies that the kernel $K_{\beta_{9,s}}|_{\SO(3)\times\SO(3)}$ reproduces the Sobolev space $\mathscr{H}_{K_{\beta_{9,s}}}(\SO(3))=\mathbb{H}^{2}(\SO(3))$ with an equivalent norm provided that $s\geq \sqrt{3}$.

%\begin{remark}\label{rem:001}
%The statements on the reproducing kernel Hilbert space holds in more generality. 
For $d\geq 2$, $K_{\beta_{d^2,\sqrt{d}}}|_{\SO(d)\times\SO(d)}$ reproduces the Sobolev space $\mathbb{H}^{\frac{d(d-1)+2}{4}}(\SO(d))$ with an equivalent norm. Indeed, Theorem~\ref{thm:compact set} and Section~\ref{sec:sphere} yield that $K_{\beta_{d^2,d}}|_{\S^{d^2-1}\times \S^{d^2-1}}$ reproduces $\mathbb{H}^{\frac{d^2}{2}}(\S^{d^2-1})$ with equivalent norms. Rescaling implies that $K_{\beta_{d^2,d}}|_{d\,\S^{d^2-1}\times d\,\S^{d^2-1}}$ reproduces $\mathbb{H}^{\frac{d^2}{2}}({d\,}\S^{d^2-1})$. Since $\SO(d)\subset d\,\S^{d^2-1}$, the assertion for $\SO(d)$ follows from results on restricting kernels in \cite{Fuselier:2012jt}.  
%\end{remark}

\subsection{Numerical examples on $\SO(3)$}
Proposition~\ref{prop:SO} yields the coefficients of the kernel expansion
\begin{equation*}
K_{\beta_{9,\sqrt{3}}}(x,y) = \sum_{m=0}^\infty a_m \sum_{k,l=-m}^m \mathcal{D}^m_{k,l}(x)\overline{\mathcal{D}^m_{k,l}(y)},\quad x,y\in\SO(3).
\end{equation*}
For $\supp(\mu),\supp(\nu_n)\subset \SO(3)$, the $L_2$-discrepancy \eqref{eq:fund eq min} for $K_{\beta_{9,\sqrt{3}}}$ becomes
\begin{equation}\label{eq:series truncatingaSO}
\mathscr{D}_{\beta_{9,\sqrt{3}}}(\mu,\nu_n) = \sum_{m=0}^\infty a_m \sum_{k,l=-m}^m  \left|\hat{\mu}^m_{k,l}- \frac{1}{n}\sum_{j=1}^n \overline{\mathcal{D}^m_{k,l}(x_j)}\right|^2,
\end{equation}
where $\hat{\mu}^m_{k,l}$ denotes the Fourier coefficient of $\mu$ with respect to $\mathcal{D}^m_{k,l}$, cf.~\eqref{eq:fund eq min}. We 
truncate the series \eqref{eq:series truncatingaSO} at $M=8$ and minimize 
\begin{equation}\label{eq:min SO}
\sum_{m=0}^M a_m \sum_{k,l=-m}^m  \left|\hat{\mu}^m_{k,l} - \frac{1}{n}\sum_{j=1}^n \overline{\mathcal{D}^m_{k,l}(x_j)}\right|^2
\end{equation}
among all $n$-point sets $\{x_1,\ldots,x_n\}\subset\SO(3)$ for fixed $n=30$. We efficiently solve the least squares minimization by using the nonequispaced fast Fourier transform on $\SO(3)$, cf.~\cite{Graf:2009ye,Potts:2009gb}. 
%For visualization, $\SO(3)$ is parametrized by unit quaternions, the latter corresponds to $\S^3$, which is mapped to $\mathbb{B}^3$ by stereographic projection. 
Figure~\ref{fig:SO3} shows the minimizing points mapped onto $\mathbb{B}^3$. 
\begin{figure}
\includegraphics[width=.32\textwidth]{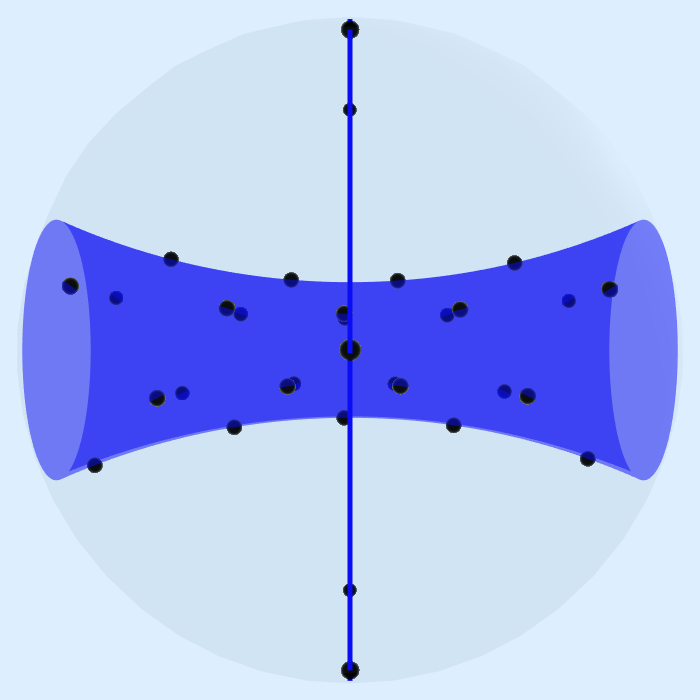}
\includegraphics[width=.32\textwidth]{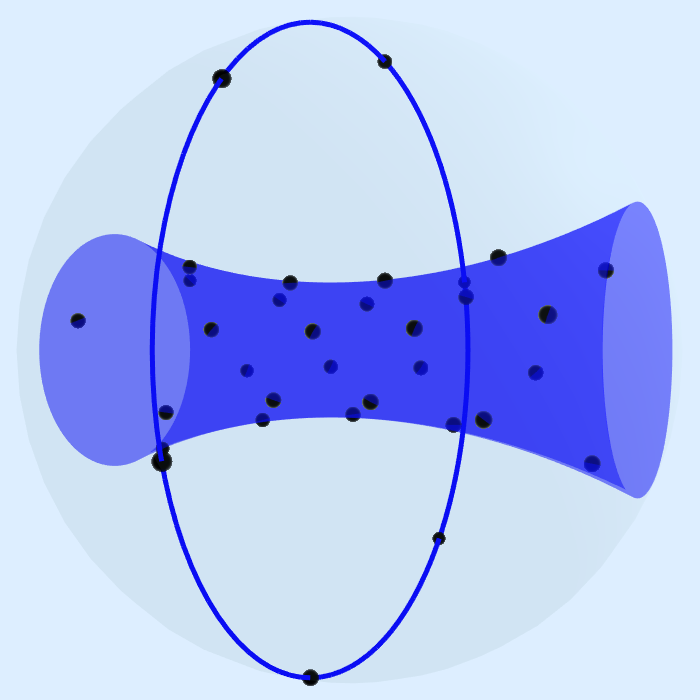}
\includegraphics[width=.32\textwidth]{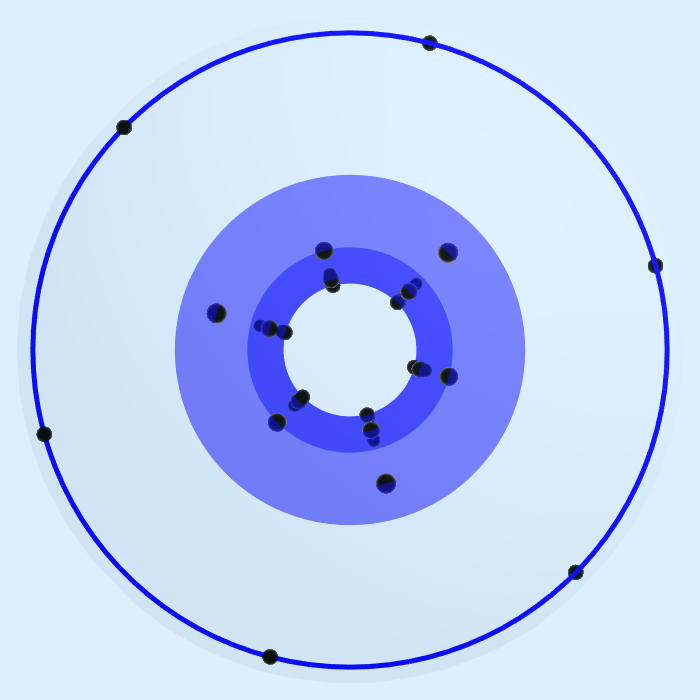}
\caption{We use the standard parametrization of  $\SO(3)$ by $\S^3$ via unit quaternions, which is then mapped into $\mathbb{B}^3$ by stereographic projection. The target measure  $\mu$ is supported on two disjoint parts with weight ratio $9/1$ colored in darker blue by the cylindrical surface and the great circle. Numerical minimization of \eqref{eq:min SO} splits $30$ points in $\SO(3)$ into $27$ points on the inner surface and $3$ points on the great circle. We plotted $6$ points on the great circle but antipodal points correspond to the same point in $\SO(3)$.}\label{fig:SO3}
\end{figure}

\section{The Grassmannian $\G_{2,4}$}\label{sec:d}
First, the Fourier expansion of the discrepancy kernel on $\G_{2,4}$ is computed. To prepare for developing the nonequispaced fast Fourier transform on $\G_{2,4}$, we then explicitly parametrize the Grassmannian $\G_{2,4}$ by its double covering $\S^2\times\S^2$. Next, we derive the nonequispaced fast Fourier transform on $\S^2\times\S^2$ and provide numerical minimization experiments on $\G_{2,4}$. 
\subsection{Fourier expansion on $\G_{2,4}$}\label{sec:G24 a}
Proposition~\ref{thm:compact set} also applies to the Grassmannian
\begin{equation*}
\G_{2,4}:=\{x\in\R^{4\times 4} : x^\top = x,\; x^2=x,\; \rank(x)=2\}
\end{equation*}
with $s=\sqrt{2}$ when $\R^{4\times 4}$ is identified with $\R^{16}$. To derive the Fourier expansion on $\G_{2,4}$, we require some preparations. We shall use integer partitions $\lambda=(\lambda_1,\lambda_2)\in\N^2$ with $\lambda_1\geq \lambda_2\geq 0$. We also denote $|\lambda|:=\lambda_1+\lambda_2$. The orthogonal group $\OOO(4)$ acts transitively on $\G_{2,4}$ by conjugation and induces the irreducible decomposition 
\begin{equation}\label{eq:decomp L2 single}
L_{2}(\G_{2,4},\sigma_{\G_{2,4}}) = \bigoplus_{\lambda_1\geq\lambda_2\geq 0} H_{\lambda}(\G_{2,4}),\qquad
H_{\lambda}(\G_{2,4}) \perp H_{\lambda'}(\G_{2,4}), \quad \lambda \ne \lambda',
\end{equation}
where $\sigma_{\G_{2,4}}$ is the normalized orthogonally invariant measure on $\G_{2,4}$ and $H_{\lambda}(\G_{2,4})$ is equivalent to the
irreducible representation $\mathcal H_{2\lambda}^{4}$ of $\OOO(4)$ with type 
$2\lambda$, cf.~\cite{Bachoc:2002aa,James:1974aa}. The normalized eigenfunctions of the Laplace--Beltrami operator on $\G_{2,4}$ form an orthonormal basis for $L_{2}(\G_{2,4})$, and each $H_\lambda(\G_{k,d})$ is contained in the eigenspace $E_{\alpha_\lambda}$ associated 
with the eigenvalue $\alpha_\lambda = 4(\lambda_1^2+\lambda_2^2+\lambda_1)$, cf.~\cite{Bachoc:2006aa,Bachoc:2004fk,Bachoc:2002aa,Ehler:2014zl,James:1974aa,Roy:2010fk}. 

Let $Q_\lambda$ be the reproducing kernel of $H_\lambda(\G_{2,4})$. Any orthonormal basis $\{\varphi_{\lambda,l}\}_{l=1}^{\dim(\mathcal{H}^{4}_{2\lambda})}$ for $H_{\lambda}(\G_{2,4})$ yields the spectral decomposition 
\begin{equation}\label{eqLrep kernel Q}
Q_\lambda(x,y) = \sum_{l=1}^{\dim(\mathcal{H}^{4}_{2\lambda})}\varphi_{\lambda,l}(x)\overline{\varphi_{\lambda,l}(y)},\qquad x,y\in\G_{2,4},\quad \lambda_1\geq \lambda_2\geq 0.
\end{equation}
The orthogonal decomposition \eqref{eq:decomp L2 single} leads to the Fourier expansion 
\begin{equation}\label{eq:kernel def gras}
2^{-\frac{p}{2}}\| x -y\|_{\F}^{p}=\sum_{\lambda_1\geq\lambda_2\geq 0}a_\lambda(p,\G_{2,4}) Q_\lambda(x,y),\qquad x,y \in \G_{2,4},\quad p>0.
\end{equation}
The coefficients $a_\lambda(p,\G_{2,4})$ in \eqref{eq:kernel def gras} are defined by
\begin{equation}\label{eq:integ to be}
a_\lambda(p,\G_{2,4}) := \frac{1}{\dim(\mathcal{H}^{4}_{2\lambda})}\iint\limits_{\G_{2,4}\times \G_{2,4}}2^{-\frac{p}{2}}\| x -y\|_{\F}^{p} Q_\lambda(x,y) {\rm d}\sigma_{\G_{2,4}}(x){\rm d}\sigma_{\G_{2,4}}(y).
\end{equation}
In order to determine $a_\lambda(p,\G_{2,4}) $, we shall make use of the hypergeometric coefficients $ (f)_{(\lambda_1,\lambda_2)} := (f)_{\lambda_1} (f-\tfrac12)_{\lambda_2}$.
\begin{thm}\label{thm:G24}
For $p>-4$, 
we have
\begin{equation}\label{eq:a alpha}
  a_\lambda(p,\G_{2,4}) =   \frac{4^{-|\lambda|}|\lambda|!}{(\tfrac{3}{2})_{|\lambda|} (\tfrac{3}{2})_{\lambda}}(-\tfrac{p}{2})_{|\lambda|} 
  \;\pFq{4}{3}{\tfrac{|\lambda|+1}{2},\tfrac{|\lambda|+2}{2},\tfrac{|\lambda|}{2}-\tfrac{p}{4},\tfrac{|\lambda|+1}{2}-\tfrac{p}{4}}{|\lambda|+\tfrac{3}{2},\lambda_1+\tfrac{3}{2},\lambda_2+1}{1}.
\end{equation}
In particular, if $p\not\in 2\N$, then 
\begin{equation} \label{eq:dec-G24}
|a_\lambda(p,\G_{2,4})|= \left|\frac{\Gamma(\frac{p}{2}+2)}{2\Gamma(-\frac{p}{2})}\right|\|\lambda\|^{-(p+4)} (1+o(1)),
\qquad \lambda_1\geq\lambda_2\geq0,
\end{equation}
and the series \eqref{eq:kernel def gras} terminates if $p\in 2\N$. 
\end{thm}
The proof of this theorem is contained in Section~\ref{sec:A7.1} of
Appendix~\ref{sec:A7}. 
%The verification of \eqref{eq:dec-G24} starts on Page \pageref{proof:G2}. 
If $p\in 2\N$, then $ a_\lambda(p,\G_{2,4})=0$ for all $|\lambda|>\frac{p}{2}$. For $\tau>2$, the Sobolev space $\mathbb{H}^\tau(\G_{2,4})$ is the reproducing kernel Hilbert space with associated reproducing kernel
\begin{equation}\label{eq:coeff G str}
(x,y)\mapsto \sum_{\lambda_1\geq\lambda_2\geq 0}(1+4\lambda_1^2+4\lambda_2^2+4\lambda_1)^{-\tau}Q_\lambda(x,y),\quad x,y\in\G_{2,4},
\end{equation}
cf.~\cite{Brandolini:2014oz,Breger:2016rc}. 
 Since the coefficients in \eqref{eq:coeff G str} behave asymptotically as $\|\lambda\|^{-2\tau}$, the choice $p=1$ in Theorem~\ref{thm:G24} implies that the kernel $K_{\beta_{16,s}}|_{\G_{2,4}\times\G_{2,4}}$ reproduces the Sobolev space $\mathscr{H}_{K_{\beta_{16,s}}}(\G_{2,4})=\mathbb{H}^{\frac{5}{2}}(\G_{2,4})$ with an equivalent norm provided that $s\geq \sqrt{2}$. 

%\begin{remark}
By invoking \cite{Fuselier:2012jt}, we deduce that, for $d\geq 2$, $K_{\beta_{d^2,\sqrt{k}}}|_{\G_{k,d}\times\G_{k,d}}$ reproduces the Sobolev 
space $\mathbb{H}^{\frac{k(d-k)+1}{2}}(\G_{k,d})$ with an equivalent norm. 
%\end{remark}

\subsection{Parametrization of $\G_{2,4}$ by $\faktor{(\mathbb{S}^2\times \mathbb{S}^2)}{\pm 1}$}\label{sec:G24 b}
%We shall provide an explicit parametrization of $\G_{2,4}$ by $\faktor{(\mathbb{S}^2\times \mathbb{S}^2)}{\pm 1}$. 
%Efficient minimization of \eqref{eq:fund eq min} requires the fast Fourier transform, cf.~\cite{Graf:2013zl,Graf:2011lp}, which is not yet available for $\G_{2,4}$. Here, we shall construct a suitable parametrization of $\G_{2,4}$ by $\faktor{(\mathbb{S}^2\times \mathbb{S}^2)}{\pm 1}$ that enables the use of the fast Fourier transform on $\S^2\times\S^2$. The latter shall be designed from the respective transform on $\S^2$ in the subsequent section. 

To derive the nonequispaced fast\break Fourier transform on $\G_{2,4}$, we shall first explicitly construct the parametrization of $\G_{2,4}$ by its double covering $\S^2\times\S^2$. We denote the $d\times d$-identity matrix by $I_d$, and the cross-product between two vectors $x,y\in\S^2$ is denoted by $x\times y\in\R^3$. The mapping $\mathcal{P}:\S^2\times\S^2\rightarrow\G_{2,4}$ given by
\begin{equation}\label{eq:ultimate P}
\begin{aligned}
(x,y) \mapsto 
%& \frac{1}{2}
%\left(
%  \begin{smallmatrix}
%  1+x_1y_1+x_2y_2+x_3y_3   &   x_3y_2-x_2y_3         &   x_1y_3-x_3y_1         &   x_2y_1-x_1y_2         \\
%    x_3y_2-x_2y_3          & 1+x_1y_1-x_2y_2-x_3y_3  &   x_1y_2+x_2y_1         &   x_1y_3+x_3y_1         \\
%    x_1y_3-x_3y_1          &   x_1y_2+x_2y_1         & 1-x_1y_1+x_2y_2-x_3y_3  &   x_2y_3+x_3y_2         \\
%    x_2y_1-x_1y_2          &   x_1y_3+x_3y_1         &   x_2y_3+x_3y_2         & 1-x_1y_1-x_2y_2+x_3y_3  \\
%    \end{smallmatrix}
%    \right) = \\
    & \frac{1}{2}
    \begin{pmatrix}
    1 + x^\top y   &&   - (x \times y)^\top   \\
    - x \times y   & &x y^\top + y x^\top + (1-x^\top y) I_3   \\
    \end{pmatrix}%\in\G_{2,4}
    \end{aligned}
    \end{equation}
is surjective and, for all $x,y,u,v\in\mathbb{S}^2$,
\begin{equation}\label{eq:ambigui0}
\mathcal{P}(u,v)=\mathcal{P}(x,y)\quad \text{ if and only if } \quad (u,v) \in \{\pm (x,y) \},
\end{equation}
see Section~\ref{sec:proofs of sec 7.2} and Theorem~\ref{th:ide th} 
of Appendix~\ref{sec:A7}. 
%Let $\mathcal{P}$ denote the mapping specified in \eqref{eq:ultimate P} throughout the remaining part of this section.
In order to specify the inverse map, note that $\faktor{\mathbb{S}^2\times \mathbb{S}^2}{\pm 1}$ 
can be identified with
$%\begin{equation*}
\mathcal{M}(3):=\{xy^\top\in\mathbb{R}^{3\times 3} : x,y\in\mathbb{S}^2\}.
$ %\end{equation*}
We now define $\mathcal{L}:\mathcal{G}_{2,4}\rightarrow \mathcal{M}(3)$,
\begin{equation}\label{eq:inv of P}
\mathcal{L}(P) :=     \left(\begin{smallmatrix}
       \frac12(P_{11}+P_{22}-P_{33}-P_{44}) & P_{23} - P_{14}                      & P_{24} + P_{13}                   \\
      P_{23} + P_{14}                     & \frac12(P_{11}-P_{22}+P_{33}-P_{44}) & P_{34}- P_{12}                      \\
      P_{24} - P_{13}                     & P_{34} + P_{12}                    & \frac12(P_{11}-P_{22}-P_{33}+P_{44})  \\
    \end{smallmatrix}\right),
  \end{equation}
and direct computations lead to
\begin{equation}\label{eq:L L L}
\mathcal{L}(\mathcal{P}(x,y)) = xy^\top,\quad x,y\in\mathbb{S}^2.
\end{equation}
The right-hand side determines $x$ and $y$ up to the ambiguity \eqref{eq:ambigui0}. Under the Frobenius norm, $\mathcal{P}$ is distance preserving in the sense 
%identification of $\G_{2,4}$ with $\mathcal{M}(3)$, i.e.,
\begin{equation}\label{eq:isometry F}
\|\mathcal{P}(x,y)-\mathcal{P}(u,v)\|_{\F} = \|xy^\top - uv^\top\|_{\F},\quad x,y,u,v\in\mathbb{S}^2.
\end{equation}
The latter follows from \eqref{eq:skal pro vec} in Lemma~\ref{lemma:222} in
Section~\ref{sec:proofs of sec 7.2} of Appendix~\ref{sec:A7}. 

%The proof of Theorem \ref{th:Hpi} is based on the identity \eqref{eq:00} and the following lemma:

We shall now check how the spherical harmonics $Y^m_l$ on $\S^2$ relate to the eigenfunctions $\varphi_{\lambda,l}\in H_\lambda(\G_{2,4})$ of the Laplace--Beltrami operator on $\G_{2,4}$, cf.~\eqref{eqLrep kernel Q}. The functions $Y_{k,l}^{m,n}:\mathcal{G}_{2,4}\to \mathbb C$ given by
  \begin{equation}\label{eq:tensor struct}
    Y_{k,l}^{m,n}(\mathcal{P}(x,y)) := Y_k^m(x) \cdot Y_l^n(y)%\qquad k=-m,\dots,m,\quad l=-n,\dots,n,
  \end{equation}
are well-defined for $m+n \in 2\mathbb N$, the latter taking into account the ambiguity \eqref{eq:ambigui0}. 
\begin{thm}\label{th:Hpi}
For $m_\lambda:=(\lambda_1+\lambda_2)$ and $n_\lambda:=(\lambda_1-\lambda_2)$, we have
\begin{equation}\label{eq:rep of H}
H_{\lambda}(\mathcal{G}_{2,4}) = \mathrm{span}\{ Y_{k,l}^{m_\lambda,n_\lambda}, Y_{l,k}^{n_\lambda,m_\lambda} :
  k=-m_\lambda,\dots,m_\lambda,\; l=-n_\lambda,\dots,n_\lambda \}.
\end{equation}
\end{thm}
The proof is presented 
at the end of Section~\ref{sec:proofs of sec 7.2} of Appendix~\ref{sec:A7}. 
%\begin{remark}
Note that the geodesic distance on $\G_{2,4}$ is $\dist_{\G_{2,4}}(P,Q)=\sqrt{2}\sqrt{\theta_1^2+\theta_2^2}$, where $\theta_1,\theta_2\in[0,\pi/2]$ are the principal angles determined by the two largest eigenvalues $\cos^2(\theta_1)$ and $\cos^2(\theta_2)$ of the matrix $PQ$. 
Aside from \eqref{eq:isometry F}, $\mathcal{P}$ is also distance-preserving with respect to the respective geodesic distances, i.e.,\footnote[3]{The geodesic distance on $\S^2$ induces the geodesic distance on $\faktor{\mathbb{S}^2\times \mathbb{S}^2}{\pm 1}$ by
\begin{equation*}
\dist^2_{\faktor{\mathbb{S}^2\times \mathbb{S}^2}{\pm 1}}\big(\left(\begin{smallmatrix} x\\y\end{smallmatrix}\right),\left(\begin{smallmatrix} u\\v\end{smallmatrix}\right)\big)=\min\big\{\dist^2_{\S^2}(x,u)+\dist^2_{\S^2}(y,v),\dist^2_{\S^2}(-x,u)+\dist^2_{\S^2}(-y,v)\big\}.
\end{equation*}
} 
\begin{equation*}
\dist_{\G_{2,4}}(\mathcal{P}(x,y),\mathcal{P}(u,v)) = \dist_{\faktor{\mathbb{S}^2\times \mathbb{S}^2}{\pm 1}}\big(\left(\begin{smallmatrix} x\\y\end{smallmatrix}\right),\left(\begin{smallmatrix} u\\v\end{smallmatrix}\right)\big),\quad x,y,u,v\in\S^2.
\end{equation*} 
This equality follows from \eqref{eq:ambi scal xi} in Lemma~\ref{lemma:222} in the appendix 
via further direct calculations. 
%\end{remark}

%\begin{remark}\label{rem:phi to Y}
The identity \eqref{eq:rep of H} provides explicit expressions for the orthonormal basis $\{\varphi_{\lambda,l}\}_{l=1}^{\dim(\mathcal{H}^{4}_{2\lambda})}$ of $H_{\lambda}(\G_{2,4})$ that is used to construct the reproducing kernel $Q_\lambda$ in \eqref{eqLrep kernel Q}. It also provides a fast Fourier transform on $\G_{2,4}$ from the respective transform on $\S^2\times\S^2$ that is developed in the subsequent section.
%\end{remark}

\subsection{Nonequispaced Fast Fourier Transform on $\G_{2,4}$}\label{sec:FFT}
The nonequispaced fast (spherical) Fourier transform on $\S^2$ has been developed in \cite{Keiner:2009cv,Kunis:2003pt} under the acronym 
\texttt{nfsft}. Here, we shall derive the analogous transform on $\S^2\times\S^2$, which induces the nonequispaced fast Fourier transform on $\G_{2,4}$ via the mapping $\mathcal{P}$ and \eqref{eq:rep of H} with \eqref{eq:tensor struct}.

For a given finite set of coefficients
$f^{m_1,m_2}_{k,l}\in\mathbb{C}$, $m_1,m_2=0,\ldots,M$, $k=-m_1,\ldots,m_1$, $l=-m_2,\ldots,m_2$, 
we aim to evaluate
\begin{equation}\label{eq:f at scatter2}
F(x,y):=\sum_{m_1,m_2=0}^M \sum_{k=-m_1}^{m_1} \sum_{l=-m_2}^{m_2} f^{m_1,m_2}_{k,l} Y^{m_1}_k(x)Y^{m_2}_l(y)
\end{equation}
at $n$ scattered locations $(x_j,y_j)_{j=1}^n\subset\S^2\times\S^2$. 
%For $b_i\in\mathbb{C}$, $i=1,\ldots,M$, we also aim to compute
%\begin{equation}\label{eq:adjoit fft}
%\hat{b}^{m,n}_{k,l}:=\sum_{i=1}^M b_i \overline{Y^m_k(x_i)Y^n_l(y_i)}.
%\end{equation}
%Let us first take care of \eqref{eq:f at scatter2}. 
Direct evaluation of \eqref{eq:f at scatter2} leads to $O(n M^4)$ operations. We shall now derive an approximative algorithm that is more efficient for $n\gg M$. 

By following the ideas in \cite{Keiner:2009cv,Kunis:2003pt}, switching to spherical coordinates reveals that \eqref{eq:f at scatter2} is a $4$-dimensional trigonometric polynomial. This enables the use of the $4$-dimensional nonequispaced fast Fourier transform \texttt{nfft} to significantly reduce the complexity. In spherical coordinates the spherical harmonics are trigonometric polynomials such that
%\begin{equation*}
%z(\theta,\varphi) = \begin{pmatrix}
%\sin(\theta)\cos(\varphi)\\
%\sin(\theta)\sin(\varphi)\\
%\cos(\theta)
%\end{pmatrix}
%\in\S^2,\quad 0\leq\theta\leq\pi,\quad 0\leq\varphi\leq 2\pi,
%\end{equation*}
\begin{equation}\label{eq:Y in e}
Y^m_k(z(\theta,\varphi)) =e^{{\rm i}k\varphi} \sum_{k'=-m}^m c^m_{k,k'}e^{{\rm i}k'\theta},\qquad z(\theta,\varphi) = \left(\begin{smallmatrix}
\sin(\theta)\cos(\varphi)\\
\sin(\theta)\sin(\varphi)\\
\cos(\theta)
\end{smallmatrix}
\right)
\in\S^2,
\end{equation}
where $0\leq\theta\leq\pi$, $0\leq\varphi\leq 2\pi$, and $c^m_{k,k'}\in\mathbb{C}$ are suitable coefficients that we assume to be given or precomputed. Thus, for
$
x = z(\theta_1,\varphi_1)$ and $y=z(\theta_2,\varphi_2)$,  
there are coefficients $b^{k,l}_{k',l'}\in\mathbb{C}$ such that
\begin{equation}\label{eq:bsss}
F(x,y)=\sum_{k,l,k',l'=-M}^M b^{k,l}_{k',l'} e^{{\rm i}k\varphi_1}e^{{\rm i}l\varphi_2} e^{{\rm i} k'\theta_1}e^{{\rm i} l'\theta_2}.
\end{equation}
We check 
in Section~\ref{app:sec:FFT} of Appendix~\ref{sec:A7} that the set of coefficients $b^{k,l}_{k',l'}$ can be computed by $O(M^5)$ operations provided that the numbers  $c^m_{k,k'}$ in \eqref{eq:Y in e} are given. The expression \eqref{eq:bsss} can be evaluated at $n$ scattered locations by the nonequispaced discrete Fourier transform \texttt{ndft} with $O(n M^4)$ operations, cf.~\cite{Keiner:2009cv,Kunis:2003pt}. An efficient  approximative algorithm is the nonequispaced fast Fourier transform \texttt{nfft} that requires only $O(M^4\log(M)+n|\log(\epsilon)|^4)$ operations with accuracy $\epsilon$, see \cite{Keiner:2009cv,Kunis:2003pt} for details on accuracy. Thus, our algorithm for evaluating \eqref{eq:f at scatter2} at $n$ scattered locations requires $O(M^5+n|\log(\epsilon)|^4)$ operations. For $n\gg M$, this is a significant reduction in complexity compared to the original $O(nM^4)$ operations. We shall choose $n\sim M^4$ in the subsequent section, so that the complexity is reduced from $O(M^8)$ to $O(M^5+M^4|\log(\epsilon)|^4)$ operations. For potential further reduction, we refer to Remark \ref{rem:app poly transform} in the appendix.

%
%
%We still need to add the complexity of computing the coefficients $b^{k,l}_{k',l'}$ that correspond to the transition from $\{Y^m_k(x):m=0,\ldots,N,\; k=-m,\ldots,m\}$ in \eqref{eq:f at scatter2} to $\{e^{{\rm i}k\varphi}e^{{\rm i}l\theta}: k,l=-N,\ldots,N \}$ in \eqref{eq:bsss}. We check in the appendix on Page \pageref{app:sec:FFT} that they can be computed by $O(N^5)$ operations provided that the numbers  $c^m_{k,k'}$ in \eqref{eq:Y in e} are given. 
%

\subsection{Numerical example on $\G_{2,4}$}
By Theorem~\ref{thm:G24}, we can calculate the coefficients of the kernel expansion
\begin{equation*}
K_{\beta_{16,\sqrt{2}}}(x,y) = \sum_{\lambda_1\geq\lambda_2\geq 0} a_\lambda \sum_{l=1}^{\dim(\mathcal{H}^{d}_{2\lambda})} \varphi_{\lambda,l}(x)\overline{\varphi_{\lambda,l}(y)},\quad x,y\in\mathcal{G}_{2,4}.
\end{equation*}
The eigenfunctions $\varphi_{\lambda,l}$ are given by the tensor products of spherical harmonics in \eqref{eq:tensor struct}, cf.~Theorem~\ref{th:Hpi}. For $\supp(\mu),\supp(\nu_n)\subset\G_{2,4}$, the $L_2$-discrepancy \eqref{eq:fund eq min} of the kernel $K_{\beta_{16,\sqrt{2}}}|_{\G_{2,4}\times\G_{2,4}}$ is 
\begin{equation}\label{eq:series truncating}
\mathscr{D}_{\beta_{16,\sqrt{2}}}(\mu,\nu_n) = \sum_{\lambda_1\geq\lambda_2\geq 0} a_\lambda \sum_{l=1}^{\dim(\mathcal{H}^{d}_{2\lambda})}  \left|\hat{\mu}_{\lambda,l}- \frac{1}{n}\sum_{j=1}^n \overline{\varphi_{\lambda,l}(x_j)}\right|^2,
\end{equation}
where $\hat{\mu}_{\lambda,l}$ is the Fourier coefficient of $\mu$ with respect to $\varphi_{\lambda,l}$, cf.~\eqref{eq:fund eq min}.  

Let us consider $\mu=\sigma_{\G_{2,4}}$. 
According to \cite{Brandolini:2014oz} (see also \cite{Breger:2016vn,Breger:2016rc}), 
the lower bound 
\begin{equation*}
n^{-5/4}\lesssim \mathscr{D}_{\beta_{16,\sqrt{2}}}(\mu,\nu_n)
\end{equation*}
holds for all $n$-point sets $\{x_1,\ldots,x_n\}\subset\G_{2,4}$. We truncate the series \eqref{eq:series truncating} and let $\nu^{M}_n=\frac{1}{n}\sum_{j=1}^n\delta_{x^M_j}$ denote a minimizer of 
\begin{equation}\label{eq:min to zero}
\sum_{\lambda+\lambda_2\leq M} a_\lambda \sum_{l=1}^{\dim(\mathcal{H}^{d}_{2\lambda})}   \left|\hat{\mu}_{\lambda,l} - \frac{1}{n}\sum_{j=1}^n \overline{\varphi_{\lambda,l}(x_j)}\right|^2
\end{equation}
among all $n$-point sets $\{x_1,\ldots,x_n\}\subset\G_{2,4}$. A suitable choice $n\sim M^4 $ leads to the optimal rate 
\begin{equation}\label{eq:123454321}
\mathscr{D}_{\beta_{16,\sqrt{2}}}(\mu,\nu^{M}_n)\sim  
n^{-5/4},
\end{equation}
cf.~\cite{Brandolini:2014oz,Etayo:2016qq}. 
Note that we can efficiently solve  the least squares minimization \eqref{eq:min to zero} by using the nonequispaced fast Fourier transform on $\G_{2,4}$ derived from the nonequispaced fast Fourier transform on $\S^2\times\S^2$ of Section~\ref{sec:FFT} and applying Theorem~ \ref{th:Hpi}. Figure~\ref{fig:Grass} shows logarithmic plots of the number of points versus the $L_2$-discrepancy. We observe a line with slope $-5/4$ as predicted by \eqref{eq:123454321}.

\begin{figure}
\includegraphics[width=.5\textwidth]{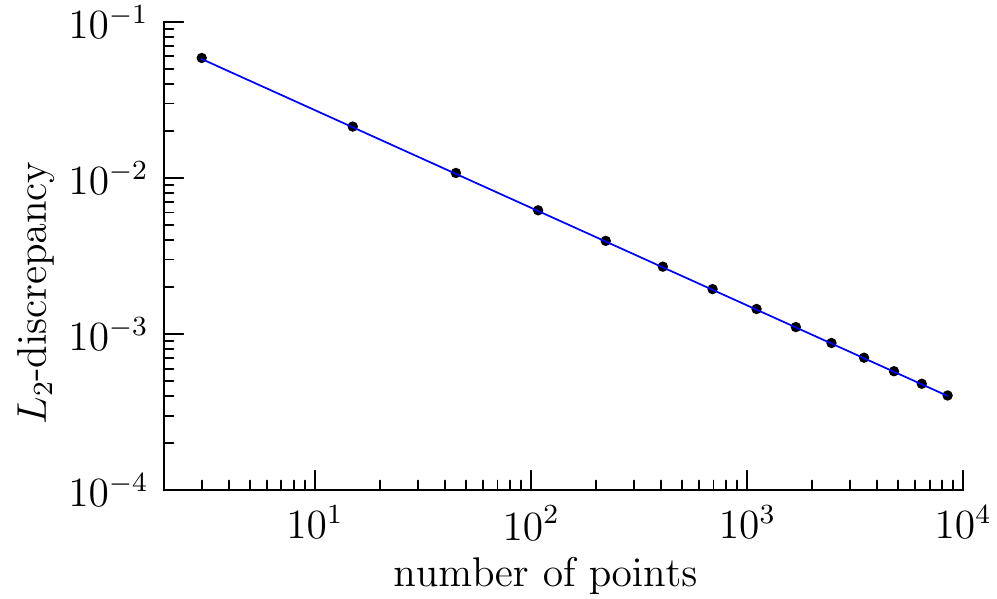}
\caption{Logarithmic plot of the number of points $n$ versus the $L_2$-discrepancy $\mathscr{D}_{\beta_{16},\sqrt{2}}(\mu,\nu^{M}_n)$ on $\G_{2,4}$, where $\nu^{M}_n$ is derived from numerical minimization. }\label{fig:Grass}
\end{figure}

\section*{Acknowledgement}
ME and MG have been funded by the Vienna Science and Technology Fund (WWTF) through project VRG12-009. The research of CK has been partially supported by the Austrian Science Foundation (FWF) 
grant SFB~F50 in the framework of the
Special Research Program
``Algorithmic and Enumerative Combinatorics". ME~would like to thank Christine Bachoc, Karlheinz Gr\"ochenig, Andreas Klotz, and Gerald Teschl for helpful discussions.

\appendix
\section{Proofs for Section \ref{sec:intro 2}}
\label{sec:Askey}
\begin{proof}[Proof of Theorem~\ref{thm:Askey}]\label{proof:Askey}
Let $h_d:[0,\infty)\rightarrow\R$ denote Euclid's hat function given by 
\begin{equation*}
h_d(\|x\|)=\frac{1}{\vol(\mathbb{B}^d_{1/2})}(1_{\mathbb{B}^d_{1/2}} * 1_{\mathbb{B}^d_{1/2}})(x),\quad  x\in\R^d,
\end{equation*}
where $1_{\mathbb{B}^d_{1/2}}$ is the indicator function of $\mathbb{B}^d_{1/2}$. 
For $x,y\in\R^d$ and $t=\|x-y\|$, we derive
\begin{align*}
K_{d}(x,y) & = \int_0^\infty \frac{1}{\vol(\mathbb{B}^d_{r/2})} \int_{\R^d} 1_{\mathbb{B}^d_{r/2}(z)}(x)1_{\mathbb{B}^d_{r/2}(z)}(y) {\rm d} z \,{\rm d}G_d(r)  =\int_t^{\infty} h_d(t/r){\rm d}G_d(r)\\
& = \int_t^{1} h_d(t/r)G_d'(r) {\rm d}r -h_d(t)G_d(1) =t\int_t^{1} r^{-2}h_d'(t/r)G_d(r) {\rm d}r,
\intertext{where the last equality is due to partial integration. 
An explicit expression 
for $h'_d$ is stated in \cite[Equation~(11)]{Gneiting:1999ay}, so that we obtain}
K_{d}(x,y) %& =t\int_t^{1} r^{-2}h_d'(t/r)G_d(r) {\rm d}r\\
& =\frac{d\,\Gamma(\frac{d}{2})}{\sqrt{\pi}\Gamma(\frac{d+1}{2})} t \int_1^{\frac{1}{t}} (1-r^2t^2)^{\frac{d-1}{2}} \pFq{2}{1}{-\frac{d+1}{4},-\frac{d-1}{4}}{-\frac{d}{2}}{r^{-2}} {\rm d}r \\
& = \frac{d\,\Gamma(\frac{d}{2})}{\sqrt{\pi}\Gamma(\frac{d+1}{2})} t  \sum_{j=0}^{\frac{d-1}{2}}(-1)^j\binom{\frac{d-1}{2}}{j}t^{2j} \sum_{l=0}^{\frac{d+1}{4}} \frac{(-\frac{d+1}{4})_l (-\frac{d-1}{4})_l}{(-\frac{d}{2})_l l!}  \int_1^{\frac{1}{t}} r^{2(j-l)}  {\rm d}r \\
%& = \frac{d\,\Gamma(\frac{d}{2})}{\sqrt{\pi}\Gamma(\frac{d+1}{2})} \sum_{j=0}^{\frac{d-1}{2}}(-1)^j\binom{\frac{d-1}{2}}{j} \sum_{l} \frac{(-\frac{d+1}{4})_l (-\frac{d-1}{4})_l}{(-\frac{d}{2})_l}  \frac{t^{2l}}{2(j-l)+1}-\\
%& \qquad- \frac{d\,\Gamma(\frac{d}{2})}{\sqrt{\pi}\Gamma(\frac{d+1}{2})} \sum_{j=0}^{\frac{d-1}{2}}(-1)^j\binom{\frac{d-1}{2}}{j}t^{2j+1} \sum_{l} \frac{(-\frac{d+1}{4})_l (-\frac{d-1}{4})_l}{(-\frac{d}{2})_l}  \frac{1}{2(j-l)+1}\\
& = \frac{d\,\Gamma(\frac{d}{2})}{\sqrt{\pi}\Gamma(\frac{d+1}{2})} \sum_{l=0}^{\frac{d+1}{4}} t^{2l} \frac{(-\frac{d+1}{4})_l (-\frac{d-1}{4})_l}{(-\frac{d}{2})_l l!}  \sum_{j=0}^{\frac{d-1}{2}} \frac{(-1)^j\binom{\frac{d-1}{2}}{j} }{2j-2l+1}\\
& \qquad\qquad- \frac{d\,\Gamma(\frac{d}{2})}{\sqrt{\pi}\Gamma(\frac{d+1}{2})} \sum_{j=0}^{\frac{d-1}{2}}(-1)^j\binom{\frac{d-1}{2}}{j}t^{2j+1} \sum_{l=0}^{\frac{d+1}{4}} \frac{(-\frac{d+1}{4})_l (-\frac{d-1}{4})_l}{(-\frac{d}{2})_l l! (2j-2l+1)}  .
\end{align*}
We now compare coefficients of powers of~$t$ with those of the polynomial $(1-t)^{\frac{d+1}{2}}$. In order to check the coefficient of $t^{2l}$, we first observe 
\begin{align*}
\sum_{j=0}^{\frac{d-1}{2}}  \frac{(-1)^j\binom{\frac{d-1}{2}}{j}}{2j-2l+1} =
\pFq{2}{1}{-\frac{d-1}{2},\frac{1}{2}-l}{\frac{3}{2}-l}{1}\frac{1}{1-2l}
%& = \frac{(1)_{\frac{d-1}{2}}}{(\frac{3}{2}-l)_{\frac{d-1}{2}}}\frac{1}{1-2l}\\
%& = \frac{\Gamma(\frac{d+1}{2})}{\Gamma(\frac{d}{2}-l+1)}\frac{\Gamma(\frac{3}{2}-l)}{1-2l}\\
=\frac{\Gamma(\frac{d+1}{2})  \Gamma(\frac{1}{2}-l)}{2\Gamma(\frac{d}{2}-l+1)},
%\sum_{j=0}^{\frac{d-1}{2}} \frac{(-\frac{d-1}{2})_j (-\frac{1}{2}-l)_j}{(\frac{1}{2}-l)_jj!}\frac{1}{1+2l}
\end{align*}
where the last equality makes use of Gauss' Theorem for the hypergeometric series
evaluated at~1 (cf.\ \cite[Equation~(1.7.6); Appendix~(III.3)]{SlatAC}). 
Direct computation yields
\begin{equation*}
%\sum_{j=0}^{\frac{d-1}{2}}  \frac{(-1)^j\binom{\frac{d-1}{2}}{j}}{2j-2l+1} =  
%%\frac{(-1)^{\frac{d+1}{2}}\Gamma(\frac{d+1}{2})2^{\frac{d-1}{2}}}{(2l-1)(2l-3)\cdots(2l-d)}=
%%\frac{\Gamma(\frac{d+1}{2})}{2(\frac{1}{2}-l)_{\frac{d+1}{2}}}
%\frac{\Gamma(\frac{d+1}{2})  \Gamma(\frac{1}{2}-l)}{2\Gamma(\frac{d}{2}-l+1)}
%,\qquad 
\frac{(-\frac{d+1}{4})_l (-\frac{d-1}{4})_l}{(-\frac{d}{2})_l}
=
%\frac{(-1)^l}{2^{2l}}\frac{(\frac{d+1}{2})(\frac{d+1}{2}-1)\cdots(\frac{d+1}{2}-2l+1)}{(\frac{d}{2})(\frac{d}{2}-1)\cdots(\frac{d}{2}-(l-1))}=
\binom{\frac{d+1}{2}}{2l}\frac{(2l)!}{(-4)^{l}} \frac{\Gamma(\frac{d}{2}-l+1)}{\Gamma(\frac{d}{2}+1)},
\end{equation*} 
so that the coefficient of $t^{2l}$ in $K_{d}(x,y) $ is 
\begin{equation*}
\binom{\frac{d+1}{2}}{2l} \frac{(2l)!}{(-4)^l l!} \frac{ \Gamma(\frac{1}{2}-l) }{\Gamma(\frac{d}{2}+1)  2} \frac{d\,\Gamma(\frac{d}{2})}{\sqrt{\pi}} = \binom{\frac{d+1}{2}}{2l} \frac{(2l)!}{(-4)^{l} l!} \frac{\Gamma(\frac{1}{2}-l) }{ \sqrt{\pi}}
%&=\binom{\frac{d+1}{2}}{2l}\frac{(2l)!}{\sqrt{\pi}(-4)^{l} l!}\Gamma(\frac{1}{2}-l)
=\binom{\frac{d+1}{2}}{2l}.
\end{equation*}
Hence, the coefficients for even powers of $t$ match. To check the coefficient of $t^{2j+1}$, we first assume $\frac{d-1}{4}\in \N$. 
The Pfaff--Saalsch\"utz Theorem (cf.\ \cite[Equation~(2.3.1.3); Appendix~(III.2)]{SlatAC}) yields
\begin{align*}
\sum_{l=0}^{\frac{d+1}{4}} \frac{(-\frac{d+1}{4})_l (-\frac{d-1}{4})_l}{(-\frac{d}{2})_l l! (2j-2l+1)}& = \pFq{3}{2}{-\frac{d+1}{4},-\frac{d-1}{4},-\frac{1}{2}-j}{-\frac{d}{2},\frac{1}{2}-j}{1}\frac{1}{1+2j} \\
&=
\frac{(-\frac{d-1}{4})_{\frac{d-1}{4}}(j-\frac{d-1}{2})_{\frac{d-1}{4}} }{(-\frac{d}{2})_{\frac{d-1}{4}}(j-\frac{d-3}{4})_{\frac{d-1}{4}}}\frac{1}{1+2j}.
%& = (-1)^{\frac{d-1}{4}}(\frac{d-1}{4})! \frac{\Gamma(-\frac{d-1}{4}+j)\Gamma(-\frac{d-1}{4}+\frac{1}{2}+j)\Gamma(-\frac{d}{2})}{\Gamma(-\frac{d-1}{2}+j)\Gamma(\frac{1}{2}+j)\Gamma(-\frac{d+1}{4})} 
\end{align*}
Thus, the coefficient of $t^{2j+1}$ in $K_{d}(x,y)$ is nonzero if and only if $j\leq\frac{d-1}{4}$. Moreover, it is given by 
\begin{align*}
\frac{-d\,\Gamma(\frac{d}{2})(-1)^j\binom{\frac{d-1}{2}}{j}(-\frac{d-1}{4})_{\frac{d-1}{4}}(j-\frac{d-1}{2})_{\frac{d-1}{4}} }
{\sqrt{\pi}\Gamma(\frac{d+1}{2})(1+2j)(-\frac{d}{2})_{\frac{d-1}{4}}(j-\frac{d-3}{4})_{\frac{d-1}{4}}}
=\frac{-\frac{d}{\sqrt{\pi}}\,\Gamma(\frac{d}{2})(-1)^j (\frac{d-1}{4})!(\frac{d+3}{4}-j)_{\frac{d-1}{4}} }
{(\frac{d-1}{2}-j)! j!(1+2j)(\frac{d+1}{4}+1)_{\frac{d-1}{4}}(\frac{1}{2}-j)_{\frac{d-1}{4}}}.
%- \frac{d\,\Gamma(\frac{d}{2})}{\sqrt{\pi}\Gamma(\frac{d+1}{2})}\frac{(-1)^j (\frac{d-1}{2})!}{(\frac{d-1}{2}-j)! j!(1+2j)}\frac{(\frac{d-1}{4})!(\frac{d-1}{4}+1-j)_{\frac{d-1}{4}} }{(\frac{d+1}{4}+1)_{\frac{d-1}{4}}(\frac{1}{2}-j)_{\frac{d-1}{4}}} \\
%=- \frac{d\,\Gamma(\frac{d}{2})}{\sqrt{\pi}\Gamma(\frac{d+1}{2})}\frac{(-1)^j (\frac{d-1}{2})!}{(\frac{d-1}{2}-j)! j!(1+2j)}\frac{(\frac{d-1}{4})!\Gamma(\frac{d-1}{2}+1-j)\Gamma(\frac{d+1}{4}+1)}{\Gamma(\frac{d-1}{4}+1-j)\Gamma(\frac{d}{2}+1)(\frac{1}{2}-j)_{\frac{d-1}{4}}} \\
%=- \frac{2(-1)^j }{ (1+2j)}\frac{(\frac{d-1}{4})!\Gamma(\frac{d+1}{4}+1)(-4)^j }{(2j)!\Gamma(\frac{d-1}{4}+1-j)\Gamma(\frac{d+1}{4}-j)} \\
%=- \frac{2(-1)^j \Gamma(\frac{d}{4}+\frac{3}{4})\Gamma(\frac{d}{4}+\frac{3}{4}+\frac{2}{4})(-4)^j }{(2j+1)!\Gamma(\frac{d-1}{4}+1-j)\Gamma(\frac{d+1}{4}-j)} \\
%=- \frac{\Gamma(\frac{d+3}{2})  }{(2j+1)!\Gamma(\frac{d+1}{2}-2j)} \\
%= - \frac{(\frac{d+1}{2})!}{(2j+1)! (\frac{d+1}{2}-2j-1)!}\\
%=-\binom{\frac{d+1}{2}}{2j+1}
%,
\end{align*}
Further computations using the duplication formula eventually lead to $-\binom{\frac{d+1}{2}}{2l+1}$. The case $\frac{d+1}{4}\in \N$ is checked analogously, and we conclude the proof. 
%\begin{align*}
%\sum_{l=0}^{\frac{d+1}{4}} \frac{(-\frac{d+1}{4})_l (-\frac{d-1}{4})_l}{(-\frac{d}{2})_l l! (2j-2l+1)} = \frac{\pFq{3}{2}{-\frac{d+1}{4},-\frac{d-1}{4},-\frac{1}{2}-j}{-\frac{d}{2},\frac{1}{2}-j}{1}}{1+2j}=\begin{cases}
%\frac{(-\frac{d-1}{4})_{\frac{d-1}{4}}(-\frac{d-1}{2}+j)_{\frac{d-1}{4}} }{(-\frac{d}{2})_{\frac{d-1}{4}}(-\frac{d-1}{4}+\frac{1}{2}+j)_{\frac{d-1}{4}}},& \frac{d-1}{4}\in \N\\
%-\frac{(-\frac{d+1}{4})_{\frac{d+1}{4}}(-\frac{d-1}{2}+j)_{\frac{d+1}{4}}}{(-\frac{d}{2})_{\frac{d+1}{4}}(\frac{1}{2}-j)_{\frac{d+1}{4}}},& \frac{d+1}{4}\in \N
%\end{cases}.
%\end{align*}
%Analogous to the above computations, one now checks that the coefficient of $t^{2l+1}$ is $-\binom{\frac{d+1}{2}}{2l+1}$.  
%One checks
%\begin{equation*}
%\frac{d\,\Gamma(\frac{d}{2})}{\sqrt{\pi}\Gamma(\frac{d+1}{2})}\sum_{l=0}^{\frac{d+1}{4}} \frac{(-\frac{d+1}{4})_l (-\frac{d-1}{4})_l}{(-\frac{d}{2})_l l! (2j-2l+1)} = \frac{\frac{d+1}{2}}{2^{\frac{2n - 1 + (-1)^n}{4}}}\frac{(j-3)(j-4)\cdots (j-\frac{d-1}{2})}{(2j+1)(2j-1)\cdots(2j-\frac{d-1}{2})}
%\end{equation*}
\end{proof}
\section{Proofs for Section \ref{sec:sphere}}
\label{sec:A3}
\begin{proof}[Proof of Proposition~\ref{prop:010}]\label{proof:last}
By expressing the $\R^d$-Fourier transform $\widehat{1_{\mathbb{B}^d_{r}}}(\xi)$ in terms of the Bessel function of the first kind of order $d/2$ and using its asymptotics, we deduce
%At least up to a constant factor, the Fourier transform $\widehat{1_{\mathbb{B}^d_{r}}}(\xi)$ equals $\|\xi\|^{-\frac{d}{2}}J_{d/2}(\|\xi\|)$, where $J_{d/2}$ denotes the Bessel function of the first kind of order $d/2$. Its asymptotics yield 
\begin{equation*}
\left|\widehat{1_{\mathbb{B}^d_{r}}}(\xi)\right| \lesssim\left(1+\|\xi\|^2\right)^{-(d+1)/4}, \quad\xi\in\R^d.
\end{equation*}
Due to the 
zeros of the Bessel function, the respective lower bound cannot hold. This implies the embedding claims for $\mathbb{H}^{\frac{d+1}{2}}(\R^d)$,  in particular, for $d=3$. 

To address the restricted kernel, we observe that the value $K_{d,r}(x,y)$ coincides with the $\R^d$-volume of the two intersecting balls $\mathbb{B}^d_{r}(x)\cap \mathbb{B}^d_{r}(y)$. This volume has been explicitly computed in \cite[Section~2.4.3]{Graf:2013zl}. For $d=3$, we obtain
\begin{equation*}
K_{3,r}(x,y)=\pi(2r-\|x-y\|)^2(\|x-y\|+4r)/12,\quad x,y\in\S^2,
\end{equation*}
so that $K_{3,r}|_{\S^{2}\times\S^{2}}$ is a polynomial of degree $3$ in $\|x-y\|$. Its Fourier coefficients $(a_m)_{m\in\N}$ are linear combinations of the Fourier coefficients of the monomial terms, so that \eqref{eq:as sphe} implies $a_m\lesssim m^{-3}$. 

We have checked that there are no cancelations in these linear combinations. Therefore, the asymptotics \eqref{eq:as sphe} also imply the associated bound from below, which leads to $a_m \sim m^{-3}$. Thus, $K_{3,r}|_{\S^{2}\times\S^{2}}$ reproduces the Sobolev space $\mathbb{H}^{\frac{3}{2}}(\S^{2})$ with equivalent norms. 
\end{proof}

\section{Proofs for Section \ref{sec:b}}
\label{sec:A4}
\begin{proof}[Proof of Proposition~\ref{prop:interval}]\label{proof:interval}
Let us consider $s=1$. The general case follows from rescaling. In order to determine the Fourier expansion of $K_{\beta_{1,1}}|_{[-1,1]\times[-1,1]}$, we need to determine the eigenfunctions with respect to the positive eigenvalues of the integral operator
\begin{equation*}
T : L_2([-1,1])\rightarrow \mathscr{C}([-1,1]),\quad f\mapsto \int_{-1}^{1} f(t)\left(1-\tfrac{1}{2}|\cdot-t|\right) {\rm d}t.
\end{equation*}
For $\lambda>0$, the equation $T\phi=\lambda\phi$ is equivalent to the associated Sturm--Liouville eigenvalue problem. Indeed, 
differentiating twice on both sides and carrying out a short calculation,
we arrive at the second order homogeneous differential equation 
\begin{equation}\label{eq:diff 2}
\phi''(x) +\frac{1}{\lambda} \phi(x)=0,\quad x\in(-1,1),
\end{equation}
with boundary conditions 
$ %\begin{equation}\label{eq:side constraints}
\phi'(1) = - \phi'(-1)$ and 
%\qquad  
$\phi(1) + \phi(-1) = -2\phi'(1). 
$ %\end{equation}
The general solution of \eqref{eq:diff 2} is 
$ %\begin{equation*}
\phi(x)=c_1\cos(\frac{x}{\sqrt{\lambda}})+c_2\sin(\frac{x}{\sqrt{\lambda}}).
$ %\end{equation*}
Direct calculations using the boundary conditions determine $c_1,c_2$ and $\lambda$, so that Mercer's Theorem and normalization of the eigenfunctions provide the claimed Fourier expansion of the kernel. 

Computations analogous to \cite[Section~9.5.5]{Nowak:2010rr} and \cite{Dick:2014bj} verify the claimed form of the reproducing kernel Hilbert space.
\end{proof}

%\begin{proof}[Proof of Theorem \ref{thm:compact set}]\label{proof:norm}
%For $x,y\in \mathbb{B}^d_s$, interchanging the order of integration leads to 
%\begin{align*}
%K_{\beta_{d,s}}(x,y) 
%&= \int_{\S^{d-1}}\int_{-{s}}^{\min\{\langle z,x\rangle,\langle z,y\rangle \}} {\rm d} r \;{\rm d} \sigma_{d-1}(z)
%& ={s}  + \int_{\S^{d-1}} \min\{\langle z,x\rangle,\langle z,y\rangle \}{\rm d} \sigma_{d-1}(z).
%\end{align*}
%This extends the respective computations in \cite{Brauchart:2013pt} from $x,y\in\S^{d-1}$ to $\mathbb{B}^d_s$. The remaining proof is now analogous to \cite{Brauchart:2013pt} and omitted. 
%\end{proof}

\section{Proofs for Section \ref{sec:ball}}\label{sec:app sec ball}
The two linearly independent eigenfunctions of the differential operator 
\begin{equation}\label{eq:DM}
D_m := \partial^{2}_r + \tfrac{d-1}{r}\partial_r - \tfrac{m(m+d-2)}{r^{2}}
\end{equation}
on $[0,1]$ with respect to a possibly complex eigenvalue $-\omega^2$ are
\begin{equation*}
\mathcal{J}^{d,\omega}_{m}(r):=\frac{J_{m+\frac{d}{2}-1}(\omega r )}{r^{\frac{d}{2}-1}},\quad\text{and }\quad\mathcal{Y}^{d,\omega}_{m}(r):= \frac{Y_{m+\frac{d}{2}-1}(\omega r )}{r^{\frac{d}{2}-1}},
\end{equation*}
where $J_\nu$ and $Y_\nu$ are the Bessel functions of first and second type, respectively. 
\begin{lemma}\label{lemma:helping is good}
For odd $d\geq 3$, odd $p>1-d$, and $m\in\N$, any eigenfunction of $T_m^{d,p}$ with eigenvalue $\lambda\neq 0$ is a linear combination of $\left\{\mathcal{J}^{d,\omega_{\ell}}_{m}: \ell=1,\ldots,\frac{d+p}{2}\right\}$, where
\begin{equation}
\label{eq:general solutions Dn2lin}
  \omega_{\ell} =   |v|^{\frac{1}{d+p}}
  \begin{cases}
    {\rm e}^{\pi{\rm i} \frac{2\ell}{d+p}}, & (-1)^{\frac{d+p}{2}}v > 0,\\
    {\rm e}^{\pi{\rm i}\frac{2\ell+1}{d+p}}, & (-1)^{\frac{d+p}{2}}v < 0,    
  \end{cases}
%\mathcal{J}^{d,\omega_l}_{m},\qquad \ell=0,\ldots,\frac{d+p}{2}-1,
\end{equation}
and
\[
   v=-\lambda^{-1} 2^{d+p-2}(d+2m-2)\left(2-\tfrac{d}{2}\right)_{\frac{d+p}{2}-1} \left(\tfrac{d+p}{2}-1\right)!.
\]
% where $\omega$
%\begin{equation*}
%\zeta_\ell\lambda^{\frac{1}{d+p}}
%\end{equation*}
%for suitable $\lambda\neq 0$. 
\end{lemma}

\begin{proof}[Proof of Lemma~\ref{lemma:helping is good}]
Up to a constant depending on $d$ and $p$, $K^{d,p}$ is the Green's function of the polyharmonic equation $\Delta^{\frac{d+p}{2}}u=f$ on $\mathbb{B}^d$ with certain nonlocal boundary conditions, cf.~\cite{Kalmenov:2012fq}. In particular and by specifying the constant, one deduces that any eigenfunction of 
\begin{equation}\label{eq:def of Tmxx}
T^{d,p} : L_2(\mathbb{B}^d) \rightarrow \mathcal{C}(\mathbb{B}^d),\quad f\mapsto\int_{\mathbb{B}^d} K^{d,p}(\cdot,y)f(y){\rm d} y
\end{equation}
with eigenvalue $\tilde \lambda\neq 0$ is an eigenfunction of $\Delta^{\frac{d+p}{2}}$ with eigenvalue 
\begin{equation*}
- \tilde \lambda^{-1} 2^{d+p-2}\left(2-\tfrac{d}{2}\right)_{\frac{d+p}{2}-1} \left(\tfrac{d+p}{2}-1\right)! (d-2)\vol(\S^{d-1}).
\end{equation*}

The Laplacian in polar coordinates is $\Delta = \partial^2_r + \frac{d-1}{r}\partial_r + \frac{1}{r^2}\Delta_{\mathbb S^{d-1}}$. The decomposition \eqref{eq:spec total} (see also \eqref{eq:spec_KM}) 
%$\phi^{d,p}_{m,j,l}= \varphi^{d,p}_{m,j}Y^m_l$ 
yields $\Delta^{\frac{d+p}{2}}\varphi^{d,p}_{m,j,l}=\big(D_{m}^{\frac{d+p}{2}} \varphi^{d,p}_{m,j}\big)Y^m_l$ on $\mathring{\mathbb{B}}^d$, where $D_m$ is as in \eqref{eq:DM}. Since $\tilde \lambda = \frac{(d-2)\mathrm{vol}(\mathbb S^{d-1})}{2m+d-2} \lambda$, any eigenfunction of $T^{d,p}_m$ with eigenvalue $\lambda\neq 0$ is an eigenfunction of $D^{\frac{d+p}{2}}_m$ with eigenvalue 
\begin{equation*}
v=-\lambda^{-1} 2^{d+p-2}(d+2m-2)\left(2-\tfrac{d}{2}\right)_{\frac{d+p}{2}-1} \left(\tfrac{d+p}{2}-1\right)!.
\end{equation*}
The linearly independent eigenfunctions of $D_m^{\frac{d+p}{2}}$ with respect to any eigenvalue $v \ne 0$ are 
\begin{equation}\label{eq:general solutions Dn}
\mathcal{J}^{d,\omega_{\ell}}_{m},\quad\text{and }\quad \mathcal{Y}^{d,\omega_{\ell}}_{m},\qquad \ell=1,\dots,\tfrac{d+p}{2},
\end{equation}
where
\[
  (-\omega_\ell^{2})^{\frac{d+p}{2}} = (-1)^{\frac{d+p}{2}} \omega_{\ell}^{d+p} = v, \qquad \ell=1,\dots,\tfrac{d+p}{2},
\]
and we take
\[
  \omega_{\ell} := |v|^{\frac{1}{d+p}}
  \begin{cases}
    {\rm e}^{2\pi{\rm i} \frac{\ell}{d+p}}, & (-1)^{\frac{d+p}{2}}v > 0\\
    {\rm e}^{\pi{\rm i}\frac{2\ell+1}{d+p}}, & (-1)^{\frac{d+p}{2}}v < 0,    
  \end{cases}
  \qquad \ell=1,\dots,\frac{d+p}{2}.
\]
As an eigenfunction of a positive integer power of the Laplacian, $\phi^{d,p}_{m,j,l}$ is real analytic on $\mathbb{B}^d$, cf.~\cite{Aronszajn:1983mb,John:1950ek}. Hence, the radial part $\varphi^{d,p}_{m,j,l}$ must be an analytic function on $[0,1]$ with even or odd parity for $m$ even or odd, respectively, cf.~\cite{Baouendi:1974pr}. 
%However, linear combinations, in which any $\mathcal{Y}^{d,\zeta_\ell\lambda^{\frac{1}{d+p}}}_{m}$ occurs with nonzero coefficient, is neither even nor odd. 
The functions $\mathcal{Y}^{d,\omega_{\ell}}_{m}$ do not have matching parity, which concludes the proof. 
\end{proof}

In order to identify the eigenvalues and the linear combination in Lemma~\ref{lemma:helping is good}, we check how $T_m^{d,p}$ in \eqref{eq:def of Tm} acts on $\mathcal{J}^{d,\omega}_{m}$.
\begin{lemma}\label{lemma:T of I}
For $d$ and $p>1-d$ odd, we have
\begin{multline}\label{eq:T I etc}
(T_m^{d,p} \mathcal{J}^{d,\omega}_{m})(r) =  \omega^{-(d+p)} \,2^{d+p-2}(d+2m-2)\left(-\tfrac p2\right)_{\frac{d+p}{2}-1}\left(\tfrac{d+p}{2}-1\right)! \mathcal{J}^{d,\omega}_{m}(r) \\
-   \frac{\left(-\frac p2\right)_m}{\left(\frac d2 -1\right)_m} \sum_{i=1}^{\frac{d+p}{2}}  2^{i-1}\left(\tfrac{d+p}{2}-i+1\right)_{i-1}  \omega^{-i} \mathcal{J}^{d,\omega}_{m-i}(1) r^m \pFq{2}{1}{i-\tfrac{d+p}{2}, m-\tfrac p2}{\tfrac d2 + m}{r^2}.
\end{multline}
\end{lemma}
\begin{proof}[Proof of Lemma~\ref{lemma:T of I}]
The idea of the proof is to use the series expansion of the Bessel functions, apply the integral operator to each term, and eventually recover the right-hand side of \eqref{eq:T I etc}. We shall provide the skeleton of the proof and omit some lengthy computations.

Let $d \ge 1$, $p \ge 1-d$ and $r\ge s$. If $d+p$ is even, direct computations yield
\begin{equation*}
  \mathcal{K}_m^{d,p}(r,s)  = \frac{(-\frac p2)_{m}}{(\frac d2 - 1)_{m}}
   \sum_{k=0}^{\frac{d+p}{2}-1} ( -1)^k  {\binom {\frac{d+p}{2}-1} k} \frac{(-\frac p2+m)_{k}}{(\frac d2 + m)_{k}} s^{m+2k} r^{p-m-2k}.
\end{equation*}
We obtain that $( T_m^{d,2l-d} \mathcal{J}^{d,\omega}_{m})(r) $ equals
 \begin{align*}
& \frac{(\frac d2 -l)_m}{(\frac d2 -1)_m}
  \!\sum_{k=0}^{l-1} \! \frac{{\binom {l-1} k}(\frac d2-l+m)_{k}}{( -1)^k(\frac d2 + m)_{k}}  
   \! \Bigg[ \! \frac{\int_0^r J_{m+\frac{d}{2}-1}(\omega s ) s^{\frac{d}{2}+m+2k} \mathrm d s}{r^{2l-d-m-2k}}  \!+ \! \frac{\int_r^1 J_{m+\frac{d}{2}-1}(\omega s ) s^{2l-\frac{d}{2}-m-2k}\mathrm ds}{r^{m+2k}} \!\Bigg]\!.
  \end{align*}
    For $\alpha > 0$, $k \in \mathbb N_0$, $\omega \in \mathbb C$ and $r > 0$, 
integration of each term of the power series of the Bessel function eventually yields
\begin{align}
 \int_0^r J_{\alpha}(\omega s) s^{\alpha + 2k +1} \mathrm d s & =  \sum_{i=0}^{k} \frac{2^i (-k)_{i}}{ \omega^{i+1}} \frac{J_{\alpha + i+1}(\omega r)}{r^{i-\alpha -2k-1}},\label{eq:F1}\\
%{\color{red} \int_0^r J_{\alpha}(\omega s) s^{-\alpha + 2k +1} \mathrm d s }&=   \frac{k!\left(\frac \omega 2 \right)^{\alpha-2k-2}}{2\Gamma(\alpha-k)} \,- \,r^{-\alpha+2k+2}\sum_{i=0}^{k} \frac{(-2)^{i} (-k)_{i}}{(\omega r)^{i+1}} J_{\alpha-i-1}(\omega r) .
\int_r^1 J_{\alpha}(\omega s) s^{-\alpha + 2k +1} \mathrm d s &=  \sum_{i=0}^{k} \frac{(-2)^{i} (-k)_{i}}{\omega^{i+1}} \Big(\frac{J_{\alpha-i-1}(\omega r)}{r^{i+\alpha-2k-1}} -J_{\alpha-i-1}(\omega )\Big),
\label{eq:F2}
 \end{align}
  which follows from direct computations and 
   \begin{equation*}
\sum_{j=0}^{\infty} \frac{(-k)_{n+j-1}}{(\alpha+n)_{j+1}} =   \frac{(-k)_{n-1}}{\alpha+n}\pFq{2}{1}{n-k-1,1}{\alpha+n+1}{1}=
\frac{(-k)_{n-1}}{\alpha+k+1},
  \end{equation*}
for $\alpha \in \mathbb R$, $k \in \mathbb N_{0}$, and $n \in \mathbb N$ with $\alpha \not\in \{-(k+1),-k,\dots,-n \}$. 
By applying \eqref{eq:F1} and \eqref{eq:F2}, 
we can express $( T_m^{d,2l-d} \mathcal{J}^{d,\omega}_{m})(r)$ as a sum of Bessel functions of various different orders. %Next, we aim to isolate the orders $m+\frac{d}{2}-1$. 
Straightforward but lengthy computations combined with the identity
\begin{equation}\label{eq:E}
  J_{\alpha}(x) = \sum_{i,k=0}^{l} \frac{(\frac{x}{2})^{2l+1-i}(-1)^{k}(-k)_{i} }{(\alpha-l)_{l+1}(l-k)! k!} \left[  \frac{(\alpha-l)_{k}J_{\alpha+i+1}(x)}{(\alpha+1)_{k}}   
%  \right.\\
%\left.\hspace{3cm} +(-1)^{l+i}
%       \frac{(\alpha-l)_{l-k}}{(\alpha+1)_{l-k}}  J_{\alpha-i-1}(x)\right]
+
       \frac{(\alpha-l)_{l-k} J_{\alpha-i-1}(x)}{(-1)^{l+i}(\alpha+1)_{l-k}} \right]
\end{equation} 
eventually lead to the claimed equality \eqref{eq:T I etc}. 

In order to verify \eqref{eq:E},
according to the definition of the Bessel function $J_\alpha(x)$,
we have to show that the coefficient
of $(x/2)^{\alpha+2m}$ on the right-hand side of \eqref{eq:E} equals
$(-1)^m/(m!\Gamma(m+\alpha+1))$, for $m=0,1,\dots$. Let first $m\ge l$.
Then this coefficient $R(\alpha,m)$ equals
\begin{align}
\label{eq:Besselexpr}
R(\alpha,m)& :=
\sum_{i,k=0}^{l} \frac{(-1)^{k+m-l-1}(-k)_{i}
}{(\alpha+k-l)_{l+1}(l-k)! k!(m-l-1)!\Gamma(m-l+\alpha+i+1)} \\
\notag
&\kern1cm
+
\sum_{i,k=0}^{l} \frac{(-1)^{k+m}
(-k)_i}{(\alpha-k)_{l+1}(l-k)! k!(m-l+i)!\Gamma(m-l+\alpha)} 
\\
\notag
&=
\sum_{k=0}^{l} \frac{(-1)^{k+m-l-1}\pFq{2}{1}{1,-k}{m-l+\alpha+1}{1}
}{(\alpha+k-l)_{l+1}(l-k)! k!(m-l-1)!\Gamma(m-l+\alpha+1)} 
\\
&\kern1cm
+
\sum_{k=0}^{l} \frac{(-1)^{k+m}\pFq{2}{1}{1,-k}{m-l+1}{1}
}{(\alpha-k)_{l+1}(l-k)! k!(m-l)!\Gamma(m-l+\alpha)} .
\notag
\intertext{
Both hypergeometric series can be evaluated by means of Gauss'
Theorem. Thus, we obtain}
%\begin{multline*}
R(\alpha,m)&=\sum_{k=0}^{l} \frac{(-1)^{k+m+l+1}
}{(\alpha+k-l)_{l+1}(l-k)! k!(m-l-1)!\Gamma(m-l+\alpha)(m-l+k+\alpha)} \notag\\
&\kern1cm
+
\sum_{k=0}^{l} \frac{(-1)^{k+m}
}{(\alpha-k)_{l+1}(l-k)! k!(m-l-1)!\Gamma(m-l+\alpha)(m-l+k)}. \notag
%\end{multline*}
\end{align}
Now we reverse the order of summation in the second sum, i.e.,
we replace $k$ by $l-k$ there. Then both sums can be conveniently put
together to yield
\begin{align*}
R(\alpha,m)&=
\sum_{k=0}^{l} 
\frac{(-1)^{k+l+m}(\alpha-l+2k)
}{(\alpha-l+k)_{l+1}(l-k)!
  k!(m-l-1)!\Gamma(m-l+\alpha)(m-l+k+\alpha)(m-k)}\\
  &=\frac{(-1)^{l+m} \pFq{5}{4}{\alpha-l,\frac {\alpha} {2}-\frac {l}
  {2}+1,\alpha-l+m,-m,-l}{\frac {\alpha} {2}-\frac {l}
  {2},1-m,\alpha-l+m+1,\alpha+1}{1}.
}{(\alpha-l+1)_{l}l!
  (m-l-1)!\Gamma(m-l+\alpha)(m-l+\alpha)(m)}
\end{align*}
The hypergeometric function can be evaluated by means of the classical
terminating very-well-poised $_5F_4$-summation (cf.\ 
\cite[Equation~(2.3.4.6); Appendix~(III.13)]{SlatAC}). 
After some simplification one arrives at the desired expression
$(-1)^m/(m!\Gamma(m+\alpha+1))$.

If $m<l$, then the first sum in \eqref{eq:Besselexpr} does not
contribute anything because of the term $(m-l-1)!$ in the denominator.
In the second sum, the summation over~$i$ may be started at $i=l-m$,
which can be evaluated by means of the binomial theorem. The result
is zero except if $k=l$. Again, in the end one obtains 
$(-1)^m/(m!\Gamma(m+\alpha+1))$.
\end{proof}

\begin{proof}[Proof of Theorem~\ref{thm:ball ultimate}] \label{proof:ball ultimate}
We now combine Lemmata~\ref{lemma:T of I} and~\ref{lemma:helping is good}. Let the linear combination
$ %\begin{equation*}
f=\sum_{\ell=1}^{\frac{d+p}{2}} c_\ell \mathcal{J}^{d,\omega_{\ell}}_{m}
$ %\end{equation*}
be an eigenfunction of $T^{d,p}_m$ with eigenvalue $\lambda$ and let $\omega_{\ell}$ be as in \eqref{eq:general solutions Dn2lin}. We obtain for any $\ell = 1,\dots,\frac{d+p}{2}$ that
\[
  \begin{aligned}
    \omega_{\ell}^{d+p} & = (-1)^{\frac{d+p}{2}} v \\
    & = \lambda^{-1} 2^{d+p-2}(d+2m-2)(-1)^{\frac{d+p}{2}-1}(2-\tfrac{d}{2})_{\frac{d+p}{2}-1} (\tfrac{d+p}{2}-1)! \\
    & = \lambda^{-1} 2^{d+p-2}(d+2m-2)(-\tfrac{p}{2})_{\frac{d+p}{2}-1} (\tfrac{d+p}{2}-1)!,
  \end{aligned}
\]
and thus
\begin{equation}\label{eq:Tmeig}
T^{d,p}_m f - \lambda f =  \frac{(-\frac{p}{2})_m}{(\frac d2 -1)_m} F(r)\!^{\top} A(\omega) \,c,
\end{equation}
where $A(\omega)$ is as in \eqref{eq:det formula eigenvalue}, $c=(c_{\ell})_{\ell=1,\dots,\frac{d+p}{2}}$, and 
\begin{equation*}
      F(r)  := \left( 2^{i-1}(\tfrac{d+p}{2}-i+1)_{i-1} |\omega|^{-i} r^{m} \pFq{2}{1}{i-\tfrac{d+p}{2},m-\tfrac{p}{2}}{ \tfrac d2 + m}{r^2}  \right)_{i=1}^{\frac{d+p}{2}}.
\end{equation*}
For $i=1,\dots,\tfrac{d+p}{2}$, the hypergeometric functions in $F$ are polynomials of exact degree $d+p-2i$ and thus linearly independent. Hence, for $c \ne 0$, the right-hand side of \eqref{eq:Tmeig}
  vanishes if and only if $A(\omega)$ is singular and $c$ is in its nullspace. %This concludes the proof.
\end{proof}

%\begin{proof}[Proof of Corollary \ref{cor:example}]
%The formulas are derived from Theorem \ref{thm:ball ultimate}. Since $J_{m-\frac{1}{2}}({\rm i}\omega)\neq 0$, for all $\omega\in\R\setminus\{0\}$, the eigenspaces are $1$-dimensional and \eqref{eq:minor formula2} is not the zero-function.  
%
%%To determine the eigenfunctions of the kernel $C-\|x-y\|_{2}^{p}$, we must replace $T^{d,p}_0$ with 
%%\begin{equation*}
%%\tilde{T}_{0}^{d,p} f := C \int_{0}^{1} f(s)s^{d-1}\mathrm ds - T_{0}^{d,p}f.
%%\end{equation*}
%%The relation $\int_{0}^{1} J_{\frac{d-2}{2}}(\omega s) s^{d-1}\mathrm ds = \omega^{-1} J_{\frac{d}{2}}(\omega)$ eventually leads to 
%%\begin{align*}
%%&  (\tilde{T}_0^{d,p} \mathcal{J}^{d,\omega}_m)(r) =\ldots
%%%  = - \omega^{-(d+p)}2^{d+p-2} (\tfrac{d+p}{2}-1)! (d-2)\big(-\tfrac{p}{2}\big)_{\frac{d+p}{2}-1} \mathcal{J}^{d,\omega}_m(r) \\
%%%   &  +\sum_{i=1}^{l} \;_2F_1\big(i-l, \tfrac d2-l; \tfrac d2; r^2\big)  2^{i-1}(l-i+1)_{i-1}   \Big( \omega^{-i} J_{\frac{d-2}{2}-i}(\omega) + \delta_{i,l} \frac{C}{ 2^{l-1}(l-1)!} \omega^{-1} J_{\frac d2}(\omega) \Big),\\
%%\end{align*}
%%The eigenvalues are then determined as in Theorem~\ref{thm:ball ultimate} by replacing the last line of $A(\omega)$ appropriately.
%%
%%
%%
%\end{proof}
%

\section{Proofs for Section \ref{sec:SO3}}
\label{sec:A6}
\begin{proof}[Proof of Proposition~\ref{prop:SO}]\label{proof:SO3}
We shall derive the coefficients $a_m(p,\SO(3))$ from the family of spherical coefficients $a_m(p,\S^1)$. The half-angle identity $\sin(\frac{t}{2})=\sqrt{\frac{1-\cos(t)}{2}}$, for $t\in[0,\pi]$, implies
\begin{equation}\label{eq:a1}
2^{-\frac{p}{2}}\|x-y\|^p = \sqrt{1-\langle x,y\rangle}^p = 2^{\frac{p}{2}}\sin(\tfrac{s}{2})^p,\qquad x,y\in\S^{d-1},
\end{equation}
where $s=\arccos(\langle x,y\rangle)$. For $d=2$, the addition theorem yields
\begin{equation}\label{eq:a2}
Y^m_1(x)\overline{Y^m_l(y)}+Y^m_2(x)\overline{Y^m_2(y)}=2\mathcal{T}_m(\langle x,y\rangle),\qquad x,y\in\S^1,\quad m=1,2,\ldots,
\end{equation}
where $\mathcal{T}_m$ are the Chebyshev polynomials of the first kind, i.e., $\mathcal{T}_m\big(\cos(s)\big)=\cos(ms)$, for $s\in[0,\pi]$. The relation \eqref{eq:expans sphere uni} for $d=2$ with \eqref{eq:a1} and \eqref{eq:a2} leads to 
\begin{equation}\label{eq:sphere basic oo}
  2^{\frac{p}{2}} \sin(\tfrac{s}{2})^p = a_0(p,\S^1) + 2 \sum_{m=1}^\infty a_m(p,\S^1) \mathcal{T}_m\big(\cos(s)\big),\qquad s\in[0,\pi],\quad p>0.
\end{equation} 

We now switch to $\SO(3)$. For $x,y\in\SO(3)$, the relation 
$$\|x-y\|_{\F}=2^{\frac{1}{2}}\sqrt{3-\trace(x^\top y)},$$ 
the choice $s=\arccos(\frac{\trace(x^\top y)-1}{2})$, and \eqref{eq:sphere basic oo} imply
\begin{align*}
2^{-p} \|x-y\|_{\F}^p  = 2^{-\frac{p}{2}} \sqrt{3-\trace(x^\top y)}^p &= 2^{\frac{p}{2}} \sin(\tfrac{s}{2})^p \\
&  = a_0(p,\S^1) + 2 \sum_{m=1}^\infty a_m(p,\S^1) \mathcal{T}_m\big(\cos(s)\big)\\
& = a_0(p,\S^1) + 2\sum_{m=1}^\infty a_m(p,\S^1) \mathcal{T}_{2m}\big(\cos(\tfrac{s}{2})\big).
\end{align*}
The identity $\mathcal{T}_m = \frac{1}{2}(\mathcal{C}^1_{m}-\mathcal{C}^1_{m-2})$ with $\mathcal{T}_0 = \mathcal{C}^1_0$ and $\mathcal{C}^0_{-1} = 0$, for $m \in \mathbb N$, leads to
\begin{equation}\label{eq:aasdw}
2^{-p} \|x-y\|_{\F}^p  =\sum_{m=0}^\infty (a_{m}(p,\S^1) - a_{m+1}(p,\S^1)) \mathcal{C}^1_{2m}\big(\cos(\tfrac{s}{2})\big).
\end{equation}
By calculating the differences $2^{\frac{p}{2}}(a_{m}(p,\S^1) - a_{m+1}(p,\S^1))$ in \eqref{eq:aasdw} and applying the addition theorem of the Wigner 
$\mathcal{D}$-functions,
\begin{equation*}
  \sum_{k,l=-m}^m \mathcal{D}_{k,l}^m(x)\overline{\mathcal{D}_{k,l}^m(y)} =
  (2m+1)\mathcal{C}^1_{2m}\big(\cos(\tfrac{s}{2})\big),
\end{equation*}
we derive \eqref{eq:into SO}. Analytic continuation and well-known relations for the gamma function cover the remaining values of $p > -3$.

Standard calculations yield $\frac{\Gamma(m-\frac{p}{2})}{\Gamma(m+\frac{p}{2}+2)}=m^{-2-p}(1+o(1))$, which concludes the proof.
\end{proof}

\section{Proofs for Section \ref{sec:d}}
\label{sec:A7}
\subsection{Proofs for Section \ref{sec:G24 a}}
\label{sec:A7.1}
\begin{proof}[Proof of Theorem~\ref{thm:G24}]\label{proof:G1}
Proof of \eqref{eq:a alpha}: According to \cite{Davis:1999dn}, $Q_\lambda$ is explicitly given in terms of Legendre polynomials by 
\begin{equation}\label{eq:K lambda}
Q_\lambda(x,y) =   \frac{c_\lambda}{2} \left(\mathcal{C}^{\frac{1}{2}}_{\lambda_1+\lambda_2}(\xi_{+})\mathcal{C}^{\frac{1}{2}}_{\lambda_1-\lambda_2}(\xi_{-}) + \mathcal{C}^{\frac{1}{2}}_{\lambda_1+\lambda_2}(\xi_{-})\mathcal{C}^{\frac{1}{2}}_{\lambda_1-\lambda_2}(\xi_{+})\right),
\end{equation}
where $c_\lambda:=\dim(\mathcal{H}^4_{2\lambda})=\left(2-\delta_{0,\lambda_2}\right)\left((2\lambda_1+1)^2-4\lambda^2_2\right)$ and $\theta_1$, $\theta_2$ denote the principal angles between $x$ and $y$ and 
\begin{equation}\label{eq:transf}
\xi_{+} = \cos(\theta_1+\theta_2), \quad \xi_{-} = \cos(\theta_1-\theta_2).
\end{equation}
In order to write the integral \eqref{eq:integ to be} in terms of the variables $\xi_\pm$, we first observe
\begin{align*}
  2^{-\frac{p}{2}}\|x-y\|^p_{\F}  = \left(2- \trace(xy)\right)^{\frac{p}{2}} 
  & = \left(2-\cos(\theta_1)^2-\cos(\theta_2)^2\right)^{\frac{p}{2}}\\
  & = \left(1-\tfrac{1}{2}\cos(2\theta_1)-\tfrac{1}{2}\cos(2\theta_2)\right)^{\frac{p}{2}}\\
  & = \left(1-\xi_{+}\xi_{-}\right)^{\frac{p}{2}}.
\end{align*}
We set $-q:=p/2$ as well as $m:=\lambda_1+\lambda_2$ and $n:=\lambda_1-\lambda_2$. According to \cite{Davis:1999dn}, the measure $\mu_{\G_{2,4}}\otimes \mu_{\G_{2,4}}$ in \eqref{eq:integ to be} turns into $d\xi_{+}d\xi_{-}$ for the variables $\xi_\pm$ on $|\xi_{+}| \le \xi_{-} \le 1$, so that we obtain
\begin{align*}
a_\lambda(p,\G_{2,4})  &= \int_{0}^{1} \int_{-\xi_{-}}^{\xi_{-}} (1-\xi_{+}\xi_{-})^{-q}
  \frac{1}{2}\left(\mathcal{C}^{\frac{1}{2}}_n(\xi_{+})\mathcal{C}^{\frac{1}{2}}_m(\xi_{-})+\mathcal{C}^{\frac{1}{2}}_n(\xi_{-})C^{\frac{1}{2}}_m(\xi_{+})\right) d\xi_{+}d\xi_{-}.
  \intertext{Symmetry arguments yield}
a_\lambda(p,\G_{2,4})   
  &=  \frac14 \int_{-1}^{1} \int_{-1}^{1} (1-xy)^{-q}
  \frac{1}{2}\left(\mathcal{C}^{\frac{1}{2}}_n(x)\mathcal{C}^{\frac{1}{2}}_m(y)+\mathcal{C}^{\frac{1}{2}}_n(y)\mathcal{C}^{\frac{1}{2}}_m(x)\right) {\rm d}x {\rm d} y\\
  &=  \frac14 \int_{-1}^{1} \int_{-1}^{1} (1-xy)^{-q}
  \mathcal{C}^{\frac{1}{2}}_n(x)\mathcal{C}^{\frac{1}{2}}_m(y) {\rm d}x {\rm d}y.
    \intertext{The series expansion of $(1-xy)^{-q}$ converges absolutely for $q<2$, and the Legendre polynomials $\mathcal{C}^{\frac{1}{2}}_m$ are orthogonal to the monomials $x^k$ for $k<m$ or $m \not\equiv k \!\!\!\mod 2$. Therefore, the orthogonality relations yield} 
 a_\lambda(p,\G_{2,4})   &= \frac14  \sum_{k=0}^{\infty} \frac{(q)_k}{k!} \int_{-1}^1 \mathcal{C}^{\frac{1}{2}}_m(x)x^k dx \int_{-1}^1  \mathcal{C}^{\frac{1}{2}}_n(y)y^k dy\\
   &= \frac14 \sum_{k=0}^{\infty} \frac{(q)_{m+2k}}{(m+2k)!} \int_{-1}^1 \mathcal{C}^{\frac{1}{2}}_m(x)x^{2k+m} dx \int_{-1}^1 \mathcal{C}^{\frac{1}{2}}_n(y)y^{2k+m} dy\\
   &=\frac14 \sum_{k=0}^{\infty} \frac{(q)_{m+2k}}{(m+2k)!} c_k^m c_{\tfrac{m-n}{2}+k}^n,\qquad\qquad\text{where }   c_k^m := \int_{-1}^1 x^{m+2k} \mathcal{C}^{\frac{1}{2}}_m(x) dx.
\end{align*}
The use of the generating function and further calculations lead to 
\begin{equation}\label{eq:11}
  c_k^m  
  = \frac{m!\,(m+1)_{2k}}{2^{m-1}(\tfrac32)_m\,4^k\, k!\, (m+\tfrac32)_k}.
\end{equation}
Application of \eqref{eq:11} and $(q)_{m+2k}=(q)_m(m+q)_{2k}$ yields
\begin{multline*}
a_\lambda(p,\G_{2,4})  
  =  \frac14 (q)_m\sum_{k=0}^{\infty} \frac{(m+q)_{2k}}{(m+2k)!}
  \cdot\frac{m!\,(m+1)_{2k}}{2^{m-1}\,(\tfrac32)_m\, 4^k\, k!\, (m+\tfrac32)_k}\\
  \cdot\frac{n!\,(n+1)_{m-n+2k}}{2^{n-1}(\tfrac32)_n\, 4^{\frac{m-n}{2}+k}\, (\tfrac{m-n}{2}+k)!\, (n+\tfrac32)_{\tfrac{m-n}{2}+k}}.
\end{multline*}
By reordering and making use of $ n!\,(n+1)_{m-n+2k} = (m+2k)!$, which
cancels the identical term in the denominator, 
we are led to
\begin{align*}
a_\lambda(p,\G_{2,4})  
 & =   \frac{(q)_m\, m!}{4^{m}\,(\tfrac32)_m}\sum_{k=0}^{\infty}\frac{1}{ (m+\tfrac32)_k}  \cdot \frac{(m+1)_{2k}}{4^k}\cdot \frac{(m+q)_{2k}}{4^k}\\
 &\hspace{5cm}  \cdot \frac{1}{(\tfrac{m-n}{2}+k)!}
  \cdot\frac{1}{ (\tfrac32)_n\,(n+\tfrac32)_{\tfrac{m-n}{2}+k}}
  \cdot\frac{1}{k!}\\
  & =   \frac{(q)_m\, m!}{4^{m}\,(\tfrac32)_m\,(\tfrac{m-n}{2})!\,(\tfrac32)_n\,(n+\tfrac32)_{\tfrac{m-n}{2}}}\sum_{k=0}^{\infty}\frac{ 
   (\tfrac{m+1}{2})_{k}\,(\tfrac{m+2}{2})_{k}\,(\tfrac{m+q}{2})_{k}\,(\tfrac{m+q+1}{2})_{k} }{ (m+\tfrac32)_k\,(\tfrac{m-n+2}{2})_k\,(\tfrac{m+n+3}{2})_k} 
  \cdot\frac{1}{k!},
\end{align*}
where the equality in the last line is due to the identities
\begin{align*}
   \frac{(m+1)_{2k}}{4^k}&=(\tfrac{m+1}{2})_{k}\,(\tfrac{m+2}{2})_{k}, &
 \frac{(m+q)_{2k}}{ 4^{k}}& = (\tfrac{m+q}{2})_{k}\,(\tfrac{m+q+1}{2})_{k}  ,\\
   (\tfrac{m-n}{2}+k)!&=(\tfrac{m-n}{2})!\,(\tfrac{m-n+2}{2})_k,&
   (n+\tfrac32)_{\tfrac{m-n}{2}+k}&=(\tfrac{m+n+3}{2})_k\,(n+\tfrac32)_{\tfrac{m-n}{2}}.
\end{align*}
The relation $(\tfrac32)_n (n+\tfrac32)_{\tfrac{m-n}{2}}=(\tfrac32)_{\tfrac{m+n}{2}}$ and 
our choices $m=|\lambda|$, $n=\lambda_1-\lambda_2$  yield $(\tfrac{m-n}{2})!(\tfrac32)_n (n+\tfrac32)_{\tfrac{m-n}{2}}=(\tfrac32)_{\lambda_1}\lambda_2!=(\tfrac32)_\lambda$, so that $q=-p/2$ and the definition of 
the ${}_4F_3$ hypergeometric series conclude the proof of \eqref{eq:a alpha}.

\medskip
Proof of \eqref{eq:dec-G24}: \label{proof:G2}
The proof of the decay property for $a_\la(p,\G_{2,4})$ requires
some preparation and auxiliary results. For $p\not\in 2\N$, direct calculations yield
\begin{equation}
\label{eq:AA}
a_\la(p,\G_{2,4})
=\frac{2^{-\frac{p}{2}-3}}{\Ga\big(-\frac {p} {2}\big)
}
\sum_{k=0}^\infty
\frac {
\Gamma\big(\frac {|\la|} {2}+\frac {1}{2}+k\big)\,
\Gamma\big(\frac {|\la|} {2}+1+k\big)\,
\Gamma\big(\frac {|\la|} {2}+\frac {1} {2}-\frac {p} {4}+k\big)\,
\Gamma\big(\frac {|\la|} {2}-\frac {p} {4}+k\big)
}
{\Gamma\big(|\la|+\frac {3} {2}+k\big)\,
\Gamma\big(\la_1+\frac {3} {2}+k\big)\,
\Gamma\big(\la_2+1+k\big)\,
\Gamma\big(k+1\big)}.
\end{equation}
In the following, we treat the case 
$\la_1=an$ and $\la_2= (1-a)n$
for $n\in\N$ and $n\to\infty$, with a fixed $a\in[\frac{1}{2},1]$, so that $|\la|=n$. Summarizing the proof of \eqref{eq:dec-G24}, we shall first verify that the summand in \eqref{eq:AA},
\begin{equation} \label{eq:AB}
S(n,k):=\frac {
\Gamma\big(\frac {n} {2}+\frac {1} {2}+k\big)\,
\Gamma\big(\frac {n} {2}+1+k\big)\,
\Gamma\big(\frac {n} {2}+\frac {1} {2}-\frac {p} {4}+k\big)\,
\Gamma\big(\frac {n} {2}-\frac {p} {4}+k\big)
}
{
\Gamma\big(n+\frac {3} {2}+k\big)\,
\Gamma\big(an+\frac {3} {2}+k\big)\,
\Gamma\big((1-a)n+1+k\big)\,
\Gamma\big(k+1\big)},
\end{equation}
as a sequence in $k$, is unimodal, i.e., it first increases until it has reached its maximum and then 
decreases, see Lemma~\ref{lem:2}. Second, approximation of $S(n,k)$ for $k\ge\ep  n^2$ with the help of Stirling's formula, where $\ep>0$
but fixed, leads to an asymptotic formula for 
$\sum_{k\ge\ep n^2}S(n,k)$; see Lemma~\ref{lemma:1}. 
Third, we let $\ep\to0$ to obtain the asymptotic behavior of the full sum
$\sum_{k\ge0}S(n,k)$ for $n\to\infty$; see \eqref{eq:AL}.
If the result is substituted in \eqref{eq:AA}, the claimed decay in
\eqref{eq:dec-G24} follows immediately upon observing $\|\la\|^2=(2a^2-2a+1)n^2$.

To start with, we may consider $S(n,k)$ as a function of real~$k$. 
\begin{lemma} \label{lem:1}
Given $n\in\N$, let $k_0=k_0(n)\in (0,\infty)$ be such that $\frac {\partial} {\partial
k}S(n,k_0)=0$. Then 
\begin{equation} \label{eq:AD}
k_0=k_0(n)=\frac {2a^2-2a+1}{p+6}n^2+O(n),
\quad 
\text{for }n\to\infty.
\end{equation}
\end{lemma}
\begin{lemma} \label{lem:2}
If $n$ is large enough, $S(n,k)$ has a unique --- local
and global --- maximum for $k\in[0,\infty)$.
\end{lemma}
\begin{lemma} \label{lemma:1}
Let $\ep>0$ be fixed and let $A:=\frac{2a^2-2a+1}{p+6}$. Then
\begin{equation*}
\sum_{k\ge \ep n^2}S(n,k)=A^{-\frac{p+4}{2}}
n^{-(p+4)}
\big(1+O(\tfrac {1} {n})\big)
\int_{-1+\frac {\ep} {A}}^\infty
(1+x)^{-\frac{p+6}{2}}
\exp\big(-\tfrac {p+6} {2(1+x)}\big){\rm d}x,
\quad \text{for }n\to\infty.
\end{equation*}
\end{lemma}
We postpone the proofs of Lemmata~\ref{lem:1}, \ref{lem:2}, and \ref{lemma:1}, and discuss their consequences first. 
If we let $\ep\to0$ and substitute $t=\frac {p+6} {2(1+x)}$ in the above integral, then the
integral definition of the gamma function yields 
\begin{equation}\label{eq:lower bound}
\int_{-1}^\infty
\left(1+x\right)^{-\frac{p+6}{2}}
\exp\Big(
-\frac {p+6} {2\left(1+x\right)}\Big){\rm d}x
=\Big(\frac {p+6} {2}\Big)^{-\frac{p+6}{2}}\,
\Ga\Big(\frac {p+4} {2}\Big).
\end{equation}
Lemma~\ref{lemma:1} and \eqref{eq:lower bound} provide the asymptotic
lower bound on $\sum_{k\ge 0}S(n,k)$ of the following two-sided claim:
\begin{equation} \label{eq:AL}
\sum_{k\ge 0}S(n,k)=
\Big(\frac {2} {2a^2-2a+1}\Big)^{\frac{p+4}{2}}
\Ga\Big(\frac {p+4} {2}\Big)n^{-p-4}
\left(1+o(1)\right),
\quad \text{for }n\to\infty.
\end{equation}
To verify the asymptotic upper bound on $\sum_{k\ge 0}S(n,k)$ in \eqref{eq:AL}, we observe \begin{equation*}
\sum_{k\ge0}S(n,k)\leq  \sum_{k\ge\ep n^2}S(n,k)
+\ep n^2 S(n,\ep n^2) ,\quad \ep<A,
\end{equation*}
which is due to Lemma~\ref{lem:2}, saying that $S(n,k)$ grows until its maximum at $k=k_0\sim An^2$. 
By Lemma~\ref{lemma:1} and letting $\ep\to0$, we obtain the upper bound in \eqref{eq:AL}. By taking into account the additional factor $2^{-\frac{p}{2}-3}/\Ga\big(-\frac {p} {2}\big)$ in \eqref{eq:AA} and the relation $\|\la\|^2=(2a^2-2a+1)n^2$, we observe that \eqref{eq:AL} provides our claim \eqref{eq:dec-G24}. 

To complete the proof of \eqref{eq:dec-G24} in Theorem~\ref{thm:G24}, it remains to prove Lemmata~\ref{lem:1}, \ref{lem:2}, and \ref{lemma:1}.
\begin{proof}[Proof of Lemma~\ref{lem:1}]
The condition $0=\frac {\partial} {\partial
k}S(n,k_0)$ implies $0=\frac {\partial} {\partial k}\big(\log S(n,k_0)\big)$. By using the digamma function $\psi(z)$, the logarithmic derivative of \eqref{eq:AB}
can be written as
{\small 
\begin{multline*}
 \frac {\partial} {\partial k}\big(\log S(n,k)\big) = \psi\Big(\frac {n} {2}+\frac {1} {2}+k\Big)
+\psi\Big(\frac {n} {2}+1+k\Big)
+\psi\Big(\frac {n} {2}+\frac {1} {2}-\frac {p} {4}+k\Big)
+\psi\Big(\frac {n} {2}-\frac {p} {4}+k\Big)\\
 -\psi\Big(n+\frac {3} {2}+k\Big)
-\psi\Big(an+\frac {3} {2}+k\Big)
-\psi\big((1-a)n+1+k\big)
-\psi(k+1).
\end{multline*}}%
For $k=o(n)$, the above expression is certainly positive for large~$n$,
hence nonzero. If the order of magnitude of $k$ is at least the one of~$n$
(in symbols, $n=O(k)$), then we may estimate the logarithmic derivative by 
{\small
\begin{multline}
 \frac {\partial} {\partial k}\big(\log S(n,k)\big) =
\log\big(\frac {n} {2}+\frac {1} {2}+k\big)
+\log\big(\frac {n} {2}+1+k\big)
+\log\big(\frac {n} {2}+\frac {1} {2}-\frac {p} {4}+k\big)
+\log\big(\frac {n} {2}-\frac {p} {4}+k\big)\\
 -\log\big(n+\frac {3} {2}+k\big)
-\log\big(an+\frac {3} {2}+k\big)
-\log\big((1-a)n+1+k\big)
-\log\big(k+1\big)
 +O\big(n^{-1}\big),\label{eq:ACa}%\nonumber
\end{multline}
}%
\noindent 
where the asymptotics $
\psi(x)=\log x-\frac {1} {2x}+O\big(x^{-2}\big)$, for $x\to\infty$, 
cf.~\cite[Equation~1.18(7)]{ErdeAA}), were used. 
%
%
%
%
%\begin{small}
%\begin{equation}\label{eq:AC}
%\begin{split}
% \frac {\partial} {\partial k}\big(\log S(n,k)\big) &= \psi\Big(\frac {n} {2}+\frac {1} {2}+k\Big)
%+\psi\Big(\frac {n} {2}+1+k\Big)
%+\psi\Big(\frac {n} {2}+\frac {1} {2}-\frac {p} {4}+k\Big)
%+\psi\Big(\frac {n} {2}-\frac {p} {4}+k\Big)\\
%& \; -\psi\Big(n+\frac {3} {2}+k\Big)
%-\psi\Big(an+\frac {3} {2}+k\Big)
%-\psi\big((1-a)n+1+k\big)
%-\psi(k+1),
%\end{split}
%\end{equation}
%\end{small}
%where $\psi(z)$ is the digamma function. The asymptotics $
%\psi(x)=\log x-\frac {1} {2x}+O\big(x^{-2}\big)$, for $x\to\infty$, 
%cf.~\cite[Eq.~1.18(7)]{ErdeAA}), applied to \eqref{eq:AC} yields
%\begin{small}
%\begin{equation}\label{eq:ACa}
%\begin{split}
%\frac {\partial} {\partial k}\big(\log S(n,k)\big)  =\log\big(\frac {n} {2}+\frac {1} {2}+k\big)
%+\log\big(\frac {n} {2}+1+k\big)
%+\log\big(\frac {n} {2}+\frac {1} {2}-\frac {p} {4}+k\big)
%+\log\big(\frac {n} {2}-\frac {p} {4}+k\big)\\
%-\log\big(n+\frac {3} {2}+k\big)
%-\log\big(an+\frac {3} {2}+k\big)
%-\log\big((1-a)n+1+k\big)
%-\log\big(k+1\big)\\+O\big(n^{-1}\big).
%\end{split}
%\end{equation}
%\end{small}
By applying the exponential function 
on both sides of $\frac {\partial} {\partial k}\big(\log S(n,k_0)\big)=0$, 
together with the above estimation for $\frac {\partial} {\partial k}\big(\log S(n,k)\big)$ we obtain 
\begin{equation}\label{eq:mit Q}
Q(n,k_0)=1+O(n^{-1}),
\end{equation}
where 
\begin{equation*}
Q(n,k):=\frac {
\big(\frac {n} {2}+\frac {1} {2}+k\big)
\big(\frac {n} {2}+1+k\big)
\big(\frac {n} {2}+\frac {1} {2}-\frac {p} {4}+k\big)
\big(\frac {n} {2}-\frac {p} {4}+k\big)}
{
\big(n+\frac {3} {2}+k\big)
\big(an+\frac {3} {2}+k\big)
\big((1-a)n+1+k\big)
(k+1)}.
\end{equation*}
For $k\sim cn$, we observe 
\begin{equation*}
\lim_{n\to\infty}Q(n,k)=\frac {\big(c+\frac {1} {2}\big)^4} 
{(c+1)(c+a)(c+1-a)c}
>1,
\end{equation*}
so that we may assume that $k_0$ is of larger asymptotic order of magnitude
than~$n$. 

Define $P(n,k)$ by 
\begin{equation}\label{eq:def of P}
Q(n,k)=1-\frac {P(n,k)} 
{
\big(n+\frac {3} {2}+k\big)
\big(an+\frac {3} {2}+k\big)
\big((1-a)n+1+k\big)
(k+1)}.
\end{equation}
For $k$ of larger asymptotic order of magnitude than~$n$, the asymptotics of the digamma function implies that the error term $O(n^{-1})$ in \eqref{eq:ACa} and \eqref{eq:mit Q} may be replaced by $O(k^{-2})$ and hence by $o(k^{-1})$. The relations \eqref{eq:mit Q} and \eqref{eq:def of P} lead to $P(n,k_0)=o(k_0^3)$. Since direct computations yield
\begin{equation}\label{eq:again boot strap}
%\begin{split}
P(n,k) =\Big(\frac{p+6}{2}\Big) k^3
- \Big(a^2 -a
   +\frac{1}{2}\Big)(k^2n^2+k n^3) -\frac{1}{16} n^4 + \text{ lower order terms},
 %  \end{split}
\end{equation}
asymptotically leading terms in $P(n,k)$ must cancel each other. We have already excluded $k_0\sim cn$, so that we now consider $k_0\sim cn^2$ for an appropriate constant~$c$. The terms $k^3$ and $k^2n^2$ in $P(n,k)$ must cancel each other, so that 
$$
\frac {1} {2}(p+6)c^3-\frac {1} {2}(2a^2-2a+1)c^2=0.
$$
The solution $c=\frac {1} {p+6}(2a^2-2a+1)$ yields the leading term in \eqref{eq:AD}. In order to derive the $O(n)$ term in \eqref{eq:AD}, we have to perform ``boot\-strap'', i.e., we substitute $k_0=cn^2+k_1$ in \eqref{eq:again boot strap} and  apply analogous arguments to eventually conclude 
$k_1=O(n)$.
\end{proof}

\begin{proof}[Proof of Lemma~\ref{lem:2}]
We already saw in the previous proof that $\frac {\partial} {\partial
k}\log S(n,0)>0$, for sufficiently large $n$. Hence, $\frac {\partial} {\partial
k} S(n,0)>0$ holds, so that $S(n,k)$ does not have a local maximum in $k=0$. Convergence of the series \eqref{eq:AA} implies $S(n,k)\rightarrow 0$ for integers $k\rightarrow\infty$. Thus, for sufficiently large $n$, $S(n,k)$ attains a local maximum at some $k_0\in(0,\infty)$.  
Lemma~\ref{lem:1} implies $k_0\sim \frac {1} {p+6}(2a^2-2a+1)n^2$. In order to investigate $S(n,k)$ in a neighborhood of $k_0$, we compute $\frac {\partial^2} {\partial
s^2}\log S(n,k_0+s)$, which is 
\begin{equation}
\begin{split}
\psi^{(1)}\big(\frac {n} {2}+\frac {1} {2}+k\big)
+\psi^{(1)}\big(\frac {n} {2}+1+k\big)
+\psi^{(1)}\big(\frac {n} {2}+\frac {1} {2}-\frac {p} {4}+k\big)
+\psi^{(1)}\big(\frac {n} {2}-\frac {p} {4}+k\big)\\
-\psi^{(1)}\big(n+\frac {3} {2}+k\big)
-\psi^{(1)}\big(an+\frac {3} {2}+k\big)
-\psi^{(1)}\big((1-a)n+1+k\big)
-\psi^{(1)}\big(k+1\big),
\end{split}\label{eq:AF}
\end{equation}
where $k=k_0+s$ and $\psi^{(1)}(x)$ denotes the derivative of $\psi(x)$. We claim that, for $\vert s\vert=o( n^2)$ and sufficiently
large~$n$, we have $\frac {\partial^2} {\partial
s^2}\log S(n,k_0+s)<0$. In order to establish this claim, we make use of $\psi^{(1)}(x)=\frac {1}
{x}+\frac {1} {2x^2}+O(x^{-3})$, for $x\to\infty$, cf.~\cite[Equations~1.16(9) and 1.18(9)]{ErdeAA}), to estimate the individual expressions in \eqref{eq:AF}. 
Using $k=k_0+s=k_0+o(n^2)= k_0+o( k_0)$, we obtain
\begin{align*}
\psi^{(1)}\Big(\frac {n} {2}+\frac {1} {2}+k\Big)
&=\frac {1} {\frac {n} {2}+\frac {1} {2}+k}
+\frac {1} {2\left(\frac {n} {2}+\frac {1}
{2}+k\right)^2}+O\big(k_0^{-3}\big)\\
&=\frac {1} {k\big(1+\frac {\frac {n} {2}+\frac {1} {2}}
k\big)}
+\frac {1} {2 k_0^2}+o\big(k_0^{-2}\big)\\
&=\frac {1} {k}\Big(1-\frac {\frac {n} {2}+\frac {1} {2}} {k}
+\frac {n^2} {4k^2}\Big)
+\frac {1} {2 k_0^2}+o\big(k_0^{-2}\big),
\end{align*}
where we have used $\frac{1}{1+x}=1-x+x^2+O(x^3)$, for $x\rightarrow 0$. 
We treat the other terms in \eqref{eq:AF} analogously. Note that $k=k_0+o(n^2)$ also implies
$ 
k= \frac {1} {p+6}(2a^2-2a+1)n^2+o(n^2)$, 
so that, by putting the individual estimates together, we arrive at
\begin{align*}
\frac {\partial^2} {\partial s^2}\log S(n,k_0+s)&=
\frac {\frac {p} {2}+3} {k^2}
+\frac {n^2} {k^3}\left(-a^2-(1-a)^2\right)+o\big(
k_0^{-2}\big)\\
&=-\frac {\frac {p} {2}+3} {k^2}+o\big(k_0^{-2}\big)
=-\frac {\frac {p} {2}+3} {k_0^2}+o\big(k_0^{-2}\big).
\end{align*}
For sufficiently large $n$ (and hence large~$k_0$), this is
evidently negative, as claimed. Thus, $S(n,k)$ has a strict local maximum in $k_0$. 

Moreover, say there are $ k_0, k'_0\in(0,\infty)$, where $S(n,k)$ has a local maximum, then Lemma~\ref{lem:1} implies that the magnitude of $|k_0- k_0'|$ is of smaller order than $n^2$. The above considerations imply $k_0=k_0'$, which completes the proof.
\end{proof}

To prove Lemma~\ref{lemma:1}, we shall approximate the summand $S(n,k)$, given in \eqref{eq:AB},
with the help of Stirling's formula \eqref{eq:AG}. 
It turns out that, under this approximation, the sum $\sum_{k\ge0}S(n,k)$ can then be interpreted as a Riemann integral.
\begin{proof}[Proof of Lemma~\ref{lemma:1}]
Let $k_0$ be the unique location of the maximum of $S(n,k)$. 
We recall that Lemma~\ref{lem:1} yields $k_0=An^2+O(n)$. We consider $\log S(n,k)$, for $k\ge\ep n^2$, and write $k=k_0+s$, so that $s\ge -k_0+\ep n^2$. 
Stirling's formula 
\begin{equation} \label{eq:AG}
\log\big(\Ga(z)\big)=\Big(z-\frac {1} {2}\Big)\log(z)-z+\frac {1} {2}\log(2\pi)
+O\big(z^{-1}\big)
\end{equation}
leads to the estimates 
\begin{align*}
\log\Big(\Ga\Big(\frac {n} {2}+\frac {1} {2}+k\Big)\Big)
&=
\Big(\frac {n} {2}+k_0+s\Big)
\log\Big(\frac {n} {2}+\frac {1} {2}+k_0+s\Big)\\
&\kern2cm
-\Big(\frac {n} {2}+\frac {1} {2}+k_0+s\Big)+\frac {1} {2}\log(2\pi)
+O\big(n^{-2}\big)\\
&=
\Big(\frac {n} {2}+k_0+s\Big)\left[
\log(k_0+s)
+
\log\Big(1+\frac {\frac {n} {2}+\frac {1} {2}} {k_0+s}\Big)\right]\\
&\kern2cm
-\Big(\frac {n} {2}+\frac {1} {2}+k_0+s\Big)+\frac {1} {2}\log(2\pi)
+O\big(n^{-2}\big),
\intertext{where we have factored out $k_0+s$. By applying $\log(1+x)=x-\frac {x^2} {2}+O(x^3)$, we obtain}
\log\Big(\Ga\Big(\frac {n} {2}+\frac {1} {2}+k\Big)\Big)&=
\Big(\frac {n} {2}+k_0+s\Big)\left[
\log(k_0)
+
\log\Big(1+\frac {s} {k_0}\Big)
+\frac {\frac {n} {2}+\frac {1} {2}} {k_0+s}
-\frac {1} {2}
\Big(\frac {\frac {n} {2}+\frac {1} {2}} {k_0+s}\Big)^2\right]\\
&\kern2cm
-\Big(\frac {n} {2}+\frac {1} {2}+k_0+s\Big)+\frac {1} {2}\log(2\pi)
+O\big(n^{-1}\big).
%Christian: Das musz repariert werden.
\intertext{Here, terms such as $\frac{n^3}{(k_0+s)^2}$ or $\frac {n} {k_0+s}$ 
are of the order $O(n^ {-1})$ and can therefore be subsumed in the error term.
Thus, we obtain}
\log\Big(\Ga\Big(\frac {n} {2}+\frac {1} {2}+k\Big)\Big)&=
\Big(\frac {n} {2}+k_0+s\Big)\left[
\log(k_0)
+
\log\Big(1+\frac {s} {k_0}\Big)\right] +\frac{1}{2}\frac{(\frac {n} {2})^2}{(k_0+s)}\\
&\kern3.7cm
-(k_0+s)+\frac {1} {2}\log(2\pi)
+O\big(n^{-1}\big).
%&=
%\Big(\frac {n} {2}+k_0+s\Big)\left[
%\log(k_0)
%+
%\log\Big(1+\frac {s} {k_0}\Big)\right]+\frac {1} {8A(1+\frac {s} {k_0})}\\
%&\kern3.7cm
%-(k_0+s)+\frac {1} {2}\log(2\pi)
%+O\big(n^{-1}\big).
\end{align*}
The reasoning for the $O(\,.\,)$-terms are based on our
restriction to $k_0+s\ge\ep n^2$. However, the constants in these error
terms do contain~$\ep$. 

For the other gamma functions in \eqref{eq:AB}, we proceed similarly. If everything is put together, then we obtain
\begin{align*}
\log \big(S(n,k)\big)
&=-\frac {p+6} {2}\left[\log(k_0)+\log\Big(1+\frac {s} {k_0}\Big)\right]
-\frac{1}{2}\frac{n^2(2a^2-2a+1)}{(k_0+s)}
+O\big(n^{-1}\big)\\
&=-\frac {p+6} {2}\left[\log(k_0)+\log\Big(1+\frac {s} {k_0}\Big)\right]
-\frac {p+6} {2\big(1+\frac {s} {k_0}\big)}
+O\big(n^{-1}\big).
\end{align*}
For the sum of the $S(n,k)$, we have
\begin{equation}\label{eq:AI}
\sum_{k\ge \ep n^2}S(n,k)
=\!\!\!\sum_{s\ge -k_0+\ep n^2} \!\!\!\!
k_0^{-\frac{p+6}{2}}
\Big(1+\frac {s} {k_0}\Big)^{-\frac{p+6}{2}}
\exp\Big(
\frac {-(p+6)} {2\big(1+\tfrac {s} {k_0}\big)}\Big)
\Big(1+O\big(n^{-1}\big)\Big).
\end{equation}
By understanding, the sum over $s$ is taken over those~$s$, for which
$k_0+s$ is an integer. The error of the Riemann sum approximation 
%We see that
\begin{equation} \label{eq:AJ}
\frac {1} {k_0}\!
\sum_{s\ge-k_0+\ep n^2}
\Big(1+\frac {s} {k_0}\Big)^{-\frac{p+6}{2}}
\exp\Big(
\frac {-(p+6)} {2\big(1+\frac {s} {k_0}\big)}\Big)
\approx \int_{-1+\frac {\ep} {A}}^\infty \!\!\!
(1+x)^{-\frac{p+6}{2}}
\exp\Big(\frac {-(p+6)} {2(1+x)}\Big){\rm d}x
\end{equation}
is of the order of magnitude $O(k_0^{-1})$. This follows from the 
fact that the summand in the sum attains a unique local and global
maximum --- namely at $s=0$ --- and therefore the error is bounded 
above by $k_0^{-1}$ times the absolute variation of the summand --- 
which equals twice the maximum. Substitution of all this in \eqref{eq:AI} concludes the proof. 
\end{proof}
Our proof of \eqref{eq:dec-G24} in Theorem~\ref{thm:G24} is now complete.
\end{proof}

\subsection{Proofs for Section \ref{sec:G24 b}}\label{sec:proofs of sec 7.2}
Recall that the action of $\SO(3)\times\SO(3)$ on $\mathbb{S}^2\times \mathbb{S}^2$ by left multiplication is transitive, and  $\SO(4)$ acts transitively on $\G_{2,4}$ by conjugation. For each $O\in\SO(4)$, there are $a,b\in\mathbb{S}^3$ such that $O=L_a R_b$, where the left and right isoclinic rotations are
\begin{equation*}
L_a :=
\begin{pmatrix}
a_1 & -a_2 & -a_3 & -a_4 \\
a_2 &  a_1 & -a_4 &  a_3 \\
a_3 &  a_4 &  a_1 & -a_2 \\
a_4 & -a_3 &  a_2 &  a_1 \\
\end{pmatrix}^\top, \quad
R_b :=
\begin{pmatrix}
b_1 & -b_2 & -b_3 & -b_4 \\
b_2 &  b_1 &  b_4 & -b_3 \\
b_3 & -b_4 &  b_1 &  b_2 \\
b_4 &  b_3 & -b_2 &  b_1 \\
\end{pmatrix}.
\end{equation*}
For $a\in\mathbb{S}^3$, the Euler--Rodrigues formula yields
\begin{equation*}
S_a:=
\begin{pmatrix}
a_1^2+a_2^2-a_3^2-a_4^2 & 2(a_2a_3-a_1a_4)       & 2(a_2a_4+a_1a_3)       \\
2( a_2a_3+a_1a_4)       & a_1^2-a_2^2+a_3^2-a_4^2 & 2( a_3a_4-a_1a_2)       \\
2( a_2a_4-a_1a_3)       & 2(a_3a_4+a_1a_2)       & a_1^2-a_2^2-a_3^2+a_4^2 \\
\end{pmatrix}^\top \in \SO(3).
\end{equation*}
We are looking for $\mathcal{P}:\S^2\times\S^2\rightarrow\G_{2,4}$ satisfying
\begin{equation}\label{eq:00}
  (L_a R_b)\cdot\mathcal{P}(x,y)\cdot(L_a R_b)^\top
  =\mathcal{P}(S_a x,S_b y),\qquad a,b\in\mathbb{S}^3,\quad x,y\in\mathbb{S}^2.
\end{equation}
%In other words, we aim for $\mathcal{P}$ such that the diagram
%\begin{equation*}
%\begin{CD}
%\mathbb{S}^2\times\mathbb{S}^2 @>\mathcal{P}>> \G_{2,4}\\
%@V (S_a,S_b) VV   @VV L_a R_b V\\
%\mathbb{S}^2\times\mathbb{S}^2 @>\mathcal{P}>> \G_{2,4}
%\end{CD}
%\end{equation*}
%commutes for all $a,b\in\S^3$. 
\begin{thm}\label{th:ide th}
There are exactly two mappings $\S^2\times\S^2\rightarrow\G_{2,4}$ satisfying \eqref{eq:00}. One is $\mathcal{P}$ as in \eqref{eq:ultimate P}, and the other is $I_4-\mathcal{P}$. In particular, $\mathcal{P}$ is surjective and, for all $x,y,u,v\in\mathbb{S}^2$,
\begin{equation}\label{eq:ambigui}
\mathcal{P}(u,v)=\mathcal{P}(x,y)\quad
\text{if and only if} \quad (u,v) \in \{\pm (x,y) \}.
\end{equation}
\end{thm}
\begin{proof}[Proof of Theorem~\ref{th:ide th}]\label{proof:P}
The identity \eqref{eq:00} for the specific choice of $\mathcal{P}$ is verified by expanding both sides of the equality and comparing the polynomial expressions. We omit the straightforward but lengthy computation.

Since the conjugate action of $\SO(4)$ on $\G_{2,4}$ is transitive, the identity \eqref{eq:00} also implies surjectivity.
Since left and right eigenspaces of $\mathcal{L}(\mathcal{P}(x,y))$ in \eqref{eq:inv of P} and \eqref{eq:L L L} are uniquely determined, we deduce that \eqref{eq:ambigui} holds.

Let us now address the uniqueness statement. For $a:=(\alpha,0,0)$, $b:=(\beta,0,0)$ with $\alpha,\beta\in\mathbb{S}^1$, we obtain the isoclinic rotations
\begin{align*}
L_a = \begin{pmatrix}
A_\alpha&0\\
0&A_\alpha
\end{pmatrix},
\qquad
R_b = \begin{pmatrix}
B_\beta&0\\
0&B^\top_\beta
\end{pmatrix},
\end{align*}
where $A_\alpha=\left(\begin{smallmatrix}\alpha_1&-\alpha_2\\
\alpha_2&\alpha_1\end{smallmatrix}\right),\;B_\beta = \left(\begin{smallmatrix}
\beta_1&-\beta_2\\
\beta_2&\beta_1
\end{smallmatrix}\right)\in\SO(2)$. Since $S_ae_1=S_b e_1=e_1$, any mapping $\tilde{\mathcal{P}}:\S^2\times\S^2\rightarrow\G_{2,4}$ satisfying \eqref{eq:00} must obey
\begin{equation}\label{eq:rot cond}
L_a R_b\tilde{\mathcal{P}}(e_1,e_2)(L_a R_b)^\top = \tilde{\mathcal{P}}(e_1,e_2),\quad \alpha,\beta\in\mathbb{S}^1.
\end{equation}
Let $V$ denote the range of $\tilde{\mathcal{P}}(e_1,e_2)$, which is a two-dimensional subspace of $\mathbb{R}^4$. The relation \eqref{eq:rot cond} means that $L_aR_b$ maps $V$ into itself, i.e., $L_aR_b V = V$. For all $U_1,U_2\in\SO(2)$, there are $\alpha,\beta\in\mathbb{S}^1$ such that
\begin{equation*}
L_aR_b = \begin{pmatrix}
A_\alpha B_\beta&0\\
0&A_\alpha B_\beta^\top
\end{pmatrix} =
\begin{pmatrix}
U_1&0\\
0&U_2
\end{pmatrix},
\end{equation*}
so that $V$ must either coincide with $\spann\{e_1,e_2\}$ or with $\spann\{e_3,e_4\}$. Hence, $\tilde{\mathcal{P}}(e_1,e_2)$ must coincide with either $\diag(1,1,0,0)$ or $\diag(0,0,1,1)$. It follows from \eqref{eq:00} and $\SO(3)$ acting transitively on $\mathbb{S}^2$ that in the first case $\tilde{\mathcal{P}}=\mathcal{P}$ and in the latter $\tilde{\mathcal{P}}=I_4-\mathcal{P}$.
\end{proof}
\begin{remark}
The Grassmannian $\G_{2,4}$ has been parametrized in \cite{Davis:1999dn} by means of angles but the explicit identity \eqref{eq:00} that steers our present approach has not been considered there. Our parametrization is compatible with the one used in \cite{Conway:1996aa}.
\end{remark}

The following properties of $\mathcal{P}$ are useful.
\begin{lemma}\label{lemma:222}
The probability measure $\sigma_{\G_{2,4}}$, induced by the Haar measure on $\OOO(4)$, is the push-forward measure of $\sigma_{\mathbb S^2}\otimes \sigma_{\mathbb S^2}$ under $\mathcal{P}$, and we have, for all $x,y,u,v\in\S^2$,
\begin{align}
\langle \mathcal{P}(x,y),\mathcal{P}(u,v)\rangle_{\F} &=1+\langle xy^\top,uv^\top\rangle_{\F}=1+ \langle x,u\rangle  \langle y,v\rangle, \label{eq:skal pro vec}\\
\mathcal{I}_-\cdot\mathcal{P}(x,y)\cdot \mathcal{I}_-^\top & = \mathcal{P}(y,x),\label{eq:22}\\
I_4-\mathcal{P}(x,y) & = \mathcal{P}(-x,y) = \mathcal{P}(x,-y),\label{eq:1100}\\
\{\langle x,u\rangle,\langle y,v\rangle\} &= \pm \{
%\xi_+,\xi_-
\cos(\theta_1+\theta_2),\cos(\theta_1-\theta_2)
\},\label{eq:ambi scal xi}
\end{align}
where $\mathcal{I}_-:=\diag(-1,1,1,1)$, and $\theta_1$, $\theta_2$ denote the principal angles between $\mathcal{P}(x,y)$ and $\mathcal{P}(u,v)$.
%, and $\theta_1$, $\theta_2$ denote the principal angles between $\mathcal{P}(x,y)$ and $\mathcal{P}(u,v)$.
%
%For $\mathcal{P}(x,y)$ and $\mathcal{P}(u,v)$ with principal angles $\theta_1$, $\theta_2$, we observe
%\begin{equation}\label{eq:ambi scal xi}
%\{\langle x,u\rangle,\langle y,v\rangle\} = \pm \{\cos(\theta_1+\theta_2),\cos(\theta_1-\theta_2)\}.
%\end{equation}
\end{lemma}
Note that the right-hand side of \eqref{eq:ambi scal xi} is $\pm \{\xi_+,\xi_-\}$ for $\xi_+,\xi_-$ as in \eqref{eq:transf}. 
\begin{proof}[Proof of Lemma~\ref{lemma:222}]
The product measure $\sigma_{\mathbb S^2}\otimes \sigma_{\mathbb S^2}$ is $\SO(3)\times\SO(3)$ invariant. According to \eqref{eq:00}, the pushforward measure of $\sigma_{\mathbb S^2}\otimes \sigma_{\mathbb S^2}$ on $\mathbb S^2 \times \mathbb S^2$ under $\mathcal{P}$ is $\SO(4)$ invariant, so that the uniqueness of the Haar measure implies the first claim of the lemma.

Each of the remaining claims is first proved for $\mathcal{P}$ as defined in Proposition~\ref{th:ide th} and then argued that it also holds for $I_4-\mathcal{P}$. The identity \eqref{eq:skal pro vec} is easily observed for $u,v=e_1$ first and then \eqref{eq:00} yields the general case. According to
\begin{equation*}
\langle I_4-\mathcal{P}(x,y),I_4-\mathcal{P}(u,v)\rangle_{\F} =\langle \mathcal{P}(x,y),\mathcal{P}(u,v)\rangle_{\F},
\end{equation*}
the identity \eqref{eq:skal pro vec} also holds for $I_4-\mathcal{P}$. The statement \eqref{eq:22} follows from expanding both sides of the equality and comparing polynomial expressions in $x$ and $y$. If it holds for $\mathcal{P}$, then it must also hold for $I_4-\mathcal{P}$. One directly calculates \eqref{eq:1100}. 

In order to check \eqref{eq:ambi scal xi}, we first recognize that principal angles between $I-\mathcal{P}(x,y)$ and $I-\mathcal{P}(u,v)$ coincide with the ones between $\mathcal{P}(x,y)$ and $\mathcal{P}(u,v)$. By $\xi_{\pm}$ as in \eqref{eq:transf} and as at the beginning of the proof of Theorem~\ref{thm:G24}, we observe $\trace(\mathcal{P}(x,y)\mathcal{P}(u,v))=1+\xi_+\xi_-$. Hence, \eqref{eq:skal pro vec} leads to
$ %\begin{equation}\label{eq:product}
\xi_+\xi_- = \langle x,u\rangle \langle y,v\rangle.
$ %\end{equation}
Theorem~\ref{th:Hpi} implies that $Q_\lambda$ in \eqref{eqLrep kernel Q} satisfies
\begin{equation*}
Q_\lambda\left(\mathcal{P}(x,y),\mathcal{P}(u,v)\right) = c_\lambda \left(\mathcal{C}^{\frac{1}{2}}_{\lambda_1+\lambda_2}(\langle x,u\rangle)\mathcal{C}^{\frac{1}{2}}_{\lambda_1-\lambda_2}(\langle y,v\rangle) + \mathcal{C}^{\frac{1}{2}}_{\lambda_1+\lambda_2}(\langle y,v \rangle)\mathcal{C}^{\frac{1}{2}}_{\lambda_1-\lambda_2}(\langle x,u \rangle)\right). 
\end{equation*}
Comparison of this identity with \eqref{eq:K lambda} and few further calculations eventually lead to \eqref{eq:ambi scal xi}. 
%We know from \cite{Davis:1999dn} that \eqref{eq:12345} also holds with $\langle x,u\rangle$ and $\langle y,v\rangle$ replaced with $\xi_+$ and $\xi_-$, respectively. We choose $\lambda=(1,1)$. Since $\mathcal{L}_2$ is an even polynomial, we derive from the respective formulas \eqref{eq:12345},
%\begin{equation*}
%\xi_+^2 + \xi_-^2 = \langle x,u\rangle^2 + \langle y,v\rangle^2.
%\end{equation*}
%Due to \eqref{eq:product}, this implies
%\begin{equation*}
%(\xi_+ + \xi_-)^2 = (\langle x,u\rangle + \langle y,v\rangle)^2.
%\end{equation*}
%The uniqueness of solutions for systems of the type \eqref{eq:uniquness} now yields \eqref{eq:ambi scal xi}.
\end{proof}

\begin{proof}[Proof of Theorem~\ref{th:Hpi}]\label{proof:eigenspaces}
The action of $\SO(3)\times\SO(3)$ leads to the irreducible decomposition 
  \begin{align*}
%  L_2(\mathbb S^2 \times \mathbb S^2) &= \bigoplus_{m,n=0}^{\infty} \mathrm{span}\{ Y_k^m\otimes Y_l^n : k = -m,\dots,m,\; l=-n,\ldots,n\},\\
    L_2\Big(\faktor{\mathbb{S}^2\times \mathbb{S}^2}{\pm 1}\Big) &=\bigoplus_{m+n\in2\mathbb{N}} \mathrm{span}\{ Y_k^m\otimes Y_l^n : k = -m,\dots,m,\; l=-n,\ldots,n\}.
  \end{align*}
The unit quaternions provide a group structure on $\S^3$, so that the mapping $a\mapsto S_a$ between $\faktor{\mathbb{S}^3}{\pm 1}$ and $\SO(3)$ as well as $(a,b)\mapsto L_a R_b$ between
$\faktor{(\mathbb{S}^3\times \mathbb{S}^3)}{\pm 1}$ and $\SO(4)$ become group isomorphisms. Their combination induces a group isomorphism between $\SO(3)\times\SO(3)$ and $\faktor{\SO(4)}{\pm 1}$. 
%The mapping $(S_a,S_b)\mapsto L_aR_b$ leads to a group isomorphism between $\SO(3)\times\SO(3)$ and $\faktor{\SO(4)}{\pm 1}$. 
Condition \eqref{eq:00} requires that the respective actions on $\S^2\times\S^2$ and $\G_{2,4}$ commute with $\mathcal{P}$. Hence, the induced pullback
\begin{equation*}
\mathcal{P}^*:L_2(\G_{2,4})\rightarrow  L_2\big(\faktor{\mathbb{S}^2\times \mathbb{S}^2}{\pm 1}\big),\qquad f\mapsto f(\mathcal{P}(\cdot,\cdot))
\end{equation*}
 is an intertwining isomorphism. In particular, $\mathcal{P}^*$ maps one irreducible subspace into the other.
% The irreducible decomposition of $L_2(\G_{2,4})$ with respect to $\faktor{\SO(4)}{\pm 1}$ is the same as for $\SO(4)$ since the action is by conjugation.
 Thus, the irreducible decomposition of $L_2(\G_{2,4})$ under the action of $\SO(4)$ is
  \begin{equation}\label{eq:decomp so4}
L_2(\G_{2,4}) =\bigoplus_{m+n\in2\mathbb{N}} \mathrm{span}\{ Y_{k,l}^{m,n}: k = -m,\dots,m,\; l=-n,\ldots,n\}.
  \end{equation}
In comparison to the irreducible decomposition of $L_2(\G_{2,4})$ for the action of $\SO(4)$ in \eqref{eq:decomp so4}, the irreducible components $H_\lambda(\G_{2,4})$ with respect to $\OOO(4)$ in \eqref{eq:decomp L2 single} are usually larger. 

Let us denote $ H^{m,n}:=\mathrm{span}\{ Y_{k,l}^{m,n}: k = -m,\dots,m,\; l=-n,\ldots,n\}$. Property \eqref{eq:22} yields that the
irreducible subspaces of $L_2(\G_{2,4})$ under the action of $\OOO(4)$ are
\begin{equation*}
V^{m,n}:=
\begin{cases}
H^{m,n} ,& m=n,\\
H^{m,n}\oplus H^{n,m},& m\neq n,
\end{cases}
\end{equation*}
where $m+n\in2\mathbb{N}$ and $m\geq n$, i.e.,
\begin{align*}
L_2(\G_{2,4}) = \bigoplus_{\substack{m+n\in2\mathbb{N}\\ m\geq n}} V^{m,n}=\bigoplus_{\lambda_1\geq\lambda_2\geq 0} V^{m_\lambda,n_\lambda},
\end{align*}
where $m_\lambda=\lambda_1+\lambda_2$ and $n_\lambda=\lambda_1-\lambda_2$. 

By considering $\mathcal{L}(P)=xy^\top$, $\mathcal{L}(P)\cdot\mathcal{L}(P)^\top=xx^\top$, and $\mathcal{L}(P)^\top\cdot\mathcal{L}(P)=yy^\top$, we observe that the homogeneous polynomials
$ %\begin{equation*}
  x_{i}y_{j}$ and $x_{i}x_{j}$, $y_{i}y_{j}$, for $i,j = 1,\dots,3, 
$ %\end{equation*}
can be written as homogeneous polynomials in the matrix entries of $P \in \G_{2,4}$ of degree $1$ and $2$, respectively. The monomial $x^\alpha y^{\beta}$ with $|\alpha|=m_\lambda$ and $|\beta|=n_\lambda$ is composed of $n_\lambda$ factors of the form $x_iy_j$ and $(m_\lambda-n_\lambda)/2$ factors of the form $x_{i}x_{j}$. Thus, $x^\alpha y^{\beta}$ is a homogeneous polynomial of degree $n_\lambda+2((m_\lambda-n_\lambda)/2) = m_\lambda$ in the matrix entries of $P$. We deduce
\begin{equation}\label{eq:eq}
\bigoplus_{\lambda_1+\lambda_2\leq t} V^{m_\lambda,n_\lambda} \subset \bigoplus_{\lambda_1+\lambda_2\leq t} H_\lambda(\G_{k,d})
\end{equation}
since the right-hand side of \eqref{eq:eq} are the polynomials of degree at most $t$ in the matrix entries of $P\in\G_{2,4}$, cf.~\cite{Bachoc:2004fk}. 
By applying \cite[Formulas~(24.29) and (24.41)]{Fulton:1991fk}, 
we see that the dimension of $H_{\lambda}(\mathcal{G}_{2,4})$ is
\begin{equation}\label{eq:dim H}
\mathrm{dim}(H_{\lambda}(\mathcal{G}_{2,4})) =(2-\delta_{m_\lambda,n_\lambda})(2m_\lambda+1)(2n_\lambda+1),
\end{equation}
which matches $\dim(V^{m_\lambda,n_\lambda})$. Hence, there holds equality in \eqref{eq:eq}. For fixed $t$, the dimensions in \eqref{eq:dim H} are pairwise different, for all $\lambda_1+\lambda_2=t$. An induction over $t$ leads to $H_{\lambda}(\mathcal{G}_{2,4})=V^{m_\lambda,n_\lambda}$.
\end{proof}

\subsection{Proofs for Section \ref{sec:FFT}}\label{app:sec:FFT}
The relation \eqref{eq:Y in e} yields
\begin{equation*}
F(x,y) = \sum_{k,l=-M}^M e^{{\rm i}k\varphi_1}e^{{\rm i}l\varphi_2} B_{k,l}(\theta_1,\theta_2),
\end{equation*} 
where $B_{k,l}(\theta_1,\theta_2)$ are given by
\begin{align*}
B_{k,l}(\theta_1,\theta_2) & := \sum_{m_1=|k|}^M\;\sum_{m_2=|l|}^M f^{m_1,m_2}_{k,l}
\sum_{k'=-m_1}^{m_1} \sum_{l'=-m_2}^{m_2} c^{m_1}_{k,k'} c^{m_2}_{l,l'} e^{{\rm i}k'\theta_1} e^{{\rm i}l'\theta_2}\\
& = \sum_{k',l'=-M}^M e^{{\rm i}k'\theta_1}e^{{\rm i}l'\theta_2} \sum_{m_1=\max(|k|,N-k')}^M \;\;\sum_{m_2=\max(|l|,M-l')}^M f^{m_1,m_2}_{k,l}  c^{m_1}_{k,k'} c^{m_2}_{l,l'}.
\end{align*}
Hence, the coefficients from \eqref{eq:bsss} satisfy
\begin{equation}\label{eq:two sums for g}
b^{k,l}_{k',l'} = \sum_{m_1=\max(|k|,M-k')}^M c^{m_1}_{k,k'} \sum_{m_2=\max(|l|,M-l')}^M f^{m_1,m_2}_{k,l}   c^{m_2}_{l,l'}.
\end{equation}
First evaluating the inner sum for $k,l,l'=-M,\ldots,M$ and $m_1=\max(|k|,M-k'),\ldots,M$, and afterwards the outer sum enables the computation of the coefficients $b^{k,l}_{k',l'}$, for \break $k,l,k',l'=-M,\ldots,M$, in $O(M^5)$ operations provided that the numbers  $c^{m}_{k,k'}$ in \eqref{eq:Y in e} are given. 

\begin{remark}\label{rem:app poly transform}
The complexity for evaluating the two sums in \eqref{eq:two sums for g} can be further reduced to $O(M^4\log^2(M))$ by using a fast polynomial transform, cf.~\cite{Keiner:2009cv,Kunis:2003pt}. 
In this way, the nonequispaced fast Fourier transform on $\S^2\times\S^2$ 
takes $O(M^4\log^2(M)+n|\log(\epsilon)|^4)$ elementary operations. For $n\sim M^4$, this implies the reduction of complexity from $O(M^8)$ to $O(M^4\log^2(M)+M^4|\log(\epsilon)|^4)$ operations.
\end{remark}

The adjoint nonequispaced fast Fourier transform is about computing 
\begin{equation}\label{eq:adjoit fft}
%\hat{b}^{m,n}_{k,l}:=
\sum_{i=1}^M b_i \overline{Y^m_k(x_i)Y^n_l(y_i)},\quad m,n=0,\ldots,N,\;\; k=-m,\ldots,m,\;\; l=-n,\ldots,n,
\end{equation} 
for given coefficients $(b_i)_{i=1}^M\subset \mathbb{C}$ and locations $(x_i,y_i)_{i=1}^M\subset \S^2\times\S^2$. Similar arguments as above and in Section~\ref{sec:FFT} provide a fast evaluation of \eqref{eq:adjoit fft} by using the adjoint nonequispaced fast Fourier transform \texttt{adjoint nfft}, cf.~\cite{Keiner:2009cv,Kunis:2003pt}.

%\newpage

%\bibliographystyle{amsplain}
%\bibliography{../biblio_ehler2}
\providecommand{\bysame}{\leavevmode\hbox to3em{\hrulefill}\thinspace}
\providecommand{\MR}{\relax\ifhmode\unskip\space\fi MR }
\providecommand{\MRhref}[2]{%
  \href{http://www.ams.org/mathscinet-getitem?mr=#1}{#2}
}
\providecommand{\href}[2]{#2}

\end{document}